\documentclass[12pt,oneside,reqno,bibliography=totoc]{scrartcl}

\usepackage{cihan}

\author{Cihan Bahran}
\affil{Department of Mathematics, Bo\u{g}azi\c{c}i University\\
Bebek, 34342 Istanbul, Turkey\\
cihanbahran@gmail.com}\date{}

\AtBeginDocument{%
  \addtolength\abovedisplayskip{-0.5\baselineskip}%
  \addtolength\belowdisplayskip{-0.5\baselineskip}%
}

\title{Regularity and stable ranges of $\FI$-modules}

\DeclareMathOperator{\conf}{PConf}
\newcommand{\FI}{\mathbf{FI}}

\DeclareMathOperator{\sr}{\mathbf{st-rank}}

\DeclareMathOperator{\FB}{\mathbf{FB}}
\DeclareMathOperator{\CB}{\mathsf{CB}}
\DeclareMathOperator{\catalan}{\mathsf{Catalan}}

\newcounter{dummy}
\makeatletter
\newcommand\sitem[1][]{\item[(#1)]\refstepcounter{dummy}\def\@currentlabel{#1}}
\makeatother

\newcommand{\cofi}[1]{\co_{#1}^{\FI}}
\newcommand{\cofip}[1]{\wt{\co}_{#1}^{\FI}}

\newcommand{\cofibol}[1]{\mathbf{H}_{#1}^{\FI}}

\DeclareMathOperator{\ab}{\mathcal{A}}

\DeclareMathOperator{\reg}{reg}
\DeclareMathOperator{\low}{low}

\DeclareMathOperator{\sh}{\FI_{\pmb{\sharp}}}

\DeclareMathOperator{\induce}{\Ind_{\FB}^{\FI}}
\DeclareMathOperator{\mon}{\underline{Mon}}
\DeclareMathOperator{\orb}{\underline{Orb}}

\DeclareMathOperator{\kk}{\mathbb{F}}

\DeclareMathOperator{\thick}{Thick}

\DeclareMathOperator{\pbc}{SPB}

\newcommand{\tgen}{t_{0}}
\newcommand{\trel}{t_{1}}
\newcommand{\shift}[2]{\mathbf{\Sigma}^{#2}   #1 }

\newcommand{\tbol}{\mathbf{t}}

\newcommand{\weak}{\delta}
\newcommand{\mult}{M}
\newcommand{\struct}{A}

\newcommand{\cyc}{\mathbf{X}}

\newcommand{\nvm}{\mathbf{m}}

\newcommand{\deriv}{\mathbf{\Delta}}

\newcommand{\bul}{\bullet}

\newcommand{\hell}{h}
\newcommand{\local}{\hell^{\text{max}}}

\DeclareMathOperator{\poly}{\mathbf{P}}

\newcommand{\locoh}[1]{\co_{\mathfrak{m}}^{#1}}

\DeclareMathOperator{\coco}{co}
\newcommand{\ters}[1]{\coco\mhyphen{#1}}

\DeclareMathOperator{\ccx}{Ch}

\newcommand{\setz}{\mathsf{Set}_{0}}

\newcommand{\manif}{\mathcal{M}}

\newcommand{\cell}[1]{\chain^{\text{\,cell}}_{#1}}

\AtBeginDocument{%
   \def\MR#1{}
}

\newcommand{\pres}{\text{presented in finite degrees}}
\newcommand{\cikcik}{\text{$\kk$-character polynomial}}
\newcommand{\specht}[2]{\operatorname{S}_{#1}(#2)}
\newcommand{\inspecht}[2]{\operatorname{M}_{#1}(#2)}

\newcommand{\yleq}{\preccurlyeq}

\newcommand{\confix}[2]{
  \ifstrempty{#1}{
    \conf_{\bullet}(#2)
  }
  {
    \conf_{#1}(#2)
  }
}

\newcommand{\inv}[3]{
  \ifstrempty{#2}{
    	\ifstrempty{#1}{
	\operatorname{inv}^{+}_{\bullet}(#3)
	}
    {
    \operatorname{inv}^{\,#1}_{\bullet}(#3)
    }
  }
  {
    \operatorname{inv}^{\,#1}_{#2}(#3)
  }
}
\DeclareMathOperator{\Sym}{Sym}
\newcommand{\coinv}[3]{
  \ifstrempty{#2}{
  	\ifstrempty{#1}{
    \operatorname{coinv}_{\bullet}(#3)
  }
  	{
		\operatorname{coinv}^{\,#1}_{\bullet}(#3)
	}
  }
  {
    \operatorname{coinv}^{\,#1}_{#2}(#3)
  }
}

\makeatletter
\def\blfootnote{\gdef\@thefnmark{}\@footnotetext}
\makeatother

\begin{document}
\ytableausetup{centertableaux}
\maketitle

\blfootnote{\textup{2010} \textit{Mathematics Subject Classification}.
Primary 18A25, 05E10; Secondary 55R80, 11F75.} 
\vspace{-0.8 in}
\begin{onecolabstract} 
We give refined bounds for the regularity of $\FI$-modules and the stable ranges of $\FI$-modules for various forms of their stabilization studied in the representation stability literature. We show that our bounds are sharp in several cases. We apply these to get explicit stable ranges for diagonal coinvariant algebras, and improve those for ordered configuration spaces of manifolds and congruence subgroups of general linear groups.
\end{onecolabstract}
\vspace{0.4 in}

{\tableofcontents}

\section{Introduction} \label{section:general-FI}
The theory of $\FI$-modules has established itself as one of the main tools in studying the stable behavior of representations of symmetric groups since the foundational work of Church--Ellenberg--Farb \cite{cef}. The contributions of this paper are threefold:
\begin{birki}
 \item We obtain two bounds for the (Castelnuovo--Mumford) \textbf{regularity} of an $\FI$-module $V$: 
\begin{itemize}
\item The first bound, \textbf{Theorem \ref{improved-regularity}}, is in terms of the generation and presentation degrees of $V$, constituting a small but final improvement on the bounds of Church--Ellenberg \cite[Theorem A]{ce-homology} and Ramos \cite[Theorem 3.19]{ramos-fig}.
\item The second bound, \textbf{Theorem \ref{regularity-from-FI-hyp}}, is in terms of the local and stable degrees of $V$ (in the sense of Church--Miller--Nagpal--Reinhold \cite{cmnr-range}) which is also often sharp.
\end{itemize}

\item We identify four types of stabilization existing in the literature for $\FI$-modules in Definition \ref{stable-ranges-defn} and we improve the corresponding \textbf{stable ranges}
\begin{itemize}
\item in terms of the local and stable degrees in \textbf{Theorem \ref{main-ranges}},
\item in terms of $\FI$-hyperhomology in \textbf{Theorem \ref{hyper-ranges}}.
\end{itemize}
 
\item We apply these results to get 
\begin{itemize}
 \item explicit stable ranges for the graded pieces of diagonal coinvariant algebras in \textbf{Theorem \ref{main-coinv}},
 \item improvements in the stable ranges for the cohomology of ordered configuration spaces of manifolds in \textbf{Theorem \ref{main-config}},
 \item improvements in the stable ranges for the homology of congruence subgroups of general linear groups in \textbf{Theorem \ref{main-cong}}. 
\end{itemize}
\end{birki}

\paragraph{Notation.} We write $\FI$ for the category of finite sets and injections, and $\FB$ for the category of finite sets and bijections. An $\FI$-module is a functor $V \colon \FI \rarr \lMod{\zz}$ and given a finite set $S$, we write $V_{S}$ for its evaluation and for $n \in \nn$ we set $V_{n} := V_{\{1,\dots,n\}}$. We write $\lMod{\FI}$ for the category of $\FI$-modules. We similarly define $\FI$-spaces, $\FI$-groups, $\FB$-modules etc. We say that an $\FI$-module $V$ is defined over a ring $R$ if it factors through the forgetful functor $\lMod{R} \rarr \lMod{\zz}$. We write $\ters{\FI}$ for the opposite category of $\FI$, so for instance the cohomology groups of a $\ters{\FI}$-space are $\FI$-modules. Given an $\FB$-module or an $\FI$-module $W$, we write 
\begin{align*}
 \deg(W) &:= \min\{d \geq -1 : W_{S} = 0 \text{ for } |S| > d\} 
 \\
 &\in \{-1,0,1,2,3,\dots\} \cup \{\infty\} \, .
\end{align*}
\paragraph{$\FI$-homology.}
Consider the functor $\pi^{*} \colon \lMod{\FB} \rarr \lMod{\FI}$ that extends an $\FB$-module to an $\FI$-module by making non-bijections of finite sets act as zero; $\pi^{*}$ has a left adjoint $\cofi{0} \colon \lMod{\FI} \rarr \lMod{\FB}$\footnote{Some authors prefer working with 
$\cofip{0} := \pi^{*} \circ \cofi{0} : \lMod{\FI} \rarr \lMod{\FI}$ and its derived functors to have $\FI$-homology groups be $\FI$-modules themselves. This does not make much of a difference because as $\pi^{*}$ is exact, we have $\cofip{i} :=\operatorname{L}_{i}(\pi^{*} \circ \cofi{0}) \cong \pi^{*} \circ \cofi{i}$, for instance by \cite[Theorem 1]{adams-rieffel}.} which, as explained in \cite{ce-homology}, has the description 
\begin{align*}
 \cofi{0}(V)_{S} = \coker\! \left(
 \bigoplus_{T \subsetneq S} V_{T} \rarr V_{S}
 \right)
\end{align*}
for every finite set $S$. For each $i \geq 0$, we write $\cofi{i} := \operatorname{L}_{i}\!\cofi{0}$ for its $i$-th left derived functor, and write 
\begin{align*}
 t_{i}(V) := \deg\left( \cofi{i}(V) \right) \, 
\end{align*}
for every $\FI$-module $V$.  We say that $V$ is \textbf{generated in degrees} $\leq g$ if $\tgen(V) \leq g$, this is equivalent to $V$ having no proper $\FI$-submodule $U \leq V$ with $U_{S} = V_{S}$ for $|S| \leq g$. We say that $V$ is \textbf{presented in finite degrees} if $\tgen(V)$ and $\trel(V)$ are both finite. 

\subsection{Bounds for the regularity of $\FI$-modules}
Given an $\FI$-module $V$, we write
\begin{align*}
  \reg(V) 
  &:= \max\{t_{i}(V) - i : i \geq 1\} \\
  &\in\{-2,-1,0,1,2,3,\dots\} \cup\{\infty\}
\end{align*}
and call this quantity the \textbf{regularity} of $V$.

The first result of this paper is a small and final improvement on the known bounds of the regularity of an $\FI$-module $V$ in terms of $\tgen(V)$ and $\trel(V)$. The tools to obtain it were already present in Church--Ellenberg \cite[Proposition 4.3]{ce-homology} and Gan \cite[Lemma 19]{gan-shift-seq}, but neither its explicit statement nor its consequences in applications have been noted before.
 
\begin{thmx} \label{improved-regularity}
For every pair of integers $a,b \geq 0$, we have
\begin{align*}
\max\!\left\{\!
\reg(V) \!: \!\!
\begin{array}{l}
 \text{$V$ is an $\FI$-module with}
 \\
 0\leq \tgen(V) \leq a \, , 0 \leq \trel(V) \leq b 
\end{array}
 \!\!\right\} \,=\, 
\begin{cases}
 a + b  - 1 \, & \text{if\, $a < b$,} \\
 2b - 2 & \text{if\, $a \geq b$.} 
\end{cases}
\end{align*}
\end{thmx}

\begin{rem}
Given an $\FI$-module with $V$ with $0 \leq \tgen(V) \leq a$ and $0 \leq \trel(V) \leq b$, the inequality  
$\reg(V) \leq \min\{a,b\} + b - 1
$ is proved by Ramos in \cite[Theorem 3.19]{ramos-fig}, building on \cite[Theorem A]{ce-homology}. Thus the improvement in Theorem \ref{improved-regularity} is only in the case $a \geq b$ by exactly one degree, and it cannot be improved further.
\end{rem}

Mirroring \cite[Corollary B]{ce-homology} and \cite[Corollary 4]{gan-li-linear}, from Theorem \ref{improved-regularity} we obtain an improved bound on the generation degrees of the homology of a chain complex of $\cofi{0}$-acyclic $\FI$-modules.

\begin{corx} \label{H0-acyclic-complex-range}
 Suppose $\chain_{\star}$ is a chain complex of $\FI$-modules such that for every $k \in \zz$, the $\FI$-module $\chain_{k}$ is $\co_{0}^{\FI}$-acyclic and is generated in degrees $\leq g_{k}$. Assuming $g_{k} \geq 0$, the $\FI$-module  $\co_{k}(\chain_{\star})$ is generated in degrees 
\begin{align*}
 \leq \begin{cases}
 g_{k-1} + g_{k} + 1 
 & \text{if $g_{k} > g_{k-1}$,}
 \\
 2g_{k} & \text{if $g_{k} \leq g_{k-1}$.}
\end{cases}
\end{align*}
\end{corx} 

\paragraph{The shift functor.} Given any $\FI$-object $V \colon \FI \rarr \mathsf{C}$ in a category $\mathsf{C}$, we write $\shift{V}{}$ for the composition 
\begin{align*}
 \FI \xrightarrow{- \sqcup \{*\}} \FI \xrightarrow{V} \mathsf{C} \, .
\end{align*}
Inductively, we also write $\shift{V}{0} := V$ and $\shift{V}{r} := \shift{\!\left(\shift{V}{r-1}\right)}{}$ for each $r \geq 1$. The shift functor has been shown to be a very useful tool in studying $\FI$-modules, starting with the paper of Church--Ellenberg--Farb--Nagpal \cite{cefn} and Nagpal's thesis \cite{nagpal-thesis}. 
\paragraph{Thick subcategories.} Given any class $\mathbf{X}$ of objects in an abelian category $\mathsf{A}$, we write $\thick\gen{\mathbf{X}}$ for the smallest full subcategory of $\mathsf{A}$ that 
\begin{itemize}
 \item contains $\mathbf{X}$ and the zero object, 
 \item is closed under taking kernels, cokernels, extensions, and direct summands.
\end{itemize}

\vspace{0.3cm}
The following hypothesis phrased in terms of the shift functor will see frequent use in this paper:
\vspace{-0.3cm}
\begin{hyp} \label{FI-hyp}
 In the triple $(V,c,g)$, we have an $\FI$-module $V$ and integers $c,g \geq -1$ such that
\begin{birki}
 \item $\shift{V}{c+1}$ is $\cofi{0}$-acyclic, and
 \vspace{0.04in}
 \item $\shift{V}{x} \in \thick\gen{\text{$\cofi{0}$-acyclic $\FI$-modules generated in degrees $\leq g$}}$ for some $x \geq 0$.
\end{birki}
\end{hyp}
In fact $(V,c,g)$ satisfying Hypothesis \ref{FI-hyp} is equivalent to $V$ having local degree $\leq c$ and stable degree $\leq g$ in the sense of \cite{cmnr-range}; see Theorem \ref{pass-to-cmnr}. These invariants and  their propagation through taking kernels, cokernels and extensions via \cite[Propositions 3.2, 3.3]{cmnr-range} permeate this work, whence the choice to collect them in a single hypothesis. We can now state our second bound for the regularity.

\begin{thmx} \label{regularity-from-FI-hyp}
  If the triple $(V,c,g)$ satisfies Hypothesis \ref{FI-hyp}, then
\begin{align*}
 \reg(V) \leq  
 \begin{cases}
 -2 & \text{if $c = -1$,} \\
 c & \text{if $g = -1$ and $c \geq 0$,} \\
 c+1 & \text{if $0 \leq g \leq \ceil{c/2}$ and $c \geq 0$,} \\
 g + \floor{c/2} + 1  & \text{if $g > \ceil{c/2}$ and $c \geq 0$.}
\end{cases}
\end{align*}
Moreover, for every $c,g \geq -1$ there exist $\FI$-modules $\mathbf{I}(g)$, $\mathbf{T}(c)$, $\mathbf{S}(c)$, $\mathbf{V}(g)$ defined over $\qq$ with finite-dimensional evaluations such that 
\begin{birki}
 \item The triple $(\mathbf{I}(g),\, -1,\, g)$ satisfies Hypothesis \ref{FI-hyp} and $\reg(\mathbf{I}(g)) = -2$.
 \vspace{0.1cm}
 \item If $c \geq 0$, the triple $(\mathbf{T}(c),\, c,\, -1)$ satisfies Hypothesis \ref{FI-hyp} and $\reg(\mathbf{T}(c)) = c$.
  \vspace{0.1cm}
 \item If $c \geq 0$, the triple $(\mathbf{S}(c),\, c,\, 0)$ satisfies Hypothesis \ref{FI-hyp} and $\reg(\mathbf{S}(c)) = c+1$.
  \vspace{0.1cm}
 \item If $g \geq 1$, the triple $(\mathbf{V}(g),\, 2g-2,\, g)$ satisfies Hypothesis \ref{FI-hyp} and $\reg(\mathbf{V}(g)) = 2g$.
\end{birki}
\end{thmx}
We prove the bound in Theorem \ref{regularity-from-FI-hyp} by using the relationship of regularity with the \textbf{local cohomology} of $V$, obtained by Li--Ramos \cite[Theorem F]{li-ramos} and Ramos \cite[Corollary 4.15]{ramos-coh}. Considering that this relationship has been upgraded from an inequality to an equality by Nagpal--Sam--Snowden in \cite[Theorem 1.1]{nss-regularity}, it is not surprising that our bound is sharp in many cases.

\subsection{Various stabilizations for $\FI$-modules}
\subsubsection{Notions specific to field coefficients} \label{field-specific}
When an $\FI$-module is defined over a field with finite-dimensional evaluations, there are two notions with which we can formulate a uniform behavior across all symmetric groups which are not available for an arbitrary commutative ring: character polynomials and Specht modules. We briefly recall these.

\paragraph{Character polynomials.}
For every $j \geq 1$, consider the class function 
 \begin{align*}
 \cyc_{j} \colon \bigsqcup_{n = 0}^{\infty} \sym{n} &\rarr \nn \\
 \sigma &\mapsto \left|\{\text{$j$-cycles in the cycle decomposition of $\sigma$}\}\right| \, .
\end{align*}
For any field $\kk$, we call a polynomial $\poly$ with variables in 
\{$\cyc_{j} : j \text{ does not divide } \ch(\kk)\}$
and coefficients in $\qq$ an \textbf{\cikcik}. For any permutation $\sigma$ in a fixed finite symmetric group, $\poly(\sigma)$ is computed via extending $\cyc_{j}(\sigma)$ as defined above, and the degree of $\poly$ is defined via extending $\deg(\cyc_{j}) := j$.


\paragraph{Partitions and Specht modules.} Given a commutative ring $R$ and a weakly decreasing sequence 
$\mu : \mu_{1} \geq \mu_{2} \geq \cdots \geq 0  \cdots$ of integers which is eventually zero, writing $m = |\mu| := \sum \mu_{i}$ so that $\mu$ is a \textbf{partition} of size $m$ (or shortly $\mu \vdash m$), the Specht module $\specht{R}{\mu}$ is a well-defined $R\sym{m}$-module \cite[Section 4]{james-sym-book}. When $R = \kk$ is a field, their isomorphism classes 
\begin{align*}
 \{\left[\specht{\kk}{\mu}\right] : \mu \vdash m\}
\end{align*}
generate the Grothendieck group $K_{0}(\kk\!\sym{m})$ \cite[Corollary 12.2]{james-sym-book}. We say a partition $\lambda$ is $\kk$\textbf{-regular} if either $\ch(\kk) = 0$ or 
\begin{align*}
 \big|\{j : \lambda_{j} = t\}\big| < \ch(\kk) \quad \text{for every $t \geq 1$.}
\end{align*}
Given any partition $\lambda : \lambda_{1} \geq \lambda_{2} \geq \cdots$ and any integer $n \geq|\lambda| + \lambda_{1}$, we write 
\begin{align*}
 \lambda[n] : n-|\lambda| \geq \lambda_{1} \geq \lambda_{2} \geq \cdots \,\,,
\end{align*}
so that $\lambda[n] \vdash n$.

\subsubsection{Stable ranges}
The following definition encodes four types of stabilization for an $\FI$-module with six parameters.

\begin{defn} \label{stable-ranges-defn}
 Let $V$ be an $\FI$-module defined over a commutative ring $R$. Let $\tgen,\trel,\local,\weak$  $\geq -1$, $\struct \geq \max\{0,2\weak-1\}$, and $\mult \geq 0$ be integers. We say that $V$ has \textbf{stable ranges} 
\begin{align*}
 \mathlarger{
 \yleq (\tgen,\,\, \trel,\,\, \struct,\,\, \local,\,\, \weak,\,\, \mult)
 }
\end{align*}
if the following holds:  
\begin{birki}
 \item (\textbf{Inductive description}) Given $n,N \in \nn$, the natural map
\begin{align*}
 \mathlarger{\underset{\substack{S \subseteq \{1,\dots,n\}\vspace{0.011in} \\ |S| \leq N}}{\colim} \!\!V_{S} \rarr V_{n}}
\end{align*}
of $R \sym{n}$-modules is surjective if $N \geq t_{0}$, and injective if $N \geq t_{1}$.
\vspace{0.05in}

\item (\textbf{Additive structure}) There exist $R$-modules $\ab_{0}, \dots, \ab_{\weak}$ such that in the range $n \geq \struct$, there is an isomorphism
\vspace{-0.05in}
\begin{align*}
 \mathlarger{V_{n} \cong \bigoplus_{r=0}^{\weak} \ab_{r}^{\oplus \binom{n}{r} - \binom{n}{r - 1}}} \text{ \footnotemark}
\end{align*}
\footnotetext{We have required $A \geq 2\weak - 1$ to guarantee $\binom{n}{r} - \binom{n}{r-1} \geq 0$ whenever $n \geq A$ and $r \leq \weak$.}of $R$-modules, with the convention that $\binom{a}{b} = 0$ unless $0 \leq b \leq a$.
\vspace{0.05in}
 \item (\textbf{Polynomiality}) If $R = \kk$ is a field and $\dim_{\kk} V_{n} < \infty$ for every $n \geq 0$, then for each $r=0,\dots,\weak$ there is a finite-dimensional $\kk\!\sym{r}$-module $W_{r}$ such that in the range $n \geq \local+1$, the sequence of symmetric group (Brauer) characters
 \vspace{-0.05in}
\begin{align*}
\mathlarger{n \mapsto \mathlarger{\chi}_{V_{n}}}
\end{align*}
is equal to the $\kk$-character polynomial 
\begin{align*}
\mathlarger{
 \sum_{r=0}^{\weak} \sum_{\lambda \,\vdash\, r}
 \chi_{W_{r}}(\lambda)
 \binom{\cyc_{1} - \local - 1}{a_{1}(\lambda)} 
 \prod_{j=2}^{r}\!\! \binom{\cyc_{j}}{a_{j}(\lambda)} \, .
 }
\end{align*}
Here $a_{j}(\lambda)$ denotes the number of parts of size $j$ in $\lambda$, and if $a_{j}(\lambda) > 0$ for some $j$ divisible by $\ch(\kk)$, we write $\chi_{W_{r}}(\lambda) = 0$. 
In particular, there are integers $d_{0}, \dots, d_{\weak} \geq 0$ (namely $d_{r} = \dim_{\kk} W_{r}$) such that in the range $n \geq \local + 1$, we have
\vspace{-0.03in}
\begin{align*}
\mathlarger{
\dim_{\kk} V_{n} = \sum_{r=0}^{\weak} d_{r} \binom{n - \local - 1}{r} \,.
}
\end{align*}
\vspace{0.03 in}
\item (\textbf{Virtual Specht stability}) If $R = \kk$ is a field and $\dim_{\kk} V_{n} < \infty$ for every $n \geq 0$, then there is a uniquely determined function
\begin{align*}
  \nvm \colon 
\{\text{$\kk$-regular partitions of size $\leq \weak$}\} \rarr \zz \, ,
\end{align*}
such that $M \geq \max\{|\lambda| + \lambda_{1} : \nvm(\lambda) \neq 0\}$,\footnote{We require this for $M$ so that the right hand side of the below equation is well-defined. It is guaranteed by $M \geq 2\weak$, which is what we often have in applications.} and for every $n \geq M$
we have 
\begin{align*}
 \mathlarger{[V_{n}] = \sum_{\lambda} \nvm(\lambda) \left[ \specht{\kk}{\lambda[n]} \right]}
\end{align*}
\vspace{-0.05in}
in the Grothendieck group of finite-dimensional $\kk\!\sym{n}$-modules. 
%
\end{birki}
\end{defn}

Definition \ref{stable-ranges-defn} is in fact a collection of general results about $\FI$-modules turned into a definition: 
\begin{itemize}
 \item The inductive description (1) with the colimit was first formulated by Church--Ellenberg--Farb--Nagpal \cite[Theorem C]{cefn}.
 \item The form of the additive structure in (2) is due to Patzt--Wiltshire-Gordon \cite[Theorem A]{patzt-wgordon}.
 \item For finitely generated $\FI$-modules over a field $\kk$, the eventual polynomiality in (3) was obtained first by Church--Ellenberg--Farb \cite[Theorem 3.3.4]{cef} when $\ch \kk = 0$ for ordinary characters, and by Harman \cite[Corollary 3.1]{harman-poly} when $\ch \kk > 0$ for Brauer characters. However, the form of the polynomials presented here is more explicit.
 \item  In characteristic 0, Specht stability in (4) as a concept actually predates $\FI$-modules: it is what Church--Farb dubbed \emph{uniform multiplicity stability} in \cite[Definition 2.7]{church-farb-rep-stab}. That finitely generated $\FI$-modules over a field of characteristic zero exhibit uniform multiplicity stability is due to Church--Ellenberg--Farb \cite[Theorem 1.13]{cef}. Over a field of positive characteristic, virtual Specht stability is due to Harman \cite[Theorem 1.3]{harman-poly}. 
\end{itemize}
All of the remaining main results of this paper, namely Theorems \ref{main-ranges}, \ref{hyper-ranges}, \ref{main-coinv}, \ref{main-config}, \ref{main-cong}, are formulated using Definition \ref{stable-ranges-defn}.

\begin{rem}[Virtual vs genuine Specht stability]
Suppose $V$ is an $\FI$-module defined over a field $\kk$ of characteristic $p > 0$ such that $\dim_{\kk} V_{n} < \infty$ for every $n$, and has stable ranges
\begin{align*}
 \yleq (\tgen,\,\, \trel,\,\, \struct,\,\, \local,\,\, \weak,\,\, \mult) \, .
\end{align*}
Note that if $p > M$, then $\kk\!\sym{\!M}$-modules are semisimple and hence the constants $\nvm(\lambda)$ in Definition \ref{stable-ranges-defn} have to be non-negative. In fact, it follows from a result of Putman \cite[Theorem E]{putman-congruence} that if furthermore $p > K:= 2\max\{\tgen,\trel\}+1$, then $V$ has stable ranges
\begin{align*}
 \yleq (\tgen,\,\, \trel,\,\, \struct,\,\, \local,\,\, \weak,\,\, K) 
\end{align*}
such that the non-negative constants $\nvm(\lambda)$ arise out of a Specht filtration of $V_{n}$ for every $n \geq K$. Together with certain compatibility criteria as $n$ varies, the existence of such filtrations is called \textbf{Specht stability} in \cite[Section 6.3]{putman-congruence} and applies even when $V_{n}$'s are infinite-dimensional. This notion does not appear anywhere else in this paper, though its traces can be seen in Remarks \ref{compare-main}, \ref{compare-hyper}, \ref{compare-coinv}, \ref{compare-config}, \ref{compare-cong}.
\end{rem}

Our next result converts Hypothesis \ref{FI-hyp} into stable ranges.
\begin{thmx} \label{main-ranges}
 If the triple $(V,c,g)$ satisfies Hypothesis \ref{FI-hyp}, then $V$ 
has stable ranges 
\begin{align*}
\begin{cases}
  \yleq(g,\,\, -1,\,\, \max\{0, 2g-1\},\,\, -1,\,\, g,\,\, \max\{0,2g\}) & \text{if $c=-1$,} 
  \\
  \yleq (c,\,\, c+1,\,\, c+1 ,\,\, c,\,\, -1,\,\, c+1) &\text{if $g = -1$ and $c \geq 0$,} \\
  \yleq (c+1,\,\, c+2,\,\, c+1,\,\, c,\,\, g,\,\, c+1) & \text{if $0\leq g \leq \ceil{c/2}$ and $c \geq 0$,} 
  \\
  \yleq(g + \floor{c/2} + 1,\,\, g + \floor{c/2} + 2,\,\,
  2g-1,\,\, c,\,\, g,\,\, 2g) & \text{if $g > \ceil{c/2}$ and $c \geq 0$.}
\end{cases}
\end{align*}
Moreover, for every $c,g \geq -1$ there exist $\FI$-modules $\mathbf{I}(g)$, $\mathbf{T}(c)$, $\mathbf{S}(c)$, $\mathbf{V}(g)$ defined over $\qq$ with finite-dimensional evaluations such that 
\begin{birki}
 \item The triple $\big(\mathbf{I}(g),\,-1,\,g\big)$ satisfies Hypothesis \ref{FI-hyp} and all of the stable ranges
\begin{align*}
 \yleq(g,\,\, -1,\,\, \max\{0, 2g-1\},\,\, -1,\,\, g,\,\, \max\{0,2g\})
\end{align*}
of $\mathbf{I}(g)$ are sharp.
 \vspace{0.1cm}
 \item If $c \geq 0$, the triple $\big(\mathbf{T}(c),\, c,\, -1\big)$ satisfies Hypothesis \ref{FI-hyp} and all of the stable ranges 
\begin{align*}
 \yleq (c,\,\, c+1,\,\, c+1 ,\,\, c,\,\, -1,\,\, c+1) 
\end{align*}
of $\mathbf{T}(c)$ are sharp.
  \vspace{0.1cm}
 \item If $0 \leq g \leq\ceil*{c/2}$ and $c \geq 0$, the triple $\big(\mathbf{S}(c) \oplus \mathbf{I}(g),\, c,\, g\big)$ satisfies Hypothesis \ref{FI-hyp} and all of the stable ranges 
\begin{align*}
 \yleq (c+1,\,\, c+2,\,\, c+1 ,\,\, c,\,\, g,\,\, c+1) 
\end{align*}
of $\mathbf{S}(c) \oplus \mathbf{I}(g)$ are sharp.  
\vspace{0.1cm}
 \item If $g \geq 1$, the triple $\big(\mathbf{V}(g),\,2g-2,\,g\big)$ satisfies Hypothesis \ref{FI-hyp} and all of the stable ranges
\begin{align*}
 \yleq(2g,\,\, 2g+1,\,\,
  2g-1,\,\, 2g-2,\,\, g,\,\, 2g)
\end{align*}
of $\mathbf{V}(g)$ are sharp.
\end{birki}
\end{thmx}
The ingredients of Theorem \ref{main-ranges} are
\begin{itemize}
 \item Theorem \ref{regularity-from-FI-hyp} for the inductive description (more specifically Corollary \ref{t0-t1-ranges}),
 \item an improvement to the stable range of the additive structure in Theorem \ref{thm-additive}, and 
 \item  an explicit form of polynomiality in arbitrary characteristic in Theorem \ref{thm-poly-specht}.
\end{itemize}
The $\FI$-modules $\mathbf{I}(g)$, $\mathbf{T}(c)$, $\mathbf{S}(c)$, $\mathbf{V}(g)$ witnessing the sharpness in Theorem \ref{main-ranges} are the same ones with those in Theorem \ref{regularity-from-FI-hyp}.

\begin{rem} \label{compare-main}
Combining \cite[Proposition 3.1]{cmnr-range}, \cite[Theorem A]{patzt-wgordon}, \cite[Theorem E]{putman-congruence}, \cite[Theorem 1.2]{harman-poly} the best stable ranges established previously in the literature under the assumptions of Theorem \ref{main-ranges} were
\begin{align*}
  \yleq (g+c+1,\, g+2c+2,\, 2g+4c+3,\, c,\, g,\, M) \, ,
\end{align*}
where $M$ was undetermined in general and if $V$ is defined over a field of characteristic $\geq 2g + 4c + 5$ or zero, one could take $M = 2g + 4c + 4$.
\end{rem}

\begin{rem}
 The sharpness of the stable ranges of $\mathbf{V}(g)$ in part (4) of Theorem \ref{main-ranges} is significant for comparison to our later applications. Both in Theorem \ref{main-coinv} for coinvariant algebras and in Theorem \ref{main-config} for configuration spaces, the relevant $\FI$-module $V$ has stable ranges of the form 
\begin{align*}
 \yleq(2g,\,\, 2g+1,\,\,
  2g-1,\,\, 2g-2,\,\, g,\,\, 2g) \,
\end{align*}
for some $g \geq 1$ as a result of the triple $(V,\,2g-2,\,g)$ satisfying Hypothesis \ref{FI-hyp}. So for an improvement to these ranges, one either needs to show that $(V,c,h)$ satisfies Hypothesis \ref{FI-hyp} for some $c < 2g - 2$ or $h<g$, or use extra knowledge about $V$.
\end{rem}

\paragraph{Stable ranges in terms of $\FI$-hyperhomology.} When $\FI$-modules arise as the homology of an $\FI$-chain complex, it has been exhibited by Church--Miller--Nagpal--Reinhold \cite{cmnr-range}, Gan--Li \cite{gan-li-linear}, and Miller--Wilson \cite{miller-wilson-PConf-hyper} that their stable ranges can be improved by working with numerical invariants attached directly to $\FI$-chain complexes. This is essentially because these ``derived stable ranges'' propagate in a more lossless way through filtrations than the stable ranges of $\FI$-modules do through spectral sequences. The recipe is that like any right exact functor between abelian categories with enough projectives, $\cofi{0}$ has left hyper-derived functors \cite[Definition 5.7.4]{weibel-hom-alg}
\begin{align*}
 \cofibol{k} := \mathbb{L}_{k}\!\cofi{0} \colon \ccx_{\geq 0}(\lMod{\FI}) \rarr \lMod{\FB}\, ,
\end{align*}
which we refer to as \textbf{$\FI$-hyperhomology}.
Given an $\FI$-chain complex $\chain_{\star}$ supported on non-negative degrees, we write 
\begin{align*}
 \tbol_{k}(\chain_{\star}) := \deg(\cofibol{k}(\chain_{\star})) \, .
\end{align*}
The next result converts bounds for the $\FI$-hyperhomology of a chain complex to stable ranges for its homology. We obtain it by running the argument of Gan--Li \cite{gan-li-linear} with Theorem \ref{improved-regularity} and Theorem \ref{main-ranges}. 

\begin{thmx} \label{hyper-ranges}
 Let $\chain_{\star}$ be a chain complex of $\FI$-modules supported on non-negative degrees and $k \geq 0$ such that $\tbol_{k}(\chain_{\star}) \leq \theta_{k}$, $\tbol_{k+1}(\chain_{\star}) \leq \theta_{k+1}$ with $\theta_{k} \geq 0$. Then $\co_{k}(\chain_{\star})$ has stable ranges 
\begin{align*} 
\begin{cases}
 \yleq (0,\, -1,\, 0,\, -1,\, 0,\, 0) & \text{if \,$\theta_{k} = \theta_{k+1} = 0$,}
 \vspace{0.1in}
 \\
  \yleq \left(
\begin{array}{r r r}
  2\theta_{k}, &2\theta_{k}+1, &2\theta_{k} - 1,\\
  2\theta_{k} - 2, &\theta_{k}, & 2\theta_{k}\,
\end{array}
 \right)
 & \text{if $\theta_{k} \geq \max\{1,\theta_{k+1}\}$,} 
 \vspace{0.1in} 
 \\
 \yleq \left(
\begin{array}{r r r}
  2\theta_{k}, &2\theta_{k+1}-1, &2\theta_{k+1} - 1,\\
  2\theta_{k+1} - 2, &\theta_{k}, & 2\theta_{k+1} - 1\,
\end{array}
 \right)
 & \text{if $\theta_{k} < \theta_{k+1}$.}
\end{cases}
\end{align*}
If furthermore $(\theta_{m} : m \geq 0)$ is a strictly increasing sequence such that $\theta_{0} \geq 0$ and $\tbol_{m}(\chain_{\star}) \leq \theta_{m}$ for every $m \geq 0$, then $\co_{k}(\chain_{\star})$ has stable ranges 
\begin{align*}
\begin{cases}
 \yleq \left(
\begin{array}{r r r}
  \theta_{1}-1, &\theta_{1}, &2\theta_{1} - 1,\\
  2\theta_{1} - 2, &\theta_{0}, & 2\theta_{1} - 1\,
\end{array}
 \right) & \text{if $k=0$ and $2\theta_{0} + 1 \geq \theta_{1}$,}
 \vspace{0.1in}
 \\
 \yleq \left(
\begin{array}{r r r}
  2\theta_{k}, &2\theta_{k}+1, &2\theta_{k+1} - 1,\\
  2\theta_{k+1} - 2, &\theta_{k}, & 2\theta_{k+1} - 1\,
\end{array}
 \right)
 & \text{if $k \geq 1$ and $2\theta_{k}+1 \geq \theta_{k+1}$,}
\vspace{0.1in}
\\
 \yleq \left(
\begin{array}{r r r}
  2\theta_{k}, &\theta_{k+1}, &2\theta_{k}+\theta_{k+1},\\
  2\theta_{k}+\theta_{k+1}-1, &\theta_{k}, & 2\theta_{k}+\theta_{k+1}\,
\end{array}
 \right) & \text{if $2\theta_{k}+1 < \theta_{k+1}$,}
\end{cases}
\end{align*}
\end{thmx}

\begin{rem} \label{compare-hyper}
Combining \cite[Theorem 5]{gan-li-linear}, \cite[Theorem A]{patzt-wgordon}, \cite[Theorem 1.3]{li-FIG-invariants}, \cite[Remark 9]{gan-li-linear}, \cite[Theorem E]{putman-congruence}, \cite[Theorem 1.2]{harman-poly}, the best stable ranges established previously in the literature under the assumptions of Theorem \ref{hyper-ranges} were
\begin{align*}
 \begin{cases}
  \yleq \left(
  \begin{array}{r r r}
   2\theta_{k} + 1, &2\theta_{k}+2, &4\theta_{k}+3,\\
   4\theta_{k} + 2, &\theta_{k}, &M_{k}\,
  \end{array} 
 \right) & \text{if $\theta_{k} \geq \theta_{k+1}$,} 
 \vspace{0.1in} 
 \\
 \yleq \left(
\begin{array}{r r r}
  2\theta_{k} + 1, &2\theta_{k+1} + 2, &4\theta_{k+1} + 3,\\
  2\theta_{k} + 2\theta_{k+1} + 2, &\theta_{k}, & M_{k}\,
\end{array}
 \right)  & \text{if $\theta_{k} < \theta_{k+1}$,}
\end{cases}
\end{align*}
where $M_{k}$ was undetermined in general and if $\chain_{\star}$ is defined over a field of characteristic $\geq 4\max\{\theta_{k},\theta_{k+1}\} + 6$ or zero, then one could take $M_{k} = 4\max\{\theta_{k},\theta_{k+1}\} + 5$.
\end{rem}

\subsection{Applications}

\subsubsection{Diagonal coinvariant algebras}
Given a commutative ring $R$ and an $R$-module $L$, we can form the symmetric algebra $\Sym(L)$. In the case $L = E ^{\oplus S}$ where $E$ is a free $R$-module with a finite basis $\B$ and $S$ is a finite set, the algebra $\Sym(E^{\oplus S})$ can be identified with the polynomial algebra $R[\B \times S]$ and equipped with an $\nn^{\B}$-grading via declaring 
\begin{align*}
\deg \bigg(
 \prod_{(x,s) \in \B \times S} (x,s)^{\alpha(x,s)}
 \bigg) := \J \in \nn^{\B} 
\end{align*}
for each monomial if and only if \ds{\sum_{s \in S} \alpha(x,s) = \J(x)} for every $x \in \B$. We call such $\J$ a \textbf{multi-degree}. For example when $\B = \{x,y\}$, the set of monomials with multi-degree \ds{(2,1) := 
\left(
\begin{array}{l}
 x \mapsto 2 \\
 y \mapsto 1
\end{array}
\right)
} is
$ \{x_{i}^{2}y_{j} :  i,j \in S \}
 \cup
 \{x_{i}x_{j}y_{k} : i,j,k \in S ,\, i \neq j \}
$, where we have written $x_{s} := (x,s)$ and $y_{s} := (y,s)$ shortly. 

The symmetric group $\sym{S}$ permutes the set of variables $\B \times S$ via the second coordinate, and the induced $\sym{S}$-action on $R[\B \times S]$ preserves the $\nn^{\B}$-grading. Given a multi-degree $\J \colon \B \rarr \nn$, we write $\Sym^{\J}(E^{\oplus S})$ for the $\J$-graded piece and write 
\begin{align*}
 \inv{\J}{S}{E} := \left(\Sym^{\J}(E^{\oplus S})\right)^{\sym{S}}
\end{align*}
for its $\sym{S}$-invariants. The \textbf{total degree} of $\J$ is $|\J| := \sum_{x \in \B} \J(x) \in \nn$. In the above example, $\inv{(2,1)}{S}{E}$ is  freely generated by the orbit sums
\begin{align*}
 \sum_{i \in S} x_{i}^{2}y_{i},\, \sum_{\substack{i,j \in S \\ i \neq j}}x_{i}^{2}y_{j},\, 
 \sum_{\substack{i,j \in S \\ i \neq j}}x_{i}x_{j}y_{i},\, 
  \sum_{\substack{i,j,k \in S \text{ distinct}}}x_{i}x_{j}y_{k}
\end{align*}
as an $R$-module. The quotient of the $R$-algebra 
\ds{
 \Sym(E^{\oplus S}) = \bigoplus_{\J \in \nn^{\B}} \Sym^{\J}(E^{\oplus S})
}
by the ideal generated by the $R$-submodule 
\begin{align*}
  \bigoplus_{0 \neq \J \in \nn^{\B}} \inv{\J}{S}{E}
\end{align*}
of positive degree $\sym{S}$-invariants (often called the \emph{Hilbert ideal} in the invariant theory literature) is the \textbf{diagonal coinvariant algebra}, which we denote by $\coinv{}{\!S}{E}$. 
\begin{rem}[The total dimension]
  When $R=\kk$ is a field, it follows from classical invariant theory  \cite[Corollary 2.1.6]{neusel-smith} that $\coinv{}{\!n}{E}$ is a finite-dimensional $\kk$-algebra. The exact dimension as $n$ varies seems to be known (even conjecturally) in only the following cases: 
\begin{birki}
 \item We have $\dim_{\cc}\!\left(\coinv{}{\!n}{\cc}\right) = n!\,$. Much more is known in this univariate case, which we revisit in Example \ref{coinv-chev}.
 \item The identity $\dim_{\cc}\!\left(\coinv{}{\!n}{\cc^{2}}\right) = (n+1)^{n-1}$ stood as a conjecture for about 15 years, and was finally established through algebraic geometry by Haiman \cite{haiman-vanishing}.
 \item It is conjectured \cite[(2)]{bergeron-ratelle}, \cite[Fact 2.8.1]{haiman-conjectures}, \cite[(2.13)]{bergeron-GL} that  
\begin{align*}
  \dim_{\cc}\!\left(\coinv{}{\!n}{\cc^{3}}\right) = 2^{n}(n+1)^{n-2} \, .
\end{align*}
\end{birki}
\end{rem}

The exponential (instead of polynomial) growth prohibits $\FI$-modules to be directly useful with the whole algebra $\coinv{}{n}{E}$, so we consider its graded pieces. The algebra $\coinv{}{\!S}{E}$ inherits the $\nn^{\B}$-grading from $\Sym(E^{\oplus S})$ so that we can write 
\begin{align*}
 \coinv{}{\!S}{E} = \bigoplus_{\J \in \nn^{\B}} \coinv{\J}{S}{E} \, .
\end{align*}
As explained in \cite[Section 5]{cef} and Section \ref{section-coinv} here, the assignment $S \mapsto \coinv{}{S}{E}$ defines an $\nn^{\B}$-graded $\ters{\FI}$-algebra $\coinv{}{}{E}$. In particular for every $\J \in \nn^{\B}$, we have an $\FI$-module 
\begin{align*}
 \coinv{\J}{}{E}^{\vee} := \Hom_{R}(\coinv{\J}{}{E},R)
\end{align*}
defined over $R$.

Our approach to diagonal coinvariant algebras is arguably more ``hands on'' than the other two applications and does not use Theorem \ref{hyper-ranges}. The $\ters{\FI}$-module $\coinv{\J}{}{E}$ is, essentially by definition, the cokernel of a map between $\ters{\FI}$-modules expressed in terms of $\inv{\J}{}{E}$ and $\Sym^{\J}(E^{\oplus \bul})$ (see (\ref{coinv-seq}) in Section \ref{section-coinv}). We work out the structure of the latter two to get stable ranges for $\coinv{\J}{}{E}^{\vee}$.

\begin{thmx} \label{main-coinv}
 Let $R$ be a commutative ring, $E$ a free $R$-module with a finite basis $\B$, and  $\J \colon \B \rarr \nn$ be a multi-degree whose total degree is $|\J| \geq 1$. Then $\coinv{\J}{}{E}^{\vee}$ has stable ranges
\begin{align*}
\yleq \left(
\begin{array}{r r r}
  2|\J|, &2|\J| + 1, &2|\J|-1,
  \vspace{0.04in}
  \\
  2|\J|-2, &|\J|, &2|\J|\,
\end{array}
 \right) \, .
\end{align*}
\end{thmx}

\begin{rem} \label{compare-coinv}
Combining \cite[Theorem 1.13]{cefn}, \cite[Theorem A]{patzt-wgordon}, \cite[Proposition 3.1, part (2)]{cmnr-range}, \cite[Theorem 1.11]{cef}, \cite[Theorem E]{putman-congruence}, \cite[Theorem 1.2]{harman-poly}, the best stable ranges established previously in the literature under the assumptions of Theorem \ref{main-coinv} were
\begin{align*}
 \yleq (N,\,\, N,\,\, 
2N-1,\,\, 
 2N-1,\,\, 
 |\J|,\,\, M)
\end{align*}
where $N, M$ were undetermined in general, and if $R$ is a field of characteristic $\geq 2N+2$ or zero, one could take $M = 2N+1$ with $N$ still undetermined. Moreover one needed to assume $R$ is Noetherian in \cite{cefn}, which is not assumed in Theorem \ref{main-coinv}.
\end{rem}

\begin{ex} \label{coinv-chev}
 When $\kk = \cc$ and $E = \cc$ (so $|\B| = 1$), if we forget the grading, $\coinv{}{\!n}{\cc}$ affords the regular $\sym{n}$-representation \cite[Theorem (B)]{chevalley-coinv}; also see \cite[Theorem 7.2.1]{neusel-smith}. The graded pieces are also completely understood as $\sym{n}$-representations: for each $j \in \nn$ we have 
\begin{align*}
\mathlarger{
 \coinv{j}{n}{\cc} \cong \bigoplus_{\mu \,\vdash\, n} \specht{\cc}{\mu}^{\oplus u^{j}(\mu)} \, ,
 }
\end{align*}
where $u^{j}(\mu)$ is the number of standard tableaux of shape $\mu$ and major index $j$ \cite[Theorem 8.8]{reut-book}. Let us compare this with what Theorem \ref{main-coinv} says about the $\FI$-module $\coinv{j}{}{\kk}^{\vee}$: it has stable ranges 
\begin{align*}
 \yleq (2j,\,\, 2j+1,\,\, 2j-1,\,\, 2j-2,\,\, j,\,\, 2j) \, .
\end{align*}
By the Specht stability part, we have a function 
$\nvm^{j} \colon \{\text{partitions of size $\leq j$}\} \rarr \zz$ such that for every $n \geq 2j$ we have 
\begin{align*}
 [\coinv{j}{n}{\cc}^{\vee}] = \sum_{\lambda} \nvm^{j}(\lambda)[\specht{\cc}{\lambda[n]}]
\end{align*}
in the Grothendieck group of finite-dimensional $\cc\sym{n}$-modules. Since complex $\sym{n}$-representations are self-dual and semisimple with irreducible Specht modules, we conclude that 
$\nvm^{j}(\lambda) = u^{j}(\lambda[n])$ for every $n \geq 2j$.\footnote{This gives a somewhat perverse but valid way of showing that the sequence $n \mapsto u^{j}(\lambda[n])$ is eventually constant after fixing $j$ and $\lambda$. Of course this can be done by elementary means directly from the tableau definition, and was done in \cite[proof of Theorem 7.1, page 302]{church-farb-rep-stab} to establish  representation stability in the graded pieces of the univariate coinvariant algebra before the $\FI$-module technology.} Now we claim that for the single block partition $(j)$, we have $\nvm^{j}(j) = u^{j}(j[2j]) = u^{j}(j,j) > 0$: indeed we may assume $j \geq 1$ and check that the standard tableau
\begin{align*}
\begin{ytableau}
   1 & 2 & \cdots & j \\
   \mathsmaller{j+1} & \mathsmaller{j+2}  & \cdots & 2j 
\end{ytableau}
\end{align*}
has shape $(j,j)$ and descent set $\{j\}$, hence major index $j$. Thus 
\begin{align*}
 \max\{|\lambda| + \lambda_{1} : \nvm^{j}(\lambda) \neq 0\} \,=\, 2j
\end{align*}
and the last two coordinates $(j, 2j)$ of the stable ranges are sharp here.
On the other hand, the purported range $n \geq 2j-1$ for the polynomial regime is not sharp. To see this, note that $\cc^{\oplus n}$ is a reflection representation of $\sym{n}$, so by \cite[Corollary 3.31]{lehrer-taylor-book} we have
\begin{align*}
 \Sym(\cc^{\oplus n}) \cong \coinv{}{\!n}{\cc} \otimes_{\cc} \inv{}{\!n}{\cc} 
\end{align*}
as $\nn$-graded $\cc\sym{n}$-modules. Writing 
\begin{itemize}
 \item $\mathbf{\Omega}$ for the ring of $\cc$-character polynomials,
 \item $\mathbf{S}^{(j)} \in \mathbf{\Omega}$ that computes the characters of the sequence $n \mapsto \Sym^{j}(\cc^{\oplus n})$,
 \item $\mathbf{C}^{(j)} \in \mathbf{\Omega}$ that (eventually) computes the characters of the sequence $n \mapsto \coinv{j}{n}{\cc}$,
 \item $p(j)$ for the number of partitions of size $j$,
 \item $p_{\leq n}(j)$ for the number of partitions of size $j$ into $\leq n$ parts,
\end{itemize}
we have
$\dim_{\cc}(\inv{j}{n}{\cc}) = p_{\leq n}(j)
$ for every $j,n$, which in turn equals $p(j)$ when $n \geq j$, sharply. Because 
\begin{align*}
 \sum_{j=0}^{\infty}\mathbf{S}^{(j)}t^{j} &= \prod_{i=1}^{\infty}(1-t^{i})^{-\cyc_{i}} \in \mathbf{\Omega}[[t]] \quad \text{by \cite[Theorem 2.6]{narayanan-char-poly}, and}
 \\
 \sum_{j=0}^{\infty}p(j)t^{j} &= \prod_{i=1}^{\infty} (1-t^{i})^{-1} \in \nn[[t]] \quad \text{by \cite[(3.50)]{wilf-gfology},}
\end{align*}
we conclude from the above $\nn$-graded isomorphism that
\begin{align*}
 \sum_{j=0}^{\infty}\mathbf{C}^{(j)}t^{j} = \prod_{i=1}^{\infty}(1-t^{i})^{1-\cyc_{i}} 
 &= \prod_{i=1}^{\infty} \sum_{a=0}^{\infty}\binom{1-\cyc_{i}}{a}(-t^{i})^{a}
 \\
 &= \prod_{i=1}^{\infty} \sum_{a=0}^{\infty}\binom{\cyc_{i}+a-2}{a}t^{ai}
 \in \mathbf{\Omega}[[t]] \, .
\end{align*}
For example
\begin{align*}
  \mathbf{C}^{(4)} &= (\cyc_{4}-1) + (\cyc_{3}-1)(\cyc_{1}-1) + \binom{\cyc_{2}}{2} + (\cyc_{2}-1)\binom{\cyc_{1}}{2}
  + \binom{\cyc_{1}+2}{4}
  \\
  &=
 \cyc_{4} + \cyc_{3}\cyc_{1} + \binom{\cyc_{2}}{2}
 + \cyc_{2}\binom{\cyc_{1}}{2} + \binom{\cyc_{1}}{4}
 - \cyc_{3} + 2\binom{\cyc_{1}}{3} - \cyc_{1} \, .
\end{align*}
Here the complex characters of the sequence $n \mapsto \coinv{j}{n}{\cc}$ are given by $\mathbf{C}^{(j)}$ in the range $n \geq j$, sharply. Writing $\operatorname{U}(n)$ for the $n \times n$ unitary group and $\operatorname{T}(n) \leq \operatorname{U}(n)$ for the subgroup of diagonal matrices, the complex cohomology algebra of the complete flag manifold 
\begin{align*}
 \GL_{n}(\cc)/\text{(Borel subgroup)} \cong \operatorname{U}(n)/\operatorname{T}(n)
\end{align*}
is precisely $\coinv{}{\!n}{\cc}$ with the $\nn$-degrees doubled\footnote{Originally a result of Borel \cite[Proposition 26.1]{borel-fibres}, this can alternatively be deduced by applying \cite[Lemma 5.13]{dwyer-wilkerson-elementary} to the fibration sequence $\operatorname{U}(n)/\operatorname{T}(n) \rarr \operatorname{BT}(n) \rarr \operatorname{BU}(n)$.}. Thus the product formula for the generating function of the $\mathbf{C}^{(j)}$'s also follows from Chen--Specter's unpublished work \cite{chen-specter}, which is where I first saw the formula. Chen--Specter's method is to use the Grothendieck--Lefschetz fixed point formula to transform the cohomology computation into counting maximal tori over finite fields. The type-B,C analog was worked out by Fulman--Jim{\'e}nez Rolland--Wilson in \cite[Theorem 1.9]{fulman-rolland-wilson}.
\end{ex}



\subsubsection{Ordered configuration spaces} 
For any topological space $X$, the assignment 
\begin{align*}
 n &\mapsto \confix{n}{X} := \{(x_{1}, \dots, x_{n}) \in X^{n} : x_{i} \neq x_{j}\}
\end{align*}
defines a $\ters{\FI}$-space for which we write $\confix{}{X}$. Hence for each $k \geq 0$ and abelian group $\ab$, taking the $k$-th cohomology with coefficients in $\ab$ defines an $\FI$-module $\co^{k}(\confix{}{X};\ab)$. 

By feeding Miller--Wilson's method \cite{miller-wilson-PConf-hyper} of using $\FI$-hyperhomology for configuration spaces into Theorem \ref{hyper-ranges}, we obtain the next result.

\begin{thmx} \label{main-config}
Let $\manif$ be a $u$-connected topological $d$-manifold with $0 \leq u \leq d-2$, and $k \geq d-1$ be a cohomological degree. Write
\begin{align*}
 k = q_{k}(d-1) + r_{k},\, \quad 0 \leq r_{k} \leq d-2
\end{align*}
 via Euclidean division so that $q_{k} = \floor*{\frac{k}{d-1}}$, and set  
\begin{align*}
\weak_{k} :=
\begin{dcases*}
 \floor*{\frac{k}{u+1}} & \text{if $u+1 < d/2$,} \\
  2q_{k} + 1  & \text{if $d/2 \leq u+1 \leq r_{k}$,} \\
  2q_{k}  & \text{if $u+1 \geq \max\{d/2,\,r_{k}+1\}$.}
\end{dcases*}
\end{align*}
In case $d=2$ and $\manif \neq \mathbb{S}^{2}-C$ for some closed subset $C \subseteq \mathbb{S}^{2}$, we reset $\weak_{k} := 2k-1$. Then for every abelian group $\ab$, the $\FI$-module $\co^{k}(\confix{}{\manif};\ab)$ has stable ranges 
\begin{align*}
\yleq \left(
\begin{array}{r r r}
  2\weak_{k}, &2\weak_{k} + 1, &2\weak_{k} - 1,
  \vspace{0.04in}
  \\
  2\weak_{k} - 2, &\weak_{k}, &2\weak_{k}\,
\end{array}
 \right) \, .
\end{align*}
\end{thmx}

\paragraph{The case $\pmb{u=0}$.}
To apply Theorem \ref{main-config}  to any connected $d$-manifold $\manif$ with $d \geq 2$, we can take $u=0$ and for each $k \geq d-1$ get
\begin{align*}
\weak_{k} = 
\begin{dcases*}
  k & \text{if $d \geq 3$,} \\
  2k-1  & \text{if $d=2$ and $\manif \neq \mathbb{S}^{2} - C$,} \\
  2k & \text{otherwise.}
\end{dcases*}
\end{align*}
Here the improvement in the second case stems from an observation of Miller--Wilson \cite[Corollary 3.36]{miller-wilson-PConf-hyper}. The dichotomy between the $d \geq 3$ and $d=2$ cases will be familiar to the readers who have seen representation stability ranges for ordered configuration spaces in the literature such as \cite[Theorem 1]{church-config}, \cite[Theorems 1.8, 1.9]{cef}, \cite[Theorem 4.3]{cmnr-range}, \cite[Theorems 3.12, 3.27]{miller-wilson-secondary-stab}, \cite[Theorem 1.1]{miller-wilson-PConf-hyper}.

\begin{rem} \label{compare-config}
Combining \cite[Theorem 1]{church-config}, \cite[Theorem 1.8]{cef},  \cite[Corollary 3.3]{casto-fig}, \cite[Theorem 1.1]{miller-wilson-PConf-hyper}, \cite[Theorem A]{patzt-wgordon}, \cite[Theorem 4.3]{cmnr-range}, \cite[Theorem E]{putman-congruence}, \cite[Theorem 1.2]{harman-poly}, \cite{bahran-linear}, the best stable ranges established previously in the literature under the assumptions of Theorem \ref{main-config} with $u=0$ were
\begin{itemize}
\item the following if $\ab = \kk$ is a field of characteristic zero:
 \begin{align*}
\begin{dcases*}
  \yleq \left(
\begin{array}{r r r}
  2k+1, &2k+2, &4k+3,\\
  2k-2, &k, &2k\,
\end{array}
\right)
& \text{
  if $d \geq 3$, $\manif$ is orientable, and $k \leq 2d-3$,
}
\\
  \yleq \left(
\begin{array}{r r r}
  2k+1, &2k+2, &4k+3,\\
  2k-1, &k, &2k\,
\end{array}
\right)
& \text{
  if $d \geq 3$ and 
\begin{tabular}{l}
 $k \geq 2d-2$ or
 \\
 $\manif$ is non-orientable,
\end{tabular}
}
\\
  \yleq \left(
\begin{array}{r r r}
  4k+1, &4k+2, &8k+3,\\
  8k-6, &2k, &4k\,
\end{array}
\right)
& \text{
  if $d = 2$ and $\manif$ is orientable,
}
\\
  \yleq \left(
\begin{array}{r r r}
  4k+1, &4k+2, &8k+3,\\
  8k-2, &2k, &4k\,
\end{array}
\right)
& \text{
  if $d = 2$ and $\manif$ is non-orientable,
}
\end{dcases*}
\end{align*}

\item the following if $\ab = \kk$ is a field of large enough characteristic:
 \begin{align*}
\begin{dcases*}
  \yleq \left(
\begin{array}{r r r}
  2k+1, &2k+2, &4k+3,\\
  2k + 2\floor*{\frac{k}{d-1}} - 4, &k, &4k+5\,
\end{array}
\right)
& \text{
\begin{tabular}{l}
  if $d \geq 3$, $\manif$ is orientable,
  \\
  and $\ch(\kk) \geq 4k+6$,
\end{tabular}
}
 \\
  \yleq \left(
\begin{array}{r r r}
  2k+1, &2k+2, &4k+3,\\
  4k-2, &k, &4k+5\,
\end{array}
\right)
& \text{
\begin{tabular}{l}
  if $d \geq 3$, $\manif$ is non-orientable,
  \\
  and $\ch(\kk) \geq 4k+6$,
\end{tabular}
}
 \\
  \yleq \left(
\begin{array}{r r r}
  4k+1, &4k+2, &8k+3,\\
  8k-6, &2k, &8k+5\,
\end{array}
\right)
& \text{
\begin{tabular}{l}
  if $d = 2$, $\manif$ is orientable,
  \\
  and $\ch(\kk) \geq 8k+6$,
\end{tabular}
}
 \\
  \yleq \left(
\begin{array}{r r r}
  4k+1, &4k+2, &8k+3,\\
  8k-2, &2k, &8k+5\,
\end{array}
\right)
& \text{
\begin{tabular}{l}
  if $d = 2$, $\manif$ is non-orientable,
  \\
  and $\ch(\kk) \geq 8k+6$,
\end{tabular}
}
\end{dcases*}
\end{align*}

\item the following for general $\ab$:
\begin{align*}
\begin{dcases*}
  \yleq \left(
\begin{array}{r r r}
  2k+1, &2k+2, &4k+3,\\
  2k + 2\floor*{\frac{k}{d-1}} - 4, &k, &M_{k}\,
\end{array}
\right)
& \text{if $d \geq 3$ and $\manif$ is orientable,
}
 \\
  \yleq \left(
\begin{array}{r r r}
  2k+1, &2k+2, &4k+3,\\
  4k-2, &k, &M'_{k}\,
\end{array}
\right)
& \text{if $d \geq 3$ and $\manif$ is non-orientable,
}
 \\
  \yleq \left(
\begin{array}{r r r}
  4k+1, &4k+2, &8k+3,\\
  8k-6, &2k, &N_{k}\,
\end{array}
\right)
& \text{if $d = 2$ and $\manif$ is orientable,
}   
 \\
  \yleq \left(
\begin{array}{r r r}
  4k+1, &4k+2, &8k+3,\\
  8k-2, &2k, &N'_{k}\,
\end{array}
\right)
& \text{if $d = 2$ and $\manif$ is non-orientable,
}  
 \\ 
\end{dcases*}
\end{align*}
where $M_{k},M'_{k},N_{k},N'_{k}$ were undetermined.
\end{itemize}
\end{rem}

\paragraph{Increasing the connectivity.}
For a $u$-connected $d$-manifold $\manif$ with $u \geq 1$ and $d \geq 3$, improved stable ranges for the homological stability of \textbf{unordered} configuration spaces of $\manif$ are known \cite[Proposition 4.1]{church-config}, but to my knowledge similar improvements have not been written down for the representation stability of ordered configuration spaces of $\manif$. Applying Theorem \ref{main-config} to such $\manif$ breaks down to three mutually exclusive cases: 
\begin{birki}
 \item $u + 1 < d/2$ and $\co_{u+1}(\manif;\zz) \neq 0$: in this case $u \leq d-2$ as well and given $k \geq d-1$, Theorem \ref{main-config} applies with $\weak_{k} = \floor*{\frac{k}{u+1}}$.  
 \item $d=2u+2$ and $\co_{u+1}(\manif;\zz) \neq 0$: such $\manif$ is often called \emph{highly connected}\footnote{There is a significant body of work for characterizing these manifolds in the smooth category that at least goes back to Wall \cite{wall-highly-conn}, yet has recent contributions \cite{burklund-senger}.}. Given $k \geq 2u+1$, writing
\begin{align*}
 k = q_{k}(2u+1) + r_{k},\, \quad 0 \leq r_{k} \leq 2u
\end{align*}
 via Euclidean division so that $q_{k} = \floor*{\frac{k}{2u+1}}$, Theorem \ref{main-config} applies with
\begin{align*}
\weak_{k} =
\begin{dcases*}
  2q_{k}  & \text{if $0 \leq r_{k} \leq u$,}
  \\
  2q_{k} + 1  & \text{if $u < r_{k} \leq 2u$.} 
\end{dcases*}
\end{align*}

\item $u + 1> d/2$: in this case Poincar{\'e} duality pushes the connectivity to the top and forces that $\manif$ is either
\begin{itemize}
 \item a homotopy $d$-sphere (hence $\manif = \mathbb{S}^{d}$ by the Poincar{\'e} conjecture!), or
 \item contractible, in which case an inspection of the spectral sequence in \cite[Theorem 1]{totaro-config} such as \cite[Section 3.2]{church-config} bears that 
\begin{align*}
  \co^{k}(\confix{}{\manif};\ab) \cong \co^{k}(\confix{}{\rr^{d}};\ab) 
\end{align*}
for any coefficients $\ab$.
\end{itemize}
In any case, here $\manif$ is $(d-2)$-connected so given $k \geq d-1$, Theorem \ref{main-config} applies with $\weak_{k} = 2\floor*{\frac{k}{d-1}}$.
\end{birki}
We obtain these improved ranges in the presence of higher connectivity by using Totaro's \cite{totaro-config} and Church--Ellenberg--Farb's \cite{cef} description of the Leray spectral sequence of the inclusion $\confix{}{\manif} \emb \manif^{\bul}$ in Theorem \ref{non-compact-generation}.



\begin{ex}
 For every $d \geq 2$ and $i \geq 1$, taking $\manif := \rr^{d}$, $u := d-2$, and $k:=i(d-1)$ yields $\weak_{i(d-1)} = 2i$ in Theorem \ref{main-config}. The implied polynomiality of degree $\leq 2i$ here is sharp: writing $D(r,\ell)$ for the number of \textbf{derangements} in $\sym{r}$ with $\ell$ cycles, it follows from the explicit computations of Hersh--Reiner \cite[(27), Remark 2.9, Corollary 2.10]{hersh-reiner} that for every $n \geq 0$ we have
\begin{align*}
 \dim_{\qq} \co^{i(d-1)}(\confix{n}{\rr^{d}};\qq) = 
 \sum_{r=i+1}^{2i} D(r,r-i) \binom{n}{r} \, .
\end{align*}
Moreover by \cite[Theorem 1.1, Corollary 2.10]{hersh-reiner}, it follows that $\co^{i(d-1)}(\confix{}{\rr^{d}};\qq)$ has stable ranges
\begin{align*}
\begin{cases}
 \yleq \left(
\begin{array}{r r r}
  2i, &-1, &4i-1,\\
  -1, &2i, &3i\,
\end{array}
\right)
& \text{if $d$ is odd,}
\vspace{0.2cm}
\\
 \yleq \left(
\begin{array}{r r r}
  2i, &-1, &4i-1,\\
  -1, &2i, &3i+1\,
\end{array}
\right)
& \text{if $d$ is even,}
\end{cases}
\end{align*}
and all of them are sharp.
\end{ex}

\begin{ex} \label{ex-sphere}
 For $d \geq 2$, Theorem \ref{main-config} yields that $\co^{d-1}(\confix{}{\mathbb{S}^{d}};\ab)$ has stable ranges
\begin{align*}
 \yleq (4,\,\, 5,\,\, 3,\,\,
 2,\,\, 2,\,\, 4) \, .
\end{align*}
On the other hand, the explicit computations of Feichtner--Ziegler \cite[Theorem 2.4, Theorem 5.4]{fe-zi-sphere} imply that when $d$ is even,
\begin{align*}
 \dim_{\qq}\co^{d-1}(\confix{n}{\mathbb{S}^{d}};\qq) 
 &= \binom{n-1}{2} - 1 = \binom{n-3}{2} + 2(n-3)
\end{align*}
in the range $n \geq 3$. Hence the (2,\,2) part of the stable ranges given by Theorem \ref{main-config} is sharp. In fact \textbf{all} of the stable ranges are sharp here (the $\FI$-module is the $\mathbf{V}(2)$, which is defined in Section \ref{section-witness}, of Theorem \ref{main-ranges}); see Example \ref{ex-sphere-deep}.
\end{ex}

\begin{rem}[Low degrees] \label{low-degrees}
 The main reason for the exclusion of the degrees $k \leq d-2$ in Theorem \ref{main-config} is that they are more easily handled via traditional means: for a connected manifold $\manif$,  the inclusion map 
$\confix{}{\manif} \emb \manif^{\bul}$ of $\ters{\FI}$-spaces is $(d-1)$-connected \cite[Theorem 3.2]{ggg-homotopy-fiber}, and hence by the relative Hurewicz and the universal coefficient theorems (together with their naturality), in degrees $k \leq d-2$ there is an isomorphism 
\begin{align*}
 \co^{k}(\confix{}{\manif};\ab) \cong \co^{k}(\manif^{\bul};\ab)
\end{align*}
of $\FI$-modules with any coefficients $\ab$.
\end{rem}

\begin{rem}[Extra structure]
 A consequence of the isomorphism in Remark \ref{low-degrees} is that when $k \leq d-2$, the $\FI$-module $\co^{k}(\confix{}{\manif};\ab)$ extends to an $\sh$-module, where $\sh$ denotes the category of partial bijections. More interestingly, when $\manif$ is non-compact an $\sh$-extension can be made in \textbf{all} degrees as in \cite[Section 6.4]{cef} and \cite[Section 3.1]{miller-wilson-secondary-stab}. In these cases, using the notation of Theorem \ref{main-config}, the stable ranges of $\co^{k}(\confix{}{\manif};\ab)$ can be improved to 
\begin{align*}
\yleq \left(
\begin{array}{r r r}
  \weak_{k}, &-1, &\max\{0,2\weak_{k} - 1\},
  \vspace{0.04in}
  \\
  -1, &\weak_{k}, &2\weak_{k}\,
\end{array}
 \right) \,
\end{align*} 
and the stabilization of the $\sym{n}$-representations is more rigid. In fact the generation in degrees $\leq \weak_{k}$ for non-compact manifolds (see Theorem \ref{reduce-PConf-to-punctured} and Theorem \ref{non-compact-generation}) is used as an input for Theorem \ref{main-config} through the \textbf{puncture resolution} \cite[Section 3]{miller-wilson-PConf-hyper}. I intend to treat the existence and consequences of such extra structures on $\FI$-modules in more detail in future work. 
\end{rem}


\subsubsection{Congruence subgroups}
For every ring $R$, the assignment $n \mapsto \GL_{n}(R)$ defines an $\FI$-group, for which we write $\GL_{\bul}(R)$. If $I$ is an ideal of $R$, as the  kernel of the mod-$I$ reduction we get a smaller $\FI$-group 
\begin{align*}
 \GL_{\bul}(R,I) := \ker \big(\!\GL_{\bul}(R) \rarr \GL_{\bul}(R/I) \big) \, ,
\end{align*}
called the $I$\textbf{-congruence subgroup} of $\GL_{\bul}(R)$.
For each $k \geq 0$ and abelian group $\ab$, taking the $k$-th homology with coefficients in $\ab$ defines an $\FI$-module $\co_{k}(\GL_{\bul}(R,I);\ab)$.

\paragraph{Stable rank of a ring.} Let $R$ be a nonzero unital (associative) ring. A column vector $\mathbf{v} \in \mat_{m \times 1}(R)$ of size $m$ is \textbf{unimodular} if there is a row vector $\mathbf{u} \in \mat_{1 \times m}(R)$ such that $\mathbf{u}\mathbf{v} = 1$. Writing $\mathbf{I}_{r} \in \mat_{r \times r}(R)$ for the identity matrix of size $r$, we say
a column vector $\mathbf{v}$ of size $m$ is \textbf{reducible} if there exists  $\mathbf{A} \in \mat_{(m-1) \times m}(R)$ with block form  $\mathbf{A} = [\mathbf{I}_{m-1} \,|\, \mathbf{x}]$ such that the column vector $\mathbf{A} \mathbf{v}$ (of size $m-1$) is unimodular. We write $\sr(R) \leq s$ if every unimodular column vector of size $> s$ is reducible. 

\begin{rem}
  If $R$ is a finite algebra over a commutative Noetherian ring of Krull dimension $d$, then $\sr(R) \leq d+1$ \cite[Theorem 11.1]{bass-k-theory}. This is a sharp bound for every $d \geq 0$: declaring $\sr(R) = s$ when $\sr(R) \leq s$ and $\sr(R) \nleq s-1$, we have $\sr(\rr[x_{1}, \dots, x_{d}]) = d+1$ by \cite[Theorem 8]{vaserstein-st-rank-top}. A witness to this sharpness is the unimodular column vector $\mathbf{v}_{d+1} := \left[x_{1},\, \cdots,\, x_{d},\, x_{1}^{2} + \cdots + x_{d}^{2} - 1\right]^{T}$
of size $d+1$ which is not reducible \cite[proof of Theorem 8]{vaserstein-st-rank-top}. As a slightly different example, fix $d \geq 1$ and consider the ring $\zz[x_{1},\dots,x_{d-1}]$: its Krull dimension is $d$, so $\sr \leq d+1$, the unimodular $\mathbf{v}_{d}$ is not reducible here either \cite[Corollary 19.1]{vas-sus}, so $\sr \nleq d-1$. The precise value turns out to be
\begin{align*}
 \sr(\zz[x_{1},\dots,x_{d-1}]) = 
\begin{cases}
 d+1 & \text{if $d=1,2$,}
 \\
 d & \text{if $d \geq 3$,}
\end{cases}
\end{align*}
by \cite{guyot-ZX} and \cite[Theorem 17.2]{vas-sus}.
\end{rem}

The degrees of $\FI$-hyperhomology groups of the chain complex $\chain_{\star}(\GL_{\bul}(R,I); \ab)$ have already been bounded in \cite[Proposition 5.4]{cmnr-range} in terms of the stable rank of $R$. Feeding this into Theorem \ref{hyper-ranges}, together with a refinement of Djament \cite[Th\'{e}or\`{e}me 2]{djament-congruence-stab} and an optimization in low degrees, we get our last application.

\begin{thmx} \label{main-cong} Let $I$ be a proper ideal in a ring $R$ with $\sr(R) \leq s$. Then for every homological degree $k \geq 1$ and abelian group $\ab$, the $\FI$-module $\co_{k}(\GL_{\bul}(R,I);\ab)$ has stable ranges
\begin{align*} 
\begin{cases}
\yleq \left(
\begin{array}{r r r}
  s+1, &s+3, &2s+4,\\
  2s+3, &2, &2s + 4\,
\end{array}
 \right) & \text{if $k=1$,}
 \vspace{0.25cm}
 \\
 \yleq \left(
\begin{array}{r r r}
  2s + 5, &2s + 6, &2s + 9,\\
  2s+8, &4, &2s + 9\,
\end{array}
 \right) & \text{if $k = 2$,}
 \vspace{0.25cm}
 \\
 \yleq \left(
\begin{array}{r r r}
  4k + 2s - 2, &4k + 2s-1, &4k + 2s + 1,\\
  4k + 2s, &2k, &4k + 2s + 1\,
\end{array}
 \right) & \text{if $k \geq 3$.}
\end{cases}
\end{align*}
\end{thmx} 

\begin{rem} \label{compare-cong}
  Combining \cite[Theorem 11]{gan-li-linear}, \cite[Theorem D']{ce-homology}, \cite[Theorem A]{patzt-wgordon}, \cite[Theorem 1.3]{li-FIG-invariants}, \cite[Remark 9]{gan-li-linear}, \cite[Th\'{e}or\`{e}me 2]{djament-congruence-stab}, \cite[Theorem E]{putman-congruence}, \cite[Theorem 1.2]{harman-poly}, the best stable ranges established previously in the literature under the assumptions of Theorem \ref{main-cong} were
\begin{align*} 
\begin{cases}
\yleq \left(
\begin{array}{r r r}
  s+1, &s+3, &2s+5,\\
  2s+3, &2, &M_{1}\,
\end{array}
 \right) & \text{if $k=1$,}
\vspace{0.25cm}
 \\
 \yleq \left(
\begin{array}{r r r}
  2s + 5, &2s + 6, &4s + 11,\\
  4s+10, &4, &M_{2}\,
\end{array}
 \right) & \text{if $k = 2$,}
 \vspace{0.25cm}
\\
\yleq \left(
\begin{array}{r r r}
  4k+2s-1, &4k+2s+4, &8k+4s+7,\\
  8k+4s+2, &2k, & M_{k}\,
\end{array}
 \right) & \text{if $k \geq 3$,} 
\end{cases}
\end{align*}
where 
\begin{align*}
 M_{k} = 
\begin{cases}
 2s + 7 & \text{if
\begin{tabular}{l}
 $k=1$ and $\ab = \kk$ is a field with \\
 $\ch(\kk) = 0$ or $\ch(\kk) \geq 2s+8$,
\end{tabular}
 }
 \\
 4s + 13 & \text{if
\begin{tabular}{l}
 $k=2$ and $\ab = \kk$ is a field with \\
 $\ch(\kk) = 0$ or $\ch(\kk) \geq 4s+14$,
\end{tabular}
 }
 \\
 8k + 4s + 9 & \text{if
\begin{tabular}{l}
 $k \geq 3$ and $\ab = \kk$ is a field with \\
 $\ch(\kk) = 0$ or $\ch(\kk) \geq 8k+4s+10$,
\end{tabular}
 }
 \\
 \text{undetermined} & \text{otherwise.} 
\end{cases}
\end{align*}
\end{rem}

\begin{ex} Given a prime $p$ and $\ell \geq 2$, the computation
 \begin{align*}
 \dim_{\kk_{p}} \co_{k}(
 \GL_{n}(\zz/p^{\ell},p);\kk_{p}
 ) = \binom{n^{2}+k-1}{k}
\end{align*}
(see \cite[the proof of Theorem D]{cmnr-range}) shows that the $2k$ in Theorem \ref{main-cong} cannot be improved. In fact the 2k is sharp for $\co_{k}(\GL_{\bul}(R,I);\zz)$ whenever $I \neq I^{2}$ by \cite[Th\'{e}or\`{e}me 2]{djament-congruence-stab} (see Proposition \ref{relate-to-faible} and Theorem \ref{djament-2k-result}).
\end{ex}

\section{$\FI$-modules} 
\subsection{Preliminaries}

\paragraph{Derivative and local cohomology functors.} The functor $- \sqcup \{*\} \colon \FI \rarr \FI$
receives a natural transformation from $\id_{\FI}$. Hence due directly to its definition, the shift functor
\begin{align*}
 \shift{}{} \colon \lMod{\FI} \rarr \lMod{\FI}
\end{align*}
receives a natural transformation from the identity functor $\id_{\lMod{\FI}}$, whose cokernel
\begin{align*}
 \deriv := \coker \left( \id_{\lMod{\FI}} \rarr \shift{}{} \right)
\end{align*}
is called the \textbf{derivative functor}.

An $\FI$-module $V$ is called \textbf{torsion} if for every finite set $S$ and $x \in V_{S}$, there exists an injection $\alpha \colon S \emb T$ such that $V_{\alpha}(x) = 0 \in V_{T}$. There is a left exact functor 
\begin{align*}
 \locoh{0} \colon \lMod{\FI} \rarr \lMod{\FI}
\end{align*}
which assigns an $\FI$-module its largest torsion $\FI$-submodule; see \cite[Section 5.1, Definition 5.11]{li-ramos}. For each $j \geq 0$, we write $\locoh{j} := \operatorname{R}^{j}\!\locoh{0}$ for its $j$-th right derived functor, and write 
\begin{align*}
 h^{j}(V) := \deg(\locoh{j}(V))
\end{align*}
for every $\FI$-module $V$.

\begin{rem}[Degree conventions]
 Our convention is that (see the introduction) the zero $\FB$-module has $\deg(0) = -1$ as in \cite{cmnr-range} while it is perhaps more common that it is taken to be $-\infty$ as in \cite{li-ramos} and \cite{nss-regularity}. 
\end{rem}

\begin{lem} \label{torsion-zero-colim}
 For an $\FI$-module $V$, the following are equivalent:
\begin{birki}
 \item $V$ is torsion.
 \item Considering $\nn = \{0,1,\dots\}$ with the usual ordering as a category and the natural embedding $\iota\colon \nn \rarr \FI$ via $\iota(n) := \{1,\dots,n\}$,
we have 
\begin{align*}
 \colim(V \circ \iota) = \colim_{n \in \nn} V_{n} = 0 \, .
\end{align*}
\end{birki}
\end{lem}
\begin{proof}
 Noting that $\nn$ is a directed set, the usual construction of a directed colimit shows that the vanishing of the colimit is equivalent to $V$ being torsion.
\end{proof}

\paragraph{Local and stable degrees.} For every $\FI$-module $V$, we write
\begin{align*}
 \local(V) &:= \max\{h^{j}(V) : j \geq 0\} \, ,
 \\&\in \{-1,0,1,\dots\} \cup \{\infty\} \, ,
\end{align*}
called the \textbf{local degree} of $V$, and 
\begin{align*}
 \weak(V) &:= \min\{r \geq -1 : \deriv^{r+1}(V) \text{ is torsion}\}
 \\&\in \{-1,0,1,\dots\} \cup \{\infty\} \, ,
\end{align*}
called the \textbf{stable degree} of $V$.

\begin{prop}\label{relate-to-faible}
 Given an $\FI$-module $V$ and $g \geq -1$, the following are equivalent: 
\begin{birki}
 \item $\weak(V) \leq g$.
 \item In the sense of \emph{\cite[D\'efinition 2.22]{djament-vespa-weakly}}, $V$ is weakly polynomial of degree $\leq g$.
\end{birki}
\end{prop}
\begin{proof}
 We employ induction on $g$. For the base case $g= - 1$, (1) means $V$ is torsion from the definition of the stable degree $\weak(V)$; and (2) means  $\pi_{\FI}(V) = 0$ inside the category $\mathbf{St}(\FI,\lMod{\zz})$ defined via \cite[D\'efinition 2.16]{djament-vespa-weakly} the quotient construction 
\begin{align*}
 \pi_{\FI} \colon \lMod{\FI} \rarr \lMod{\FI} / \mathcal{S}n(\FI,\lMod{\zz}) =: \mathbf{St}(\FI, \lMod{\zz}) \, .
\end{align*}
Here $\mathcal{S}n(\FI,\lMod{\zz})$ is the thick subcategory of \emph{stably null} $\FI$-modules \cite[D\'efinition 2.10]{djament-vespa-weakly}, which are precisely the $\FI$-modules that satisfy the condition (2) in Lemma \ref{torsion-zero-colim} by \cite[Proposition 5.7]{djament-vespa-weakly}. Finally, the equivalence 
\begin{align*}
 \pi_{\FI}(V) = 0 \in \lMod{\FI} / \mathcal{S}n(\FI,\lMod{\zz})
 \Leftrightarrow V \in \mathcal{S}n(\FI,\lMod{\zz})
\end{align*}
follows from a general result about quotient categories in the abelian setting, namely we apply \cite[Lemme 2]{gabriel-abelian} to $\id_{V}$.

Next, we fix $g \geq 0$ and assume that the equivalence between (1) and (2) holds for $g-1$.  Assuming $\weak(V) \leq g$, by the definition of stable degree $\weak(V)$ we have $\weak(\deriv(V)) \leq g-1$ and hence our induction hypothesis gives that $\deriv(V)$ is weakly polynomial of degree $\leq g-1$. For every $a \in \nn$, let us write 
\begin{align*}
 Q_{a}(V) := \coker (V \rarr \shift{V}{a}) \, ,
\end{align*}
which is denoted by $\weak_{\{1,\dots,a\}}(V)$ in \cite[page 10]{djament-vespa-weakly}. We claim that for every $a \in \nn$, the $\FI$-module $Q_{a}(V)$ is also weakly polynomial of degree $\leq g-1$, that is, 
\begin{align*}
 \pi_{\FI}(Q_{a}(V)) \in \mathcal{P}ol_{g-1}(\FI,\lMod{\zz})
\end{align*}
in the sense of \cite[Definition 2.22]{djament-vespa-weakly}. This follows by induction on $a$, noting that $\pi_{\FI}$ is exact, $Q_{0}(V) = 0$, $Q_{1}(V) = \deriv(V)$ (we have already shown $\pi_{\FI}(\deriv(V)) \in \mathcal{P}ol_{g-1}(\FI,\lMod{\zz})$ at this point) and the exact sequences 
\begin{align*}
 &\deriv(V) \rarr Q_{a+1}(V) \rarr \shift{Q_{a}(V)}{} \rarr 0 \quad 
 \text{(by \cite[Proposition 2.4, part (7)]{djament-vespa-weakly})}
 \\
 &Q_{a}(V) \rarr \shift{Q_{a}(V)}{} \rarr \deriv(Q_{a}(V)) \rarr 0 
 \quad \text{(by definition of $\deriv$)}
\end{align*}
of $\FI$-modules, because $\mathcal{P}ol_{g-1}(\FI,\lMod{\zz})$ is a thick subcategory of $\mathbf{St}(\FI, \lMod{\zz})$ by \cite[Proposition 2.25]{djament-vespa-weakly}. It now follows from the recursive part of \cite[Definition 2.22]{djament-vespa-weakly} and \cite[Proposition 2.19, part (1)]{djament-vespa-weakly} that $\pi_{\FI}(V) \in \mathcal{P}ol_{g}(\FI,\lMod{\zz})$, in other words $V$ is weakly polynomial of degree $\leq g$.

For the converse, assume $V$ is weakly polynomial of degree $\leq g$. Then $\pi_{\FI}(V) \in \mathcal{P}ol_{g}(\FI,\lMod{\zz})$, and hence in particular by the recursive part of \cite[Definition 2.22]{djament-vespa-weakly} and \cite[Proposition 2.19, part (1)]{djament-vespa-weakly}, we have
\begin{align*}
 Q_{1}(\pi_{\FI}(V)) = \pi_{\FI}(Q_{1}(V)) = \pi_{\FI}(\deriv(V)) \in \mathcal{P}ol_{g-1}(\FI,\lMod{\zz}) \, ,
\end{align*}
in other words $\deriv(V)$ is weakly polynomial of degree $\leq g-1$. Our induction hypothesis now applies to yield $\weak(\deriv(V)) \leq g-1$, and hence $\weak(V) \leq g$ by definition.
\end{proof}

We now recall several characterizations of $\cofi{0}$-acyclic modules generated in finite degrees. Here $\induce$ is the left adjoint of the restriction $\Res_{\FB}^{\FI} \colon \lMod{\FI} \rarr \lMod{\FB}$.

\begin{thm}[{\cite[Theorem A]{ramos-fig}, \cite[Proposition 14]{gan-shift-seq}, \cite[Corollary 2.13]{cmnr-range}}] \label{characterize-H0-acyclic}
Let $V$ be an $\FI$-module generated in finite degrees. Then the following are equivalent: 
\begin{enumerate}[(1)]
 \item[$(1)$] There is a finite filtration
\begin{align*}
 0 = V^{(-1)} \leq V^{(0)} \leq \cdots \leq V^{(r)} = V
\end{align*}
of $\FI$-submodules such that for each $0 \leq i \leq r$ we have $V^{(i)}/V^{(i-1)} \cong \induce(W^{(i)})$ for some $\FB$-module $W^{(i)}$ with $\deg(W^{(i)}) < \infty$.
 \vspace{0,08in}
 \item[$(1')$] In the sense of \emph{\cite[Section 2.1]{cmnr-range}}, $V$ is semi-induced.
 \item[$(1'')$] In the sense of \emph{\cite[Section 2.1]{ramos-coh},\cite[page 166]{ramos-fig}}, $V$ is $\sharp$-filtered.
 \vspace{0.08in}
 \item[$(2)$] $V$ is $\cofi{0}$-acyclic.
 \vspace{0.08in}
 \item[$(2')$] $t_{i}(V) = -1$ for every $i \geq 1$.  
 \vspace{0.08in} 
 \item [$(3)$]$\cofi{i}(V) = 0$ for some $i \geq 1$.
 \vspace{0.08in} 
 \item[$(3')$] $t_{i}(V) = -1$ for some $i \geq 1$.
 \vspace{0.08in}
 \item[$(4)$] $\locoh{j}(V) = 0$ for every $j \geq 0$.
 \vspace{0.08in} 
 \item[$(4')$] $\local(V) = -1$.
\end{enumerate}
\end{thm}
%

\begin{prop}[{\cite[page 880]{ramos-coh}}] \label{pres-vs-coh}
 Let $V$ be an $\FI$-module. The following are equivalent: 
\begin{birki}
 \item In the sense of \emph{\cite{cmnr-range}} (and this paper), $V$ is presented in finite degrees.
 \item In the sense of \emph{\cite[page 879]{ramos-coh}}, $V$ is degree-wise coherent.
\end{birki}
\end{prop}
%
 
\begin{thm} \label{pass-to-cmnr}
 Let $V$ be an $\FI$-module and $c,g \geq -1$. Then the following are equivalent:
\begin{birki}
 \item The triple $(V,c,g)$ satisfies Hypothesis \ref{FI-hyp}.
 \item In the sense of \emph{\cite{cmnr-range}} (and this paper), $V$ is presented in finite degrees such that $\local(V) \leq c$ and $\weak(V) \leq g$.
\end{birki}
\end{thm}
\begin{proof}
(1) $\imp$ (2): Assume (1) holds. First note that by \cite[Theorem A]{ramos-coh}, Proposition \ref{pres-vs-coh}, and \cite[Proposition 3.1]{cmnr-range}, we have
\begin{align*}
 &\{\text{$\cofi{0}$-acyclic $\FI$-modules generated in degrees $\leq g$}\} 
 \\
 \subseteq &\{\text{torsion-free $\FI$-modules generated in degrees $\leq g$}\} \\ \subseteq &\{\text{$\FI$-modules $X$ \pres\ with $\weak(X) \leq g$}\}
\end{align*}
and the latter forms a thick subcategory of $\lMod{\FI}$ by \cite[Proposition 2.9, part (5)]{cmnr-range}. Thus by part (2) of Hypothesis \ref{FI-hyp}, the $\FI$-module $\shift{V}{x}$ is presented in finite degrees with $\weak(\shift{V}{x}) \leq g$. To show that $V$ itself is presented in finite degrees, we may assume $x=1$ by induction. By \cite[Theorem 1]{gan-shift-seq}, there is an exact sequence 
\begin{align*}
 \cofi{1}(\shift{V}{}) \rarr \shift{\!\cofi{1}(V)}{} \rarr
 \cofi{0}(V) \rarr \cofi{0}(\shift{V}{}) \rarr \shift{\!\cofi{0}(V)}{} \rarr 0
\end{align*}
of $\FB$-modules, which implies that 
\begin{align*}
 \trel(V) - 1 \leq \max\{\trel(\shift{V}{}),\tgen(V))\} \quad \text{and} \quad \tgen(\shift{V}{}) \geq \tgen(V) - 1\, .
\end{align*}
As $\tgen(\shift{V}{}), \trel(\shift{V}{})$ are finite, so are $\tgen(V), \trel(V)$. Having shown $V$ is presented in finite degrees, we deduce $\weak(V) = \weak(\shift{V}{}) \leq g$ by \cite[Corollary 2.9, part (2)]{cmnr-range}; and part (1) of Hypothesis \ref{FI-hyp} with Theorem \ref{characterize-H0-acyclic} yields 
$
  \local(V) \leq c
$
 by \cite[Corollary 2.13]{cmnr-range}. 

(2) $\imp$ (1): Here \cite[Corollary 2.13]{cmnr-range} immediately yields part (1) of Hypothesis \ref{FI-hyp}, that is, $\shift{V}{c+1}$ is $\cofi{0}$-acyclic. Also by \cite[Proposition 2.9]{cmnr-range}, $\shift{V}{c+1}$ is generated in degrees 
$
\leq \weak(\shift{V}{c+1}) = \weak(V) \leq g
$,
verifying part (2) of Hypothesis \ref{FI-hyp}. 
\end{proof}

\begin{defn}
 For an $\FI$-module $V$ and $N \in \nn$, we write $V_{\gen{\leq N}}$ for the smallest $\FI$-submodule $U$ of $V$ such that $U_{S} = V_{S}$ for $|S| \leq N$.
\end{defn} 
 
\begin{prop}[{\cite[Proposition 4.3]{ce-homology}, \cite[Lemma 19]{gan-shift-seq}}] \label{separate-t1-gen}
 Suppose $V$ is an $\FI$-module presented in finite degrees. Then given $N \geq \trel(V)-1$, the quotient $\FI$-module $Q$ defined by the short exact sequence 
\begin{align*}
 0 \rarr V_{\gen{\leq N}} \rarr V \rarr Q \rarr 0
\end{align*}
is $\cofi{0}$-acyclic.
\end{prop}
\begin{proof}
To prove $Q$ is $\cofi{0}$-acyclic, it suffices to show that $\cofi{1}(Q) = 0$ by Theorem \ref{characterize-H0-acyclic}. Since the $\FI$-modules $V_{\gen{\leq N}}$ and $V$ acts identically on sets of size $\leq N$, the quotient $Q$ is supported in degrees $\geq\!N+1$, or in the sense of \cite[Section 3.1]{gan-shift-seq} we have $\low(Q) \geq N+1$. Hence by \cite[Lemma 5]{gan-shift-seq}, $\cofi{1}(Q)$ is supported in degrees $\geq \low(Q) + 1 \geq N+2$.

On the other hand, $\cofi{1}(V)$ is supported in degrees $\leq \trel(V) \leq N+1$ and $\cofi{0}(V_{\gen{\leq N}})$ is supported in degrees $\leq N$. Therefore the domain and codomain of each map in the exact sequence 
\begin{align*}
 \cofi{1}(V) \rarr \cofi{1}(Q) \rarr \cofi{0}(V_{\gen{\leq N}}) 
\end{align*}
of $\FB$-modules are supported in disjoint degrees, hence is identically zero. This forces $\cofi{1}(Q) = 0$.
\end{proof}

\begin{cor} \label{t1-not-0}
 Let $V$ be an $\FI$-module presented in finite degrees. Then $\trel(V) \neq 0$. If furthermore $V$ is not $\cofi{0}$-acyclic, then $\trel(V) \geq 1$.
\end{cor}
\begin{proof}
 If $\trel(V)$ were $0$, we could take $N = -1$ in Proposition \ref{separate-t1-gen} and conclude $V$ is $\cofi{0}$-acyclic, necessitating $\trel(V) = -1$. Moreover, $\trel(V) = -1$ implies $V$ is $\cofi{0}$-acyclic modules by Theorem \ref{characterize-H0-acyclic}.
\end{proof}

\begin{cor} \label{t0-vs-t1}
 For every $\FI$-module $V$ presented in finite degrees, we have 
\begin{align*}
 \tgen(V) \leq \max\{\trel(V)-1,\weak(V)\} \, ,
\end{align*}
and more specifically $\tgen(V) \geq \trel(V) \,\imp\, \tgen(V) = \weak(V)$.
\end{cor}
\begin{proof}
 We take $N:= \trel(V)-1$ and apply Proposition \ref{separate-t1-gen} to get $Q$ as stated. Because $Q$ is a quotient of an $\FI$-module generated in finite degrees and $\trel(Q) = -1$, it is presented in finite degrees. Now from the exact sequence we can read off 
\begin{align*}
 \tgen(V) \leq \max\left\{\tgen\!\left(V_{\gen{\leq N}}\right),\, \tgen(Q) \right\} \leq \max\{N,\, \weak(Q)\} \leq \max\{N,\,\weak(V)\} \, ,
\end{align*}
using \cite[Proposition 2.9, parts (1) and (6)]{cmnr-range}. The last claim follows from \cite[Proposition 2.9, part (4)]{cmnr-range}.
\end{proof}



We conclude this section by recalling the structure theorem for the local cohomology functors of $\FI$-modules.
\begin{thm}[{\cite[Theorem 2.10]{cmnr-range} and its proof}] \label{structure-complex}
 Let $V$ be an $\FI$-module such that the triple $(V,c,g)$ satisfies Hypothesis \ref{FI-hyp}. Then there is a complex 
\begin{align*}
I^{\star} \colon 0 \rarr I^{0} \rarr I^{1} \rarr \cdots \rarr I^{g+1} \rarr 0
\end{align*}
of $\FI$-modules such that 
\begin{itemize}
 \item $I^{0} = V$,
 \item For every $1\leq j \leq g+1$, the triple $\left(I^{j},\,\,-1,\,\,g-j+1\right)$ satisfies Hypothesis \ref{FI-hyp},
 \item $I^{\star}$ is exact in degrees $\geq c + 1$,
 \item $\co^{j}(I^{\star}) \cong \locoh{j}(V)$ for every $j \geq 0$, 
 \item $\deg (\co^{j} (I^{\star})) = h^{j}(V) \leq 2g - 2j + 2$ whenever $2 \leq j \leq g+1$.
\end{itemize}
Moreover if $V$ lies in a full subcategory $\mathcal{X}$ of $\lMod{\FI}$ that is closed under taking cokernels and shifts, then each $I^{j}$ can be chosen in $\mathcal{X}$.
\end{thm}

\subsection{Bounding the regularity}
In this section we prove Theorem \ref{improved-regularity}, Corollary \ref{H0-acyclic-complex-range}, and the regularity bound of Theorem \ref{regularity-from-FI-hyp} as Theorem \ref{regularity-chad-bound}.

\begin{proof}[Proof of \textbf{\emph{Theorem \ref{improved-regularity}}}]

Let $V$ be an $\FI$-module with $0 \leq \tgen(V) \leq a$ and $0 \leq \trel(V) \leq b$. By Proposition \ref{separate-t1-gen}, writing $V' := V_{\gen{\leq\, b-1}}$ we get  
\begin{align*}
 \tgen(V') &\leq \min\{a,b-1\}\,, \\
 t_{i}(V') &= t_{i}(V) \text{ for every $i \geq 1$}.
\end{align*}
In particular $\reg(V') = \reg(V)$, and by \cite[Theorem A]{ce-homology}
\begin{align*}
 \reg(V') &\leq \min\{a,b-1\} + b - 1 
=
\begin{cases}
 a + b  - 1 \, & \text{if\, $a < b$,} \\
 2b - 2 & \text{if\, $a \geq b$.} 
\end{cases}
\end{align*}
This proves the ``$\leq$'' part of the equality. 

For the ``$\geq$'' part, we shall construct an $\FI$-module that realizes the bound. First note that for every pair of integers $0 \leq g < r$, by \cite[Sharpness of Theorem E]{ce-homology} there is a finitely generated $\FI$-module $U^{g,r}$ such that
\begin{align*}
 \tgen(U^{g,r}) = g \, , \quad
 \trel(U^{g,r}) = r \, , \quad
 \deg(U^{g,r}) = g + r - 1 \, .
\end{align*}
By \cite[Lemma 4.2]{nss-regularity} we have $\reg(U^{g,r}) = g + r -1$. Next, for every integer $d \geq 0$ fix a nonzero $\FB$-module $W^{d}$ supported only in the degree $d$. Now given $a,b \geq 0$, setting
\begin{align*}
 U := 
\begin{cases}
 U^{a,b} & \text{if $a < b$,} \\
 U^{b-1,b} \oplus \induce(W^{a}) & \text{if $a \geq b$,}
\end{cases}
\end{align*}
we see that $\tgen(U) = a$, $\trel(U) = b$, and 
\begin{align*}
 \reg(U) =
 \begin{cases}
 a+b-1 & \text{if $a < b$,} 
 \\
 2b-2 & \text{if $a \geq b$,}
\end{cases}
\end{align*}
as desired.
\end{proof}

\begin{proof}[Proof of \textbf{\emph{Corollary \ref{H0-acyclic-complex-range}}}]
 Writing $\partial_{k} \colon \chain_{k} \rarr \chain_{k-1}$ for the boundary map, it suffices to show that $Z_{k} := \ker \partial_{k}$ is generated in degrees
\begin{align*}
 \leq \begin{cases}
 g_{k-1} + g_{k} + 1 
 & \text{if $g_{k} > g_{k-1}$,}
 \\
 2g_{k} & \text{if $g_{k} \leq g_{k-1}$.}
\end{cases} 
\end{align*} 
The proof is a version of the argument in \cite[Corollary 4]{gan-li-linear}. Applying $\co_{0}^{\FI}$ to the short exact sequence 
\begin{align*}
 0 \rarr Z_{k} \rarr \chain_{k} \rarr \chain_{k}/Z_{k} \rarr 0
\end{align*}
of $\FI$-modules gives rise to a long exact sequence in terms of the derived functors $\co_{i}^{\FI}$. Since $\chain_{k}$ is $\co^{\FI}_{0}$-acyclic, we read off an exact sequence 
\begin{align*}
 0 \rarr \co_{1}^{\FI}(\chain_{k}/Z_{k}) \rarr \co_{0}^{\FI}(Z_{k}) \rarr \co_{0}^{\FI}(\chain_{k}) \rarr \co_{0}^{\FI}(\chain_{k}/Z_{k}) \rarr 0
\end{align*}
and isomorphisms $\co_{i+1}^{\FI}(\chain_{k}/Z_{k}) \cong \co_{i}^{\FI}(Z_{k})$ for all $i \geq 1$. In particular, we have 
\begin{align*}
 \tgen(Z_{k}) \leq \max\{\trel(\chain_{k}/Z_{k}), \tgen(\chain_{k})\} \leq \max\{\trel(\chain_{k}/Z_{k}), g_{k}\} \, ,
\end{align*}
so it suffices to establish the stated bounds for $\trel(\chain_{k}/Z_{k})$.
To that end, write $E_{k-1} := \coker \partial_{k}$ and consider the exact sequence 
\begin{align*}
 0 \rarr \chain_{k}/Z_{k} \rarr \chain_{k-1} \rarr E_{k-1} \rarr 0 \, 
\end{align*}
of $\FI$-modules. This time we read off a pertinent isomorphism $\cofi{1}(\chain_{k}/Z_{k}) \cong \cofi{2}(E_{k-1})$ so that $\trel(\chain_{k}/Z_{k}) = t_{2}(E_{k-1})$. From the lower degrees in the long exact sequence, we can also deduce $\tgen(E_{k-1}) \leq \tgen(\chain_{k-1}) \leq g_{k-1}$ and $\trel(E_{k-1}) \leq \tgen(\chain_{k}/Z_{k}) \leq g_{k}$. Using what we have established so far and Theorem \ref{improved-regularity}, we get
\begin{align*}
 \trel(\chain_{k}/Z_{k}) = t_{2}(E_{k-1}) &\leq \reg(E_{k-1}) + 2 
 \\
 &\leq 
\begin{cases}
 g_{k-1} + g_{k} + 1 
 & \text{if $g_{k-1} < g_{k}$,}
 \\
 2g_{k} & \text{if $g_{k-1} \geq g_{k}$.}
\end{cases}
\end{align*}
\end{proof}

\begin{thm} \label{regularity-chad-bound}
  If the triple $(V,c,g)$ satisfies Hypothesis \ref{FI-hyp}, then
\begin{align*}
 \reg(V) \leq  
 \begin{cases}
 -2 & \text{if $c = -1$,} \\
 c & \text{if $g = -1$ and $c \geq 0$,} \\
 c+1 & \text{if $0 \leq g \leq \ceil{c/2}$ and $c \geq 0$,} \\
 g + \floor{c/2} + 1  & \text{if $g > \ceil{c/2}$ and $c \geq 0$.}
\end{cases}
\end{align*}
\end{thm}
\begin{proof}
 First, note that for every $j$ we have $h^{j}(V) \leq c$ by the definition of $\local$ and Theorem \ref{pass-to-cmnr}. Now using \cite[Corollary 4.15]{ramos-coh} and Theorem \ref{structure-complex}, we get
\begin{align*}
 \reg(V) &\leq \max\{h^{j}(V) + j : 0 \leq j \leq g+1\} \\
 &\leq \max\bigg( \{h^{0}(V), h^{1}(V) + 1\} \cup \big\{ \min\{c+j, 2g-j+2\} : 2\leq j \leq g+1 \big\} \bigg) 
\end{align*}
Now since $c+j \leq 2g-j+2$ if and only if $j \leq g - c/2 + 1$,
we have
\begin{align*}
 &\,\,\,\,\,\,\,\,\big\{\min\{c+j, 2g-j+2\} : 2\leq j \leq g + 1\big\} \\
 &= 
 \big\{c+j : 2\leq j \leq g - \ceil{c/2} + 1\big\} \cup
    \big\{2g-j+2: g - \ceil{c/2} + 2 \leq j \leq g+1\big\} \\
 &= 
\begin{cases}
 \empt & \text{if $g \leq 0$ and $c \leq 0$,} 
 \vspace{0.03in}
 \\
 \pmb{\big[}c+2 \,,\, g + \floor{c/2} + 1\pmb{\big]} & \text{if $g \geq 1$ and $c \leq 0$,} 
 \vspace{0.03in}
 \\
 \pmb{\big[}g + 1\,,\, g + \ceil{c/2} \pmb{\big]} & \text{if $\max\{1,g\} \leq \ceil{c/2}$,} 
 \vspace{0.03in} 
 \\
 \pmb{\big[}c+2\,,\, g + \floor{c/2} + 1\pmb{\big]} \cup \pmb{\big[}g + 1\,,\, g + \ceil{c/2} \pmb{\big]} & \text{if $g > \ceil{c/2} \geq 1$,}
\end{cases}
\end{align*}
where we have used the interval notation $\pmb{\big[}a,b\pmb{\big]}$ for the set of integers $m$ with $a \leq m \leq b$. Thus
\begin{align*}
 \max\!\big\{\!\min\{c+j, 2g-j+2\} : 2\leq j \leq g + 1\big\} = 
 \begin{cases}
 -\infty & \text{if $g \leq 0$ and $c \leq 0$,} \\
 g + \ceil{c/2} & \text{if $\max\{1,g\} \leq \ceil{c/2}$,} \\
 g + \floor{c/2} + 1  & \text{if $g > \ceil{c/2}$.}
\end{cases}
\end{align*}

Recall that $\max\{h^{0}(V), h^{1}(V) + 1\} \leq c+1$. Applying $\max\{c+1,-\}$ to each row above yields 
\begin{align*}
 \reg(V) \leq  
 \begin{cases}
 c+1 & \text{if $g \leq \ceil{c/2}$,} \\
 g + \floor{c/2} + 1  & \text{if $g > \ceil{c/2}.$}
\end{cases}
\end{align*}
We finally cover the edge cases separately. If $c=-1$, then $V$ is $\cofi{0}$-acyclic and $\reg(V) = -2$. If $g=-1$, then $\weak(V) = -1$ so $V$ is torsion and $\locoh{0}(V) = V$ and $\locoh{j}(V) = 0$ for $j \geq 1$, hence $\reg(V) \leq h^{0}(V) \leq c$ by \cite[Corollary 4.15]{ramos-coh}.
\end{proof}

\begin{cor} \label{t0-t1-ranges}
 Let $V$ be an $\FI$-module such that the triple $(V,c,g)$ satisfies Hypothesis \ref{FI-hyp}. Then
\begin{align*}
 \tgen(V)\! \leq 
 \begin{cases}
 g & \!\text{if $c=-1$,} \\
 c & \!\text{if $g=-1$ and $c \geq 0$,} \\
 c+1 & \!\text{if $0 \leq g \leq \ceil{c/2}$ and $c \geq 0$,} \\
 g + \floor{c/2} + 1  & \!\text{if $g > \ceil{c/2}$ and $c \geq 0$,}
\end{cases} 
\end{align*}
and
\begin{align*}
 \trel(V)\! \leq  
 \begin{cases}
 -1 & \!\text{if $c = -1$,} \\
 c+1 & \!\text{if $g = -1$ and $c \geq 0$,} \\
 c+2 & \!\text{if $0 \leq g \leq \ceil{c/2}$ and $c \geq 0$,} \\
 g + \floor{c/2} + 2  & \!\text{if $g > \ceil{c/2}$ and $c \geq 0$.}
\end{cases}
\end{align*}
\end{cor}
\begin{proof}
As $\trel(V) - 1 \leq \reg(V)$ by the definition of regularity, the bounds for $\trel(V)$ follow immediately from Theorem \ref{regularity-chad-bound}. And the bounds for $\tgen(V)$ follow from Corollary \ref{t0-vs-t1} since $\weak(V) \leq g$ by Theorem \ref{pass-to-cmnr}.  
\end{proof}

\subsection{The witnesses $\mathbf{I}(g)$, $\mathbf{T}(c)$, $\mathbf{S}(c)$, $\mathbf{V}(g)$} \label{section-witness}
In this section we construct the $\FI$-modules that witness the sharpness statements in Theorem \ref{regularity-from-FI-hyp} and Theorem \ref{main-ranges}.

\paragraph{The $\FI$-module $\mathbf{T}(c)$.} Given $c \geq -1$, we define 
\begin{align*}
 \mathbf{T}(c)_{A} := 
\begin{cases}
 \qq & \text{if $|A| = c$,} \\
 0 & \text{otherwise,}
\end{cases}
\end{align*}
and if $\iota \colon A \emb B$ is an injection of finite sets, we set
\begin{align*}
 \mathbf{T}(c)_{\iota} := 
\begin{cases}
 \id_{\qq} & \text{if $|A| = |B| = c$,} \\
 0 & \text{otherwise.}
\end{cases}
\end{align*}


\paragraph{The $\FI$-module $\mathbf{I}(g)$.} 
 Given $g \geq -1$, we set $\mathbf{I}(g) := \induce(\Res_{\FB}^{\FI}(\mathbf{T}(g)))$.
\begin{prop}\label{prop-Ig}
 For every $g \geq -1$, the triple $(\mathbf{I}(g),\, -1,\, g)$ satisfies Hypothesis \ref{FI-hyp} and $t_{i}(\mathbf{I}(g)) = -1$ for every $i \geq 1$.
\end{prop}
\begin{proof}
 This is immediate by Theorem \ref{characterize-H0-acyclic}.
\end{proof}

We shall find the following useful to pin down the $t_{i}$ values of the other witnessing $\FI$-modules.

\begin{thm}[{\cite[Theorem 21, Corollary 22]{gan-shift-seq}}] \label{gan-chain}
 Let $V$ be an $\FI$-module which is not $\cofi{0}$-acyclic. Then there is a chain
\begin{align*}
 0 \leq t_{1}(V) - 1 \leq t_{2}(V) - 2 \leq \cdots \leq t_{i}(V) - i \leq \cdots
\end{align*}
which stabilizes at $\reg(V)$.
\end{thm}

\begin{prop} \label{prop-Tc}
 For every $c \geq 0$, the triple $(\mathbf{T}(c),\, c,\, -1)$ satisfies Hypothesis \ref{FI-hyp} and $t_{i}(\mathbf{T}(c)) = i+c$ for every $i \geq 0$.
\end{prop}
\begin{proof}
 The $\FI$-module $\mathbf{T}(c)$ is finitely generated torsion of degree $c$, hence $\reg(\mathbf{T}(c)) = c$ by \cite[Lemma 4.2]{nss-regularity}. We also see that $\weak(\mathbf{T}(c)) = -1$ and $\tgen(\mathbf{T}(c)) = c$. By Corollary \ref{t0-vs-t1}, we get
\begin{align*}
 c = \tgen(\mathbf{T}(c)) \leq \max\{-1,\trel(\mathbf{T}(c)) - 1\} = \trel(\mathbf{T}(c)) - 1 \leq c
\end{align*}
and hence $\trel(\mathbf{T}(c)) - 1 = c$ and we conclude by Theorem \ref{gan-chain}.
\end{proof}

\paragraph{The $\FI$-module $\mathbf{S}(c)$.} Given $c \geq -1$, we define 
\begin{align*}
 \mathbf{S}(c)_{A} := 
\begin{cases}
 0 & \text{if $|A| \leq c$,} \\
 \qq & \text{otherwise,}
\end{cases}
\end{align*}
and if $\iota \colon A \emb B$ is an injection of finite sets, we set
\begin{align*}
 \mathbf{S}(c)_{\iota} := 
\begin{cases}
 0 & \text{if $|A| \leq c$,} \\
 \id_{\qq} & \text{otherwise.}
\end{cases}
\end{align*}

\begin{prop} \label{prop-Sc}
 For every $c \geq 0$, the triple $(\mathbf{S}(c),\, c,\, 0)$ satisfies Hypothesis \ref{FI-hyp} and $t_{i}(\mathbf{S}(c)) = i+c+1$ for every $i \geq 0$.
\end{prop}
\begin{proof}
 It is evident that $\tgen(\mathbf{S}(c)) = c+1$, $\shift{\mathbf{S}(c)}{} = \mathbf{S}(c-1)$ and $\deriv(\mathbf{S}(c))$ is torsion of degree $c$. Hence by definition, $\weak(\mathbf{S}(c)) = 0$ and since $\shift{\mathbf{S}(c)}{c+1} = \mathbf{S}(-1)$ is the constant $\FI$-module at $\qq$, the first claim follows by Theorem \ref{pass-to-cmnr}. Hence $\reg(\mathbf{S}(c)) \leq c+1$ by Theorem \ref{regularity-chad-bound}. Moreover by Corollary \ref{t0-vs-t1}, we have 
\begin{align*}
 c+1 = \tgen(\mathbf{S}(c)) \leq \max\{\trel(\mathbf{S}(c))-1,0\} = \trel(\mathbf{S}(c))-1 \leq c+1
\end{align*}
because $\mathbf{S}(c)$ is not $\cofi{0}$-acyclic. Thus $\trel(\mathbf{S}(c)) - 1 = c+1$ and we conclude by Theorem \ref{gan-chain}.
\end{proof}

\paragraph{The $\FI$-module $\mathbf{V}(g)$.} If $g=-1,0$, we set $\mathbf{V}(g)$ to be 0. For $g \geq 1$, following \cite[Proposition 3.4.1]{cef} with $\lambda := (g)$, we define $\mathbf{V}(g)$ to be the $\FI$-submodule of $\mathbf{I}(g)$ generated by the unique copy of the Specht module $\specht{\qq}{g,g} \leq \mathbf{I}(g)_{2g}$ in degree $2g$. It satisfies 
\begin{align*}
 \mathbf{V}(g)_{n} \cong 
\begin{cases}
 \specht{\qq}{n-g,g} & \text{if $n \geq 2g$,} 
 \\
 0 & \text{otherwise.}
\end{cases}
\end{align*}
\begin{prop} \label{prop-Vg}
 For every $g \geq 1$, the triple $(\mathbf{V}(g),\, 2g-2,\, g)$ satisfies Hypothesis \ref{FI-hyp} and $t_{i}(\mathbf{V}(g)) = i+2g$ for every $i \geq 0$.
\end{prop}
\begin{proof}
 Writing $\mathbf{U}(g) := \mathbf{I}(g) / \mathbf{V}(g)$, we claim that $\mathbf{U}(g)$ is torsion-free: for the sake of contradiction suppose $n:=h^{0}(\mathbf{U}(g)) \geq 0$. Then there is a partition $\mu \vdash n$ such that the Specht module $\specht{\qq}{\mu}$ 
\begin{itemize}
 \item appears in $\mathbf{I}(g)_{n} \cong \Ind_{\sym{g} \times \sym{n-g}}^{\sym{n}} \qq$ but not in $\mathbf{V}(g)_{n}$ (because the multiplicity of $\specht{\qq}{n-g,g}$ in $\mathbf{I}(g)_{n}$ is 1 by Pieri's rule), and 
 \item is sent inside $\mathbf{V}(g)_{n+1}$ under the transition map $\mathbf{I}(g)_{n} \emb \mathbf{I}(g)_{n+1}$.
\end{itemize}
The first criterion forces $\mu$ to have $<g$ boxes below the first row, as observed in \cite[proof of Proposition 3.4.1]{cef}. There is a commutative diagram 
\begin{align*}
 \xymatrix{
 \specht{\qq}{\mu} \ar[r] \ar[d] & \Res_{\sym{n}}^{\sym{n+1}}\mathbf{V}(g)_{n+1} \ar[d]
 \\
 \mathbf{I}(g)_{n} \ar[r] & \Res_{\sym{n}}^{\sym{n+1}}\mathbf{I}(g)_{n+1}
 }
\end{align*}
with injective $\sym{n}$-equivariant maps, in particular the top right does not vanish: so $n+1 \geq 2g$, we have $\mathbf{V}(g)_{n+1} \cong \specht{\qq}{n+1-g,g}$ and hence $\mu = (n+1-g,g-1)$. By Frobenius reciprocity, we get a commutative diagram
\begin{align*}
 \xymatrix{
 \Ind_{\sym{n}}^{\sym{n+1}}\specht{\qq}{n+1-g,g-1} \ar[r] \ar[d] 
 & \specht{\qq}{n+1-g,g} \ar[d]
 \\
 \Ind_{\sym{n}}^{\sym{n+1}}\mathbf{I}(g)_{n} \ar[r] & \mathbf{I}(g)_{n+1}
 }
\end{align*}
of $\qq\sym{n+1}$-modules. On the other hand, our $n \geq 2g-1 \geq g$ is in the \textbf{monotonicity} range \cite[Definition 1.2]{church-config} of the sequence $\{\mathbf{I}(g)_{n}\}$ by \cite[Theorem 2.8]{church-config}, which  implies that the image of the top map of the above diagram should contain a copy of $\specht{\qq}{n+2-g,g-1}$, a contradiction. Now applying $\locoh{0}$ to the short exact sequence 
\begin{align*}
 0 \rarr \mathbf{V}(g) \rarr \mathbf{I}(g) \rarr \mathbf{U}(g) \rarr 0 \, ,
\end{align*}
the associated long exact sequence gives $\locoh{0}(\mathbf{V}(g)) = 0$ and 
\begin{align*}
 \locoh{1}(\mathbf{V}(g)) = \locoh{0}(\mathbf{U}(g)) = 0 \, .
\end{align*}
Because $\weak(\mathbf{V}(g)) \leq \weak(\mathbf{I}(g)) = g$ by \cite[Proposition 2.9]{cmnr-range}, we get 
\begin{align*}
 h^{max}(\mathbf{V}(g)) = \max\{h^{j}(\mathbf{V}(g)): j \geq 2\} \leq 2g-2
\end{align*}
by Theorem \ref{structure-complex}. Consequently the triple $(\mathbf{V}(g),\,2g-2,\,g)$ satisfies Hypothesis \ref{FI-hyp} by Theorem \ref{pass-to-cmnr}.

On the other hand, by the hook length formula \cite[Theorem 20.1]{james-sym-book}, for $n \geq 2g$ we have 
\begin{align*}
 \dim_{\qq} \mathbf{V}(g)_{n} 
 &= \dim_{\qq} \specht{\qq}{n-g,g}
 \\
 &= \frac{n!}{g!(n-2g)! \cdot \prod_{i=0}^{g-1}(n-g+1-i)}
 = \frac{\prod_{j=0}^{2g-1}(n-j)}{g! \cdot \prod_{j=g-1}^{2g-2}(n-j)}
 \\
 &= \frac{n-(2g-1)}{g!}\prod_{j=0}^{g-2}(n-j) = \frac{n-(2g-1)}{g} \binom{n}{g-1}\, ,
\end{align*}
but this polynomial evaluates to a negative number when evaluated at $n = 2g-2$. Thus \cite[Proposition 2.14]{cmnr-range} forces $h^{max}(\mathbf{V}(g)) \geq 2g - 2$, and hence 
\begin{align*}
 \local(\mathbf{V}(g)) = 2g - 2 \, .
\end{align*}
Invoking Theorem \ref{structure-complex} again and using \cite[Theorem 1.1]{nss-regularity}, we get 
\begin{align*}
 \reg(\mathbf{V}(g)) &= \max\{h^{j}(\mathbf{V}(g)) + j : \locoh{j}(\mathbf{V}(g)) \neq 0\}
 \\
  &= h^{2}(\mathbf{V}(g)) + 2 
  \\
  &= \local(\mathbf{V}(g)) + 2 = 2g \, .
\end{align*}
Finally, by Corollary \ref{t0-vs-t1} we have 
\begin{align*}
 2g = \tgen(\mathbf{V}(g)) \leq \max\{\trel(\mathbf{V}(g))-1,g\} = \trel(\mathbf{V}(g))-1 \leq \reg(\mathbf{V}(g)) = 2g
\end{align*}
and hence $\trel(\mathbf{V}(g)) - 1 = 2g$ and by Theorem \ref{gan-chain} $t_{i}(\mathbf{V}(g)) - i = 2g$ for every $i \geq 2$ as well.
\end{proof}


\begin{proof}[Proof of \textbf{\emph{Theorem \ref{regularity-from-FI-hyp}}}]
Combine Theorem \ref{regularity-chad-bound} and Propositions \ref{prop-Ig},  \ref{prop-Tc}, \ref{prop-Sc}, \ref{prop-Vg}.
\end{proof}

\subsection{Complexes of $\FI$-modules and $\FI$-hyperhomology}
The main result of this section is Theorem \ref{hyper-to-usual}. We closely follow the treatment of Gan--Li \cite{gan-li-linear}.

\begin{prop} \label{gecerken}
 Suppose $\chain_{\star}$ is a chain complex of $\FI$-modules such that for every $k \in \zz$, the $\FI$-module $\chain_{k}$ is $\co_{0}^{\FI}$-acyclic and is generated in degrees $\leq g_{k}$. Then the $\FI$-module $\co_{k}(\chain_{\star})$ is presented in finite degrees, and assuming $g_{k} \geq 0$, it satisfies $\weak(\co_{k}(\chain_{\star})) \leq g_{k}$ and
\begin{align*}
 h^{j}(\co_{k}(\chain_{\star})) &\leq 
\begin{cases}
 \max\{-1,\,2g_{k+1}-2\} & \text{if $j=0$,} \\
 \max\{-1,\,2g_{k+1}-4\} & \text{if $j=1$,} \\ 
 \max\{-1,\,2g_{k}-2j+2\} & \text{if $j \geq 2$.}
\end{cases}
\end{align*}
\end{prop}
\begin{proof}
The $\FI$-module
\begin{align*}
 \co_{k}(\chain_{\star}) = \coker (\chain_{k+1} \rarr \ker (\chain_{k} \rarr \chain_{k-1}))
\end{align*}
is presented in finite degrees by \cite[Theorem 2.3, part (1)]{cmnr-range}. The first claim follows from \cite[Proposition 3.1, part (1)]{cmnr-range} and \cite[Proposition 3.3, part (1) and (2)]{cmnr-range}. Next, as $\weak(\chain_{k}) \leq g_{k}$ by Theorem \ref{t0-vs-t1} and 
\begin{align*}
 \ker (\chain_{k+1} \rarr \ker (\chain_{k} \rarr \chain_{k-1})) = 
 \ker (\chain_{k+1} \rarr \chain_{k}) \, ,
\end{align*}
by Theorem \ref{structure-complex} and \cite[proof of Proposition 3.3]{cmnr-range}, we have 
\begin{align*}
  h^{j}(\co_{k}(\chain_{\star})) 
  &\leq \max\{h^{j}(\ker (\chain_{k} \rarr \chain_{k-1})),\,
 h^{j+1}(\chain_{k+1}),\, 
 h^{j+2}(\ker (\chain_{k+1} \rarr \chain_{k}))\}
 \\
 &\leq 
\begin{cases}
 \max\{h^{0}(\chain_{k}),\,h^{1}(\chain_{k+1}),\,2g_{k+1}-2\}& \text{if $j = 0$,}
 \\
 \max\{h^{1}(\chain_{k}),\,h^{0}(\chain_{k-1}),\,
 h^{2}(\chain_{k+1}),\,2g_{k+1}-4\} 
 & \text{if $j=1$,}
 \\
 \{-1,\,2g_{k} - 2j + 2\} & \text{if $j \geq 2$.}
\end{cases}
\end{align*}
We are done because $\local(\chain_{k}) = -1$ for every $k \in \zz$ by Theorem \ref{characterize-H0-acyclic}. 
\end{proof}

\begin{thm} \label{hyper-to-usual}
 Let $\chain_{\star}$ be a chain complex of $\FI$-modules supported on non-negative degrees and $k \geq 0$ be a homological degree such that $\tbol_{k}(\chain_{\star}) , \tbol_{k+1}(\chain_{\star}) < \infty$. The $\FI$-module $\co_{k}(\chain_{\star})$ is identically zero if $\tbol_{k}(\chain_{\star}) = -1$, and otherwise setting 
\begin{align*}
 N_{k} := \max\{\tbol_{k}(\chain_{\star}), \tbol_{k+1}(\chain_{\star})\} \, ,
\end{align*}
 it satisfies the following:
\begin{birki}
 \item $\weak(\co_{k}(\chain_{\star})) \leq \tbol_{k}(\chain_{\star})$.
 \vspace{0.1cm}
 \item $\tgen(\co_{k}(\chain_{\star})) \leq 2\tbol_{k}(\chain_{\star})$.
 \vspace{0.1cm}
 \item $
  h^{j}(\co_{k}(\chain_{\star})) \leq 
\begin{cases}
 \max\{-1,\,2N_{k}-2\} & \text{if $j=0$,} \\
 \max\{-1,\,2N_{k}-4\} & \text{if $j=1$,} \\ 
 \max\{-1,\,2\weak(\co_{k}(\chain_{\star}))-2j+2\} & \text{if $j \geq 2$.}
\end{cases}
$
\vspace{0.1cm}
 \item $\local(\co_{k}(\chain_{\star})) \leq \max\{-1,\,2N_{k}-2\}$.
\vspace{0.1cm}
 \item $\reg(\co_{k}(\chain_{\star})) \leq 
\begin{cases}
 2N_{k}-2 & \text{if $\tbol_{k}(\chain_{\star}) = 0$,}
 \\
 \max\{2\weak(\co_{k}(\chain_{\star})),\,2N_{k}-2\}
 & \text{otherwise.}
\end{cases}
 $
\end{birki}
\begin{proof}
Let $P_{\star}$ be the total complex of a Cartan--Eilenberg resolution of $\chain_{\star}$. In particular $P_{\star}$ is a chain complex of projective $\FI$-modules with $\co_{k}(P_{\star}) = \co_{k}(\chain_{\star})$. We first note that projective $\FI$-modules are induced, that is, for each $j \geq 0$ we have $P_{j} = \induce(W^{j})$ for some $\FB$-module $W^{j}$; and hence 
\begin{align*}
 (P_{j})_{\gen{\leq N}} \cong \induce\! \left( \bigoplus_{n \leq N} W^{j}_{n}\, \right)
\end{align*}
is $\cofi{0}$-acyclic for any $N$, being a summand of the projective $\FI$-module $P_{j}$. It is shown in \cite[Lemma 7]{gan-li-linear} that the $k$-th homology of the subcomplex 
\begin{align*}
 \mathlarger{
 Q_{\star} :=
 (P_{\star})_{\gen{\leq{\tbol_{k}(\chain_{\star})}}} \subseteq P_{\star}
 }
\end{align*}
surjects on $\co_{k}(P_{\star})$. Thus we may assume $\tbol_{k}(\chain_{\star}) \geq 0$ in the rest of the proof. By Corollary \ref{H0-acyclic-complex-range} applied to the subcomplex $Q_{\star}$ we get 
\begin{align*}
 \tgen(\co_{k}(P_{\star})) \leq \tgen(\co_{k}(Q_{\star})) \leq 2\tbol_{k}(\chain_{\star}) \, ,
\end{align*}
because each term of $Q_{\star}$ is \emph{a fortiori} generated in degrees $\leq \tbol_{k}(\chain_{\star})$, establishing (2). (1) follows from \cite[Theorem 5.1, part (1)]{cmnr-range} or \cite[Remark 9]{gan-li-linear}.

It is shown in the proof of \cite[Lemma 8]{gan-li-linear} that the natural map
\begin{align*}
 \mathlarger{
 \co_{k} \left( (P_{\star})_{\gen{\leq N_{k}}} \right) \rarr \co_{k}(P_{\star}) = \co_{k}(\chain_{\star}) \, ,
 }
\end{align*}
is an isomorphism; hence Proposition \ref{gecerken} applies to the complex $(P_{\star})_{\gen{\leq N_{k}}}$ with $g_{k} = g_{k+1} = N_{k}$ and we get the bounds for $j=0,1$ in (3). The bounds for $j \geq 2$ follow from Theorem \ref{structure-complex}. (4) follows immediately from (1) and (3). For (5) we consider two cases: 
\begin{itemize}
 \item $\tbol_{k}(\chain_{\star}) = 0$: here by \cite[Corollary 4.15]{ramos-coh}, (1) and (3), we have
\begin{align*}
 \reg(\co_{k}(\chain_{\star})) 
 &\leq \max\{h^{0}(\co_{k}(\chain_{\star})),\, h^{1}(\co_{k}(\chain_{\star})) + 1\}
 \\
 &\leq \max\{-1,\,2N_{k}-2\} 
 = \max\{-1,\,2\tbol_{k+1}(\chain_{\star})-2\}
 \\
 &\leq 
\begin{cases}
  -1 & \text{if $2\tbol_{k+1}(\chain_{\star})-2 \leq -1$,}
  \\
  2\tbol_{k+1}(\chain_{\star})-2 & \text{otherwise.}
\end{cases}
\end{align*} 
But the inequality $2\tbol_{k+1}(\chain_{\star})-2 \leq -1$ is equivalent to $2\tbol_{k+1}(\chain_{\star})-2 \leq -2$, and the inequality $\reg(\co_{k}(\chain_{\star})) \leq -1$ implies $\trel(\co_{k}(\chain_{\star})) \leq 0$, which forces that $\co_{k}(\chain_{\star})$ is $\cofi{0}$-acyclic by Corollary \ref{t1-not-0} and hence $\reg(\co_{k}(\chain_{\star})) = -2$. Thus 
\begin{align*}
 \reg(\co_{k}(\chain_{\star})) 
  &\leq 
\begin{cases}
  -2 & \text{if $2\tbol_{k+1}(\chain_{\star})-2 \leq -2$,}
  \\
  2\tbol_{k+1}(\chain_{\star})-2 & \text{otherwise,}
\end{cases}
\\
 &\leq \max\{-2,\,2\tbol_{k+1}(\chain_{\star})-2\} = 2N_{k}-2 \, .
\end{align*}
 \item $\tbol_{k}(\chain_{\star}) \geq 1$: here by \cite[Corollary 4.15]{ramos-coh}, (1) and (3), we have
\begin{align*}
 \reg(\co_{k}(\chain_{\star})) 
 &\leq \max\{h^{j}(\co_{k}(\chain_{\star})) + j : 0 \leq j \leq \weak(\co_{k}(\chain_{\star})) + 1\}
 \\
 &\leq \max\{2N_{k}-2,\,2\weak(\co_{k}(\chain_{\star}))\} \, .
\end{align*}
\end{itemize}
\end{proof}
  
\end{thm}
%
%
\section{Stable ranges}
In this section we recall/establish bounds for the several ways $\FI$-modules exhibit stable behavior and prove Theorem \ref{main-ranges} and Theorem \ref{hyper-ranges}.

\subsection{Inductive description}
Corollary \ref{cor-inductive} of this section together with Corollary \ref{t0-t1-ranges} forms the inductive description part of Theorem \ref{main-ranges}.

\begin{thm}[{\cite[Theorem C]{ce-homology},\cite[Proof of Theorem 2.6]{gan-li-central}}] \label{induct-eqv}
 Let $V$ be an $\FI$-module defined over a commutative ring $R$, and $N \geq -1$. The following are equivalent: 
\begin{birki}
 \item $\max\{\tgen(V), \trel(V)\} \leq N$.
 \item For every $n \geq 0$, the natural map
$\mathlarger{\underset{\substack{S \subseteq \{1,\dots,n\}
 \vspace{0.011in} \\ |S| \leq N}}{\colim} \!\!V_{S} \rarr V_{n}}$
of $R \sym{n}$-modules is an isomorphism.
\item Writing $\FI_{\leq N}$ for the full subcategory of $\FI$ whose objects are sets of size $\leq N$ and $\Ind_{\FI_{\leq N}}^{\FI}$ for the left adjoint of the restriction 
$\Res_{\FI_{\leq N}}^{\FI} \colon \lMod{\FI} \rarr \lMod{\FI_{\leq N}}$, the counit map $\Ind_{\FI_{\leq N}}^{\FI}\!\Res_{\FI_{\leq N}}^{\FI} V \rarr V$ is an isomorphism.
\end{birki}
\end{thm}

\begin{cor} \label{cor-inductive}
 Let $V$ be an $\FI$-module defined over a commutative ring $R$. Given $n,N \in \nn$, the natural map
\begin{align*}
 \mathlarger{\underset{\substack{S \subseteq \{1,\dots,n\}\vspace{0.011in} \\ |S| \leq N}}{\colim} \!\!V_{S} \rarr V_{n}}
\end{align*}
of $R \sym{n}$-modules is surjective if $N \geq \tgen(V)$, and injective if $N \geq \trel(V)$. 
\end{cor}
\begin{proof}
%
That $N \geq \tgen(V)$ implies surjectivity follows immediately from the definition of $\cofi{0}$. If $N \geq \trel(V)$, the long exact homology sequence associated to the short exact sequence in Proposition \ref{separate-t1-gen} gives
\begin{align*}
 \tgen(V_{\gen{\leq N}}) \leq N \quad \text{and} \quad \trel(V_{\gen{\leq N}}) = \trel(V) \leq N \, .
\end{align*}
Considering the factorization
\begin{align*}
\underset{\substack{S \subseteq \{1,\dots,n\}\vspace{0.011in} \\ |S| \leq N}}{\colim} \!\!V_{S} \,\,= 
\underset{\substack{S \subseteq \{1,\dots,n\}\vspace{0.011in} \\ |S| \leq N}}{\colim} \!(V_{\gen{\leq N}})_{S} \rarr (V_{\gen{\leq N}})_{n} \emb V_{n}\, ,
\end{align*}
the map associated to $V$ is injective if and only if the one for $V_{\gen{\leq N}}$ is. Indeed the latter is an isomorphism by Theorem \ref{induct-eqv}.
\end{proof}

\subsection{Additive structure}
In this section we prove Theorem \ref{thm-additive}, which forms the additive structure part of Theorem \ref{main-ranges}. We also observe that our stable range here is at least as good as that of Patzt--Wiltshire-Gordon \cite{patzt-wgordon}.

\begin{thm} \label{thm-additive}
 Let $V$ be an $\FI$-module defined over a commutative ring $R$ such that the triple $(V,c,g)$ satisfies Hypothesis \ref{FI-hyp}. Then there exist $R$-modules $\ab_{0}, \dots, \ab_{g}$ such that in the range $n \geq \max\{c+1,2g-1\}$, there is an isomorphism
 \vspace{-0.05in}
\begin{align*}
 \mathlarger{V_{n} \cong \bigoplus_{r=0}^{g} \ab_{r}^{\oplus \binom{n}{r} - \binom{n}{r - 1}}}
\end{align*}
of $R$-modules, with the convention that $\binom{a}{b} = 0$ unless $0 \leq b \leq a$.
\vspace{0.05in}
\end{thm}
\begin{proof}
 For simplicity we only treat the case $R = \zz$. The general case is entirely analogous. We will establish the isomorphism in the range 
\begin{align*}
 n \geq \max\{h^{0}(V) + 1,\, h^{1}(V) + 1,\, 2g-1\}\,,
\end{align*}
which implies the desired range by Theorem \ref{pass-to-cmnr}. First of all if $g = -1$, then $V$ is torsion with $\deg(V) = c$, hence in the range $n \geq c+1$ we have $V_{n} = 0$, agreeing with the empty direct sum. From now on we assume $g \geq 0$.

There is a functor $\CB \colon \FI^{\opp} \times \FI \rarr \mathsf{Set}$ defined in
\cite[Definition 4.1]{patzt-wgordon} with the following properties:
\begin{birki}
 \item There is a map $\chi \colon \zz\!\Hom_{\FI}(-,-) \rarr \zz\!\CB$ of $\FI$-bimodules such that 
\begin{align*}
 \chi_{k,n} \colon \zz\!\Hom_{\FI}(k,n) \rarr \zz\!\CB(k,n)
\end{align*}
is an isomorphism when $n \geq 2k-1$ \cite[Proposition 4.7, Corollary 4.11]{patzt-wgordon}. 

\item The functor $\zz\!\CB \otimes_{\FI} \,\,- \colon \lMod{\FI} \rarr \lMod{\FI}$ is exact. This is because the $\FI$-module structure on the tensor product is defined pointwise from that of $\zz\!\CB$, and for every $n$ the $\ters{\FI}$-module $\zz\!\CB(-,n)$ is a direct sum of $\Xi(\ell)$'s defined in \cite[Definition 1.1]{patzt-wgordon} (see \cite[proof of Lemma 4.2]{patzt-wgordon}), each of which is flat by \cite[Proposition 3.24]{patzt-wgordon}. 

\item The specific decomposition of $\zz\!\CB(-,n)$ mentioned above is 
\begin{align*}
 \zz\!\CB(-,n) 
 &\cong \bigoplus_{\ell \in \nn} \zz\!\CB_{\ell}(-,n)
 \cong \bigoplus_{\ell \in \nn} \bigoplus_{c \in \catalan(\ell,n)}
 \!\!\!\zz\!\CB_{\ell}^{c}(-,n)
 \\
 &\cong \bigoplus_{\ell \in \nn} \bigoplus_{c \in \catalan(\ell,n)}
 \!\!\!\Xi(\ell) 
 \cong \bigoplus_{\ell \leq n/2} \Xi(\ell)^{\oplus \catalan(\ell,n)}
 \\
 &\cong \bigoplus_{\ell \leq n/2} \Xi(\ell)^{\oplus \binom{n}{\ell} - \binom{n}{\ell-1}}
\end{align*}
by \cite[Definition 1.1]{patzt-wgordon} and \cite[proof of Proposition 4.4]{patzt-wgordon}.
\end{birki}

It can be inspected from the definition of tensor products of functors in \cite[Definition 2.1]{patzt-wgordon} that for any $\ters{\FI}$-module $Y$, the functor 
\begin{align*}
 Y \underset{\FI_{\leq g}}{\otimes} - : \lMod{\FI_{\leq g}} \rarr \lMod{\zz}
\end{align*}
is isomorphic to $Y \otimes_{\FI} -$ precomposed with the inclusion $\lMod{\FI_{\leq g}} \emb \lMod{\FI}$. Thus the exactness in (2) implies that 
\begin{align*}
 \zz\!\CB \underset{\FI_{\leq g}}{\otimes} - \colon \lMod{\FI_{\leq g}} \rarr \lMod{\FI}
\end{align*}
is also exact. For any $\FI$-module $U$, let us write
\begin{align*}
 \chi^{U} := \chi \underset{\FI_{\leq g}}{\otimes} \id_{U} 
 \colon 
 \underbrace{
 \zz\!\Hom_{\FI}(-,-) \underset{\FI_{\leq g}}{\otimes} \Res_{\FI_{\leq g}}^{\FI} U}_{\mathlarger{
 \cong\, \Ind_{\FI_{\leq g}}^{\FI}\!\Res_{\FI_{\leq g}}^{\FI} U
 }} 
 \rarr 
 \zz\!\CB \underset{\FI_{\leq g}}{\otimes} \Res_{\FI_{\leq g}}^{\FI} U 
\end{align*}
for the map $\chi$ induces, which is natural in $U$. We also write 
\begin{align*}
 \eps^{U} \colon \Ind_{\FI_{\leq g}}^{\FI}\!\Res_{\FI_{\leq g}}^{\FI} U \rarr U
\end{align*}
for the counit of the adjunction $\Ind_{\FI_{\leq g}}^{\FI} \dashv \Res_{\FI_{\leq g}}^{\FI}$.

By Theorem \ref{structure-complex}, there is a complex $0 \rarr V \rarr I \xrightarrow{\alpha} J$ which is exact in degrees 
\begin{align*}
 \geq \max\{h^{0}(V) + 1, h^{1}(V) + 1\} =: d \, 
\end{align*}
such that $(I,\, -1,\, g)$ and $(J,\, -1,\, g-1)$ satisfy Hypothesis \ref{FI-hyp}. By Theorem \ref{pass-to-cmnr} and Corollary \ref{t0-vs-t1}, both $I$ and $J$ are presented in degrees $\leq g$. We shall consider the following commutative diagram
\begin{align*}
\xymatrixcolsep{2pc}\xymatrixrowsep{3pc}
 \xymatrix{
 0 \ar[r] & \zz\!\CB \underset{\FI_{\leq g}}{\otimes} \Res_{\FI_{\leq g}}^{\FI} V
 \ar[r]  
 & \zz\!\CB \underset{\FI_{\leq g}}{\otimes} \Res_{\FI_{\leq g}}^{\FI} I 
 \ar[r]^-{\overline{\alpha}}  
 & \zz\!\CB \underset{\FI_{\leq g}}{\otimes} \Res_{\FI_{\leq g}}^{\FI} J 
 \\
 0 \ar[r] & \Ind_{\FI_{\leq g}}^{\FI}\!\Res_{\FI_{\leq g}}^{\FI} V 
 \ar[r] \ar[d]^-{\eps^{V}} \ar[u]_-{\chi^{V}}
 & \Ind_{\FI_{\leq g}}^{\FI}\!\Res_{\FI_{\leq g}}^{\FI} I 
 \ar[r] \ar[d]^-{\eps^{I}} \ar[u]_-{\chi^{I}}
 & \Ind_{\FI_{\leq g}}^{\FI}\!\Res_{\FI_{\leq g}}^{\FI} J
 \ar[d]^-{\eps^{J}} \ar[u]_-{\chi^{J}}
 \\
 0 \ar[r] & V \ar[r] & I \ar[r]^{\alpha} & J  
 }
\end{align*}
of $\FI$-modules. By \cite[proof of Theorem A in Section 4.3]{patzt-wgordon}, $\chi^{I}$ and $\chi^{J}$ are isomorphisms in degrees $\geq 2g-1$. By Theorem \ref{induct-eqv}, $\eps^{I}$ and $\eps^{J}$ are isomorphisms of $\FI$-modules. The top row is exact in degrees $\geq d$ as well because it is the image of the bottom row under an exact functor. Therefore for $n \geq \max\{d,2g-1\}$, using (3) we have 
\begin{align*}
 V_{n} 
 &\cong (\ker \alpha)_{n} \cong (\ker \ov{\alpha})_{n} 
 \cong \left( 
 \zz\!\CB \underset{\FI_{\leq g}}{\otimes} \Res_{\FI_{\leq g}}^{\FI} V
 \right)_{\!n} 
 \vspace{1in}
 \cong \zz\!\CB(-,n) \underset{\FI_{\leq g}}{\otimes} \Res_{\FI_{\leq g}}^{\FI} V
 \\
 &\cong \left(
 \bigoplus_{\ell \leq n/2} \Xi(\ell)^{\oplus \binom{n}{\ell} - \binom{n}{\ell-1}}
 \right)
 \underset{\FI_{\leq g}}{\otimes} \Res_{\FI_{\leq g}}^{\FI} V
 \cong \left(
 \bigoplus_{\ell = 0}^{g} \Xi(\ell)^{\oplus \binom{n}{\ell} - \binom{n}{\ell-1}}
 \right)
 \otimes_{\FI} V
 \\
 &\cong \bigoplus_{\ell=0}^{g} \left(
 \Xi(\ell) \otimes_{\FI} V
 \right)^{\oplus \binom{n}{\ell} - \binom{n}{\ell-1}} \,\,\, ,
\end{align*}
as desired.
\end{proof}

\begin{rem}
 We have established in Theorem \ref{thm-additive} that for an $\FI$-module $V$ presented in finite degrees, the additive structure isomorphism holds in the range 
\begin{align*}
 n \geq \max\{\local(V) + 1, 2\weak(V)-1\} \, .
\end{align*}
 Let us observe that this range is at least as good as the range 
\begin{align*}
 n \geq 2\max\{\tgen(V), \trel(V)\} - 1
\end{align*}
 established in \cite[Theorem A]{patzt-wgordon}: indeed the ranges agree when $V$ is $\cofi{0}$-acyclic by Theorem \ref{characterize-H0-acyclic} and Corollary \ref{t0-vs-t1}. And when $V$ is not $\cofi{0}$-acyclic, using Theorem \ref{pass-to-cmnr},
\begin{align*}
 \max\{h^{j}(V) : j \geq 2\} &\leq 2\weak(V) - 2  \text{ by Theorem \ref{structure-complex},}
 \\
 \weak(V) &\leq \tgen(V) \, \text{ by  \cite[Proposition 2.9, part (4)]{cmnr-range},}
 \\
 \max\{h^{0}(V), h^{1}(V)\} \leq \reg(V) &\leq 2\trel(V)-2 \, \text{ 
\begin{tabular}{l}
 by \cite[Theorem 1.1]{nss-regularity} \\
 and Theorem \ref{improved-regularity}.
\end{tabular}
 } 
\end{align*}
Thus we always have
\begin{align*}
 \max\{\local(V)+1,\, 2\weak(V)-1\} \leq \max\{2\trel(V)-1,\, 2\tgen(V)-1\} \, .
\end{align*}
In fact with a bit more work, it can be checked that this inequality is strict unless $\tgen(V) \geq \trel(V)$ or $\reg(V) = h^{0}(V) = 2\tgen(V) = 2\trel(V) - 2$.
\end{rem}

%
%

\subsection{Polynomiality and Specht stability}
In this section we prove Theorem \ref{thm-poly-specht}, which forms the polynomiality and Specht stability parts of Theorem \ref{main-ranges} for $\FI$-modules defined over a field $\kk$. 

\begin{thm} \label{thm-poly-specht}
  Let $V$ be an $\FI$-module defined over a field $\kk$ with $\dim_{\kk}V_{n} < \infty$ for every $n$. If the triple $(V,c,g)$ satisfies Hypothesis \ref{FI-hyp}, then the following hold:
\begin{birki}
 \item For each $r=0,\dots,g$ there is a finite-dimensional $\kk\!\sym{r}$-module $W_{r}$ such that in the range $n \geq c+1$, the sequence of symmetric group (Brauer) characters
 \vspace{-0.05in}
\begin{align*}
\mathlarger{n \mapsto \mathlarger{\chi}_{V_{n}}}
\end{align*}
is equal to the $\kk$-character polynomial 
\begin{align*}
\mathlarger{
 \sum_{r=0}^{g} \sum_{\lambda \,\vdash\, r}
 \chi_{W_{r}}(\lambda)
 \binom{\cyc_{1} - c - 1}{a_{1}(\lambda)} 
 \prod_{j=2}^{r}\!\! \binom{\cyc_{j}}{a_{j}(\lambda)} \, .
 }
\end{align*}
Here $a_{j}(\lambda)$ denotes the number of parts of size $j$ that $\lambda$ has, and if $a_{j}(\lambda) > 0$ for some $j$ divisible by $\ch(\kk)$ we write $\chi_{W_{r}}(\lambda) = 0$.
In particular, there exist integers $d_{0}, \dots, d_{g} \geq 0$ (namely $d_{r} = \dim_{\kk} W_{r}$) such that in the range $n \geq c + 1$, we have
\vspace{-0.03in}
\begin{align*}
\mathlarger{
\dim_{\kk} V_{n} = \sum_{r=0}^{g} d_{r} \binom{n - c - 1}{r} \,.
}
\end{align*}
\vspace{0.03 in}
\item There is a uniquely determined function
\begin{align*}
  \nvm \colon 
\{\text{$\kk$-regular partitions of size $\leq g$}\} \rarr \zz \, ,
\end{align*}
such that for every $n \geq \max\{c+1,2g\}$,
we have 
\begin{align*}
 \mathlarger{[V_{n}] = \sum_{\lambda} \nvm(\lambda) \left[ \specht{\kk}{\lambda[n]} \right]}
\end{align*}

\vspace{-0.05in}
in the Grothendieck group $K_{0}(\kk\! \sym{n})$ of finite-dimensional $\kk\!\sym{n}$-modules.
\end{birki}
\end{thm}

\begin{proof}
We may choose the complex $I^{\star}$ in Theorem \ref{structure-complex} so that
\begin{align*}
 \dim I^{j}_{n} < \infty
\end{align*}
for every $1 \leq j \leq g+1$ and $n \geq 0$. As $I^{\star}$ is exact in degrees $\geq c+1$, we have
\begin{align*}
 \mathlarger{[V_{n}] = \sum_{j=1}^{g+1} (-1)^{j-1} [I^{j}_{n}] \in K_{0}(\kk\! \sym{n})}
\end{align*}
whenever $n \geq c+1$, and moreover each $I^{j}$ for $j \geq 1$ is $\cofi{0}$-acyclic. Because the statement (2) is determined at the Grothendieck group level, it is enough to prove it for $I^{j}$'s. To that end we may assume $c=-1$. Moreover by the filtration in Theorem \ref{characterize-H0-acyclic} we may also assume $V = \induce(W)$ for some $\FB$-module $W$. In fact we may assume $W$ is entirely concentrated on degree 
\begin{align*}
 d := \deg(W) = \weak(V) \leq g \, .
\end{align*}
We do so, and regard $W$ as a representation of the symmetric group $\sym{d}$. Recall that for every partition $\lambda$ of $d$, the Specht module $\specht{R}{\lambda}$ is defined over any ring $R$, so that $R \otimes \specht{\zz}{\lambda} = \specht{R}{\lambda}$. As $\kk$ is a field, the Specht modules span the Grothendieck group $K_{0}(\kk\!\sym{d})$ \cite[Corollary 12.2]{james-sym-book}, thus
\begin{align*}
 [W]  = \sum_{\lambda \vdash d} c_{\lambda} [\specht{\kk}{\lambda}] \in K_{0}(\kk\!\sym{d}) \, 
\end{align*}
for some $c_{\lambda} \in \zz$.  Writing
\begin{align*}
 \inspecht{R}{\lambda} := \induce(\specht{R}{\lambda}) 
 \,
\end{align*}
for every ring $R$ and partition $\lambda \vdash d$, because the functor $\induce$ is exact (see the identification \cite[Definition 2.2.2]{cef}), to prove (2) we may reduce to the case $V = \inspecht{\kk}{\lambda}$ for a fixed partition $\lambda \vdash d$. If $\kk$ has characteristic zero, this follows from  using the identification
\begin{align*}
 \inspecht{\kk}{\lambda}_{n} &\cong \Ind_{\sym{d} \times \sym{n-d}}^{\sym{n}} 
 \big(
 \specht{\kk}{\lambda} \boxtimes \kk 
 \big)
 \\
 &= \Ind_{\sym{d} \times \sym{n-d}}^{\sym{n}}
 \big(
 \specht{\kk}{\lambda} \boxtimes \specht{\kk}{n - |\lambda|} 
 \big)
\end{align*}
and applying \cite[Lemma 2.3]{hemmer-pieri} when $d \leq n-d$, or equivalently $n \geq 2d$. In positive characteristic by \cite[Corollary 14]{james-sym-book}, on the Grothendieck group level, Pieri's rule applies to the Specht modules the same way it does in characteristic zero; consequently we can again allude to \cite[Lemma 2.3]{hemmer-pieri}. The uniqueness of $\nvm$ follows from the following facts:
\begin{itemize}
 \item The set $\{[\specht{\kk}{\mu}] : \mu \vdash n \text{ is $\kk$-regular}\}$, being in a unitriangular correspondence with the set of simple modules, forms a basis of $K_{0}(\kk\! \sym{n})$.
 \item Once $n \geq 2g+1$, a partition $\lambda$ of size $\leq g$ being $\kk$-regular forces $\lambda[n]$ to be $\kk$-regular as well.
\end{itemize}

To prove (1), first of all we note that by Theorem \ref{pass-to-cmnr} and \cite[Proposition 2.9]{cmnr-range}, the $\FI$-module $\shift{V}{c+1}$ is $\cofi{0}$-acyclic and generated in degrees $\leq g$. Thus by Theorem \ref{characterize-H0-acyclic} there is a finite filtration
\begin{align*}
 0 = X^{(-1)} \leq X^{(0)} \leq \cdots \leq X^{(s)} = \shift{V}{c+1}
\end{align*}
of $\FI$-submodules such that for each $0 \leq i \leq s$ we have $X^{(i)}/X^{(i-1)} \cong \induce(L^{(i)})$ for some $\FB$-module $L^{(i)}$ with $\deg(L^{(i)}) \leq g$. As a result, for the character sequences we have (see \cite[Proposition 10.1.3 part (5)]{webb-rep-book} for Brauer characters)
\begin{align*}
\mathlarger{
 \chi_{\shift{V}{c+1}}
 } 
 &= 
 \mathlarger{
 \sum_{i=0}^{s} \chi_{\induce(L^{(i)})}
 }
 = 
 \mathlarger{
 \sum_{i=0}^{s} \sum_{r=0}^{g} \chi_{\induce\left(L^{(i)}_{r}\right)}
 }
 \\
 &=
 \mathlarger{
 \sum_{r=0}^{g} \sum_{i=0}^{s} \chi_{\induce\left(L^{(i)}_{r}\right)}
 }
 =
 \mathlarger{
 \sum_{r=0}^{g} \chi_{\induce\left(\bigoplus_{i=0}^{s}L^{(i)}_{r}\right)}
 } \, .
\end{align*}
Now for each $r=0,\dots,g$, consider the $\kk\!\sym{r}$-module $W_{r} := \bigoplus_{i=0}^{s}L^{(i)}_{r}$, which we shall also consider as an $\FB$-module concentrated in degree $r$. Now 
\begin{align*}
\mathlarger{
 \chi_{\shift{V}{c+1}} = \sum_{r=0}^{g} 
 \chi_{\induce\left(W_{r}\right)} \, .
 }
\end{align*}
 If $\ch(\kk) = 0$, \cite[end of the proof of Theorem 4.1.7]{cef} shows that for every $n \in \nn$ and $\sigma \in \sym{n}$ we have
\begin{align*}
 \mathlarger{
 \chi_{\induce(W_{r})_{n}}(\sigma) = 
 \sum_{\lambda \vdash r} \chi_{W_{r}}(\lambda)
 \cdot
 \prod_{j=1}^{r}\!\! \binom{\cyc_{j}}{a_{j}(\lambda)}
 (\sigma) \, .
 }
\end{align*}
If $\ch(\kk) = p > 0$, we shall follow the same argument with a bit more care. First we set the stage for Brauer characters: we can find \cite[Theorem 29.1]{matsumura} a characteristic zero field $\bb{K}$ with a discrete valuation whose valuation ring $\oo$, say with uniformizing parameter $\pi$, that has $\oo / \pi \oo = \kk$ as its residue field. In other words, $(\bb{K}, \oo, \kk)$ is a $p$-modular system \cite[Section 9.4]{webb-rep-book}. It is in fact a splitting $p$-modular system simply because every field is a splitting field for symmetric groups. Let us fix $r$ and write $W := W_{r}$ as an $\FB$-module concentrated in degree $r$. Then given a $p$-regular element $\sigma \in \sym{n}$ (that is, $p$ does not divide the order $|\sigma|$), we have
\begin{align*}
 \mathlarger{
 \chi_{\induce(W)_{n}}(\sigma)
 } 
 &= 
 \mathlarger{
 \left(
 \Ind_{\sym{r} \times \sym{n-r}}^{\sym{n}}
 \left(
 \chi_{W} \boxtimes \mathbbm{1}
 \right)
 \right)
 (\sigma)
 }
 =\!\!\!\!\!
 \sum_{\substack{\tau \,\in\, [\sym{n}/\sym{r} \times \sym{n-r}] \\ \tau^{-1}\sigma\tau \,\in\, \sym{r} \times \sym{n-r}}} (\chi_{W} \boxtimes \mathbbm{1})(\tau^{-1}\sigma\tau)
 \\
 &=
 \sum_{\substack{T \subseteq \{1,\dots,n\},\,|T|=r \\ \sigma(T)=T}} \chi_{W_{T}}(\sigma|_{T})
 =
 \sum_{\lambda \vdash r} \sum_{\substack{T \subseteq \{1,\dots,n\},\,|T|=r \\ \sigma(T)=T \\ \sigma|_{T} \text{ has cycle type $\lambda$}}} \chi_{W}(\lambda) \, ,
\end{align*}
which can be written as $\sum_{\lambda \vdash r} \mathbf{P}_{\lambda}(\sigma)\chi_{W}(\lambda)$, where
\begin{align*}
 \mathbf{P}_{\lambda}(\sigma) &:= \left|\left\{
 T \subseteq \{1,\dots, n\} : |T| = r,\,\sigma(T) = T,\,\text{ $\sigma|_{T}$ has cycle type $\lambda$}
 \right\}\right| \\
 &= \prod_{j=1}^{r}\binom{\cyc_{j}(\sigma)}{a_{j}(\lambda)} 
 = \prod_{j=1}^{r}\binom{\cyc_{j}}{a_{j}(\lambda)} (\sigma) \, .
\end{align*}
Here there is no harm in taking $\chi_{W}(\lambda) = 0$ if $\lambda$ has a part divisible by $p$ because in that case $\mathbf{P}_{\lambda}(\sigma) = 0$ as $\sigma$ is $p$-regular. 

Regardless of the characteristic of $\kk$, we have established that there exists an $\kk\!\sym{r}$-module $W_{r}$ for each $r=0, \dots, g$ such that the sequence of symmetric group (Brauer) characters of the representations
\begin{align*}
 \mathlarger{
 n \mapsto (\shift{V}{c+1})_{n} = \Res^{\sym{n+c+1}}_{\sym{n}}V_{n+c+1}
 }
\end{align*}
is equal to the $\kk$-character polynomial
\begin{align*}
 \mathlarger{ 
 \mathbf{P}(\cyc_{1},\dots,\cyc_{g}) :=
 \sum_{r=0}^{g}\sum_{\lambda \vdash r} \chi_{W_{r}}(\lambda)
 \cdot
 \prod_{j=1}^{r}\!\! \binom{\cyc_{j}}{a_{j}(\lambda)}
 } \, .
\end{align*}
The equation established in the very beginning of this proof yields
\begin{align*}
 \mathlarger{
 \chi_{V_{n}} = \sum_{j=1}^{g+1} (-1)^{j-1} \chi_{I^{j}_{n}}
 }
\end{align*}
in the range $n \geq c+1$. Here for each $j=1,\dots,g+1$ the $\FI$-module $I^{j}$ is $\cofi{0}$-acyclic and is generated in degrees $\leq g - j + 1$. Thus there exists an $\kk$-character polynomial $\mathbf{F}(\cyc_{1}, \dots, \cyc_{g})$ that is equal to the sequence of symmetric group characters 
\begin{align*}
 n \mapsto \chi_{V_{n}}
\end{align*}
in the range $n \geq c+1$. We will be done once we show 
\begin{align*}
 \mathbf{F}(\cyc_{1}, \cyc_{2}, \dots, \cyc_{g}) = \mathbf{P}(\cyc_{1}-c-1,\cyc_{2},\dots, \cyc_{g}) \, . \tag{$\star$}
\end{align*}
To that end, fix $n \geq c+1$ and let $\sigma \in \sym{n-c-1}$ whose order is not divisible by $\ch(\kk)$. Suppose $\sigma$ has $a_{j}(\sigma)$ many $j$-cycles in its cycle decomposition as an element of $\sym{n-c-1}$. Now if we consider $\sigma$ as an element of $\sym{n}$, its number of $1$-cycles becomes $a_{1}(\sigma) + c + 1$, while for $j \geq 2$ its number of $j$-cycles is still $a_{j}(\sigma)$. Therefore on one hand we have
\begin{align*}
 \chi_{V_{n}}(\sigma) 
 &= \mathbf{F}(a_{1}(\sigma) + c + 1, a_{2}(\sigma), \dots, a_{g}(\sigma)) \, ,
\end{align*}
 and on the other hand we have
\begin{align*}
 \chi_{V_{n}}(\sigma) 
 &= \left( \Res_{\sym{n-c-1}}^{\sym{n}} \chi_{V_{n}} \right)(\sigma) = \mathbf{P}(a_{1}(\sigma), a_{2}(\sigma), \dots, a_{g}(\sigma)) \, .
\end{align*}
The equality $(\star)$ follows because both sides are polynomials in $\qq[\cyc_{1}, \dots, \cyc_{g}]$ that evaluate to the same value on infinitely many $g$-tuples.

Finally for the claim about the dimension sequence, note that the identity element $\id_{n} \in \sym{n}$ has n many $1$-cycles and no $j$-cycles for $j \geq 2$, and hence for $n \geq c+1$ we have
\begin{align*}
 \dim_{\kk} V_{n} &= \chi_{V_{n}}(\id_{n}) = \mathbf{F}(n, 0, \dots, 0)
 = \mathbf{P}(n-c-1,0,\dots, 0)
 \\
 &= \sum_{r=0}^{g} \sum_{\lambda \,\vdash\, r}
 \chi_{W_{r}}(\lambda)
 \cdot
 \binom{n - c - 1}{a_{1}(\lambda)} 
 \prod_{j=2}^{r}\!\! \binom{0}{a_{j}(\lambda)}
 \\
 &= \sum_{r=0}^{g} 
 \chi_{W_{r}}(\id_{r})
 \cdot
 \binom{n - c - 1}{r} 
 = \sum_{r=0}^{g} 
 \dim_{\kk} (W_{r})
 \cdot
 \binom{n - c - 1}{r} \, .
\end{align*} 
\end{proof}

%
%

%
%
\subsection{Proofs of Theorems \ref{main-ranges},\ref{hyper-ranges}} \label{de-proof}
We are now ready to bring the threads together from the preceding sections to prove Theorem \ref{main-ranges} and Theorem \ref{hyper-ranges}.

\begin{proof}[Proof of \textbf{\emph{Theorem \ref{main-ranges}}}]
 To get the main statement, combine Corollary \ref{t0-t1-ranges}, Corollary \ref{cor-inductive}, Theorem \ref{thm-additive}, and Theorem \ref{thm-poly-specht}. 
 
For the sharpness of the inductive description in (1)-(4), we allude to Theorem \ref{induct-eqv}, Corollary \ref{cor-inductive}, and note the following:
\begin{itemize}
 \item $\tgen(\mathbf{I}(g)) = g$ and $\trel(\mathbf{I}(g)) = -1$ for every $g \geq -1$ by Proposition \ref{prop-Ig},
 
 \item $\tgen(\mathbf{T}(c)) = c$ and $\trel(\mathbf{T}(c)) = c+1$ for every $c \geq 0$ by Proposition \ref{prop-Tc},
 
 \item $\tgen(\mathbf{S}(c) \oplus \mathbf{I}(g)) = \max\{c+1,g\} = c+1$ and  $\trel(\mathbf{S}(c) \oplus \mathbf{I}(g)) = c+2$ whenever $0 \leq g \leq \ceil{c/2}$ and $c \geq 0$ by Proposition \ref{prop-Sc} and Proposition \ref{prop-Ig},
 
 \item $\tgen(\mathbf{V}(g)) = 2g$ and $\trel(\mathbf{V}(g)) = 2g+1$ for every $g \geq 1$ by Proposition \ref{prop-Vg},

\end{itemize}
In addition, because the dimensions of the witnessing modules are given by 
\begin{align*}
 \dim_{\qq} \mathbf{I}(g)_{n} &= \binom{n}{g} \text{ for every $n \geq 0$,}
\end{align*}
\begin{align*}
 \dim_{\qq} \mathbf{T}(c)_{n} &= 
\begin{cases}
 0 & \text{if $n \neq c$,}
 \\
 1 & \text{if $n = c$,}
\end{cases}
\end{align*}
\begin{align*}
 \dim_{\qq} (\mathbf{S}(c) \oplus \mathbf{I}(g))_{n} &=
\begin{cases}
   \binom{n}{g} & \text{if $n \leq c$,}
   \\
   1 + \binom{n}{g} & \text{if $n \geq c+1$,}
\end{cases}
\end{align*}
\begin{align*}
 \dim_{\qq} \mathbf{V}(g)_{n} &= 
\begin{cases}
  0 & \text{if $n \leq 2g-2$,}
  \\
  \frac{n-(2g-1)}{g} \binom{n}{g-1} & \text{if $n \geq 2g-1$,}
\end{cases}
\end{align*}
the sharpness of the ranges encoding both the additive structure and the polynomiality in (1)-(4) follow. 

Given $g \geq 0$, the Specht stability range $n \geq 2g$ is sharp for $\mathbf{I}(g)$ by \cite[Lemma 2.2]{hersh-reiner} and for $\mathbf{V}(g)$ as  observed in the proof of Proposition \ref{prop-Vg}. Given $c \geq 0$, the sharpness of the Specht stability range $n \geq c+1$ for $\mathbf{T}(c)$ and $\mathbf{S}(c)$ is immediate from their definition. Thus given $0 \leq g \leq \ceil{c/2}$ and $c \geq 0$, the sharp Specht stability range for $\mathbf{S}(c) \oplus \mathbf{I}(g)$ is $n \geq \max\{2g,c+1\} = c+1$.
\end{proof}

\begin{proof}[Proof of \textbf{\emph{Theorem \ref{hyper-ranges}}}]
 By Theorem \ref{hyper-to-usual}, the $\FI$-module $\co_{k}(\chain_{\star})$ is presented in finite degrees and satisfies the following: 
\begin{itemize}
 \item $\tgen(\co_{k}(\chain_{\star})) \leq 2\theta_{k}$.
 \item By Theorem \ref{pass-to-cmnr}, the triple
\begin{align*}
\begin{cases}
 \big( 
 \co_{k}(\chain_{\star}),\, -1,\, 0
 \big)
 & \text{if $\theta_{k} = \theta_{k+1} = 0$,}
 \\
 \big( 
 \co_{k}(\chain_{\star}),\, 2\theta_{k} - 2,\, \theta_{k}
 \big)
 & \text{if $\theta_{k} \geq \max\{1,\theta_{k+1}\}$,}
 \\
 \big( 
 \co_{k}(\chain_{\star}),\, 2\theta_{k+1} - 2,\, \theta_{k}
 \big)
 & \text{if $\theta_{k} < \theta_{k+1}$,}
\end{cases}
\end{align*}
satisfies Hypothesis \ref{FI-hyp}.
\item $\trel(\co_{k}(\chain_{\star})) - 1 \leq \reg(\co_{k}(\chain_{\star})) \leq \begin{cases}
 -2
 & \text{if $\theta_{k} = \theta_{k+1} = 0$,}
 \\
 2\theta_{k} & \text{if $\theta_{k} \geq \max\{1,\theta_{k+1}\}$,}
 \\
 2\theta_{k+1}-2
 & \text{if $\theta_{k} < \theta_{k+1}$,}
\end{cases}
$
\end{itemize}
Now invoke Corollary \ref{cor-inductive} and Theorem \ref{main-ranges} for the main conclusion. 

Next, assume $(\theta_{m} : m \geq 0)$ is a strictly increasing sequence such that $\theta_{0} \geq 0$ and $\tbol_{m}(\chain_{\star}) \leq \theta_{m}$ for every $m \geq 0$. As alluded to in \cite[proof of Lemma 5.3]{cmnr-range}, by the hyperhomology spectral sequence $E^{2}_{p,q} = \cofi{p}(\co_{q}(\chain_{\star})) \imp \cofibol{p+q}(\chain_{\star})$, we have 
\begin{align*}
 \trel(\co_{k}(\chain_{\star})) &\leq
 \max\big( 
 \{\tbol_{k+1}(\chain_{\star})\} \cup \{t_{p}(\co_{q}(\chain_{\star})) : p+q = k+2,\,0\leq q<k\}
 \big)
 \\
 &\leq 
 \max\big( 
 \{\theta_{k+1}\} \cup \{p + 2\theta_{q+1}-2 : p+q = k+2,\,0\leq q<k\}
 \big)
 \\
 &\leq 
 \max\big( 
 \{\theta_{k+1}\} \cup \{k + 2\theta_{q+1} - q : 0\leq q<k\}
 \big) 
 \\
 &\leq 
\begin{cases}
  \theta_{1} & \text{if $k=0$,}
  \\
  \max\{\theta_{k+1},\, 2\theta_{k} + 1\} & \text{if $k \geq 1$,}
\end{cases}
\end{align*}
noting that the sequence $(2\theta_{m} - m : m \geq 0)$ is weakly increasing everywhere. We now observe the following:
\begin{itemize}
\item If $2\theta_{k}+1 \geq \theta_{k+1}$, then
\begin{align*}
 \trel(\co_{k}(\chain_{\star})) \leq 
\begin{cases}
 \theta_{1} & \text{if $k=0$,}
 \\
 2\theta_{k}+1 & \text{if $k \geq 1$,}
\end{cases}
\end{align*}
and invoking Corollary \ref{t0-vs-t1} in case $k=0$, we have
\begin{align*}
  \tgen(\co_{k}(\chain_{\star})) \leq  
\begin{cases}
  \theta_{1}-1 & \text{if $k=0$,}
  \\
  2\theta_{k} & \text{if $k \geq 1$.}
\end{cases}
\end{align*}
We have already established under these conditions that the triple 
\begin{align*}
  \big( 
 \co_{k}(\chain_{\star}),\, 2\theta_{k+1} - 2,\, \theta_{k}
 \big)
\end{align*}
satisfies Hypothesis \ref{FI-hyp}.

 \item If $2\theta_{k}+1 < \theta_{k+1}$, then $\tgen(\co_{k}(\chain_{\star})) \leq 2\theta_{k}$ and $\trel(\co_{k}(\chain_{\star})) \leq \theta_{k+1}$, so 
\begin{align*}
  \local(\co_{k}(\chain_{\star})) \leq 2\theta_{k} + \theta_{k+1}-1 < 2\theta_{k+1}-2
\end{align*}
 by \cite[Proposition 3.1]{cmnr-range}, and hence the triple 
\begin{align*}
 \big(
 \co_{k}(\chain_{\star}),\,2\theta_{k}+\theta_{k+1}-1,\,\theta_{k}
 \big)
\end{align*}
satisfies Hypothesis \ref{FI-hyp} by Theorem \ref{pass-to-cmnr}.
\end{itemize}
We may now invoke Corollary \ref{cor-inductive} and Theorem \ref{main-ranges}.
\end{proof}

\section{Applications}
\subsection{Diagonal coinvariant algebras} \label{section-coinv}
In this section we shall prove Theorem \ref{main-coinv}. We begin by explaining the $\FI$ structure on the coinvariant algebras. Given an injection $\iota \colon S \emb T$ of finite sets, we consider the $R$-algebra morphism defined by 
\begin{align*}
 R[\B \times \iota] \colon R[\B \times T] &\rarr R[\B \times S]
 \\
 (x,t) &\mapsto 
\begin{cases}
 (x,s) & \text{if $\iota(s) = t$,}
 \\
 0 & \text{otherwise,}
\end{cases}
\end{align*}
on the variables. These transition maps define the $\ters{\FI}$-algebra $R[\B \times \bul] = \Sym(E^{\oplus \bul})$. Noting the effect on the monomials 
\begin{align*}
 &\,\,R[\B \times \iota] \bigg(
 \prod_{(x,t) \in \B \times T} (x,t)^{\alpha(x,t)}
 \bigg) 
 =
\begin{dcases*}
 0 & \text{if
\begin{tabular}{l}
 $\alpha(x,t) > 0$ for some\\
 $t \in T - \iota(S)$ and $x \in \B$,
\end{tabular}
  }
 \\
 \prod_{(x,s) \in \B \times S} (x,s)^{\alpha(x,\iota(s))} & \text{otherwise,}
\end{dcases*}
\end{align*}
we see that
\begin{align*}
 \Sym(E^{\oplus \bul}) = \bigoplus_{\J \in \nn^{\B}} \Sym^{\J}(E^{\oplus\bul})
\end{align*}
is in fact an $\nn^{\B}$-graded $\ters{\FI}$-algebra. A straightforward computation shows that under the assignment 
\begin{align*}
 S \mapsto \inv{\J}{S}{E} := \{ f \in \Sym^{\J\!}(E^{\oplus S}) : \sigma f = f \text{ for each }\sigma \in \sym{S}\} \, ,
\end{align*}
 the non-constant invariants form a homogeneous (with respect to the $\nn^{\B}$-grading) $\ters{\FI}$-submodule 
\begin{align*}
 \bigoplus_{0 \neq \J \in \nn^{\B}} \inv{\J}{\bul}{E} =: \inv{}{}{E} \leq \Sym(E^{\oplus \bul}) \, .
\end{align*}
Thus $\inv{}{}{E}$ generates a homogeneous $\ters{\FI}$-ideal inside $\Sym(E^{\oplus \bul})$, and the resulting quotient $\coinv{}{}{E}$ becomes an $\nn^{\B}$-graded $\ters{\FI}$-algebra. Note that the degree $\J$-component of the $\ters{\FI}$-ideal generated by $\inv{}{}{E}$ inside $\Sym(E^{\oplus \bul})$ is the image of the sum of the multiplication maps
\begin{align*}
 \bigoplus_{0 \neq \I \leq \J} \Sym^{\J-\I}(E^{\oplus \bul}) \otimes_{R} \inv{\I}{}{E}
 \, \rarr \Sym^{\J}(E^{\oplus \bul}) \, .
\end{align*}
Consequently, the degree $\J$-component of $\coinv{}{}{E}$ is given by an exact sequence
\begin{equation} \label{coinv-seq}
 \bigoplus_{0 \neq \I \leq \J} \Sym^{\J-\I}(E^{\oplus \bul}) \otimes_{R} \inv{\I}{}{E}
 \, \rarr \Sym^{\J}(E^{\oplus \bul}) \rarr \coinv{\J}{}{E} \rarr 0 \tag{$\spadesuit$} 
\end{equation}
of $\ters{\FI}$-modules. We pin down the relevant invariants of the first two terms appearing in (\ref{coinv-seq}) in Corollary \ref{about-sym-check} and Proposition \ref{about-tens-check} to bound those for $\coinv{\J}{}{E}$ in Theorem \ref{about-coinv-check}.

\begin{defn}
 Writing $\setz$ for the category of pointed sets where the distinguished point is denoted 0, the functor 
\begin{align*}
 F_{R} \colon \setz \rarr \lMod{R}
\end{align*}
sends a pointed set $S \sqcup \{0\}$ to the free $R$-module with basis $S$ and sends a pointed map to the unique $R$-module homomorphism extending it, considering $0$ as the additive zero. It is the left adjoint of the forgetful functor that assigns an $R$-module its underlying set together with its additive zero.
\end{defn}

\begin{lem} \label{sh-on-monomial}
 Let $R$ be a commutative ring, $E$ a free $R$-module with a finite basis $\B$, and $\J \colon \B \rarr \nn$ be a multi-degree. Then the assignment 
\begin{align*}
 S \mapsto \mon^{\J}_{S} := \left\{
 \prod_{(x,s) \in \B \times S} (x,s)^{\alpha(x,s)} :
 \sum_{s \in S} \alpha(x,s) = \J(x) \text{ for each } x \in \B
 \right\} \sqcup \{0\}
\end{align*}
defines a functor $\mon^{\J}_{\bul} \colon \sh \rarr \mathsf{Set}_{0}$ such that the diagram
\begin{align*}
 \xymatrixcolsep{3pc}
 \xymatrixrowsep{3pc}
 \xymatrix{
 \FI^{\opp} \ar@{_{(}->}_-{\nu}[d] \ar[drr]^-{\,\,\,\,\,\,\Sym^{\J}(E^{\oplus\bul})}\\
 \sh \ar[r]_-{\mon^{\J}_{\bul}} & \setz \ar[r]_-{F_{R}} & \lMod{R}
 }
\end{align*}
commutes, where the embedding $\nu \colon \FI^{\opp} \emb \sh$ is the one in \emph{\cite[Remark 4.1.3]{cef}}.
\end{lem}
\begin{proof}
 Let $(C,D,\phi) \colon S \rarr T$ be a morphism in $\sh$, that is,
\begin{align*}
 C \subseteq S\,, \quad D \subseteq T\,, \quad \phi \colon C \rarr D \text{ is a bijection.}
\end{align*}
To define the transition maps of our (to be) functor $\mon^{\J}_{\bul}$, we write
\begin{align*}
 (C,D,\phi)_{*} \colon \mon^{\J}_{S} 
 &\rarr \mon^{\J}_{T} 
 \\
 \prod_{(x,s) \in \B \times S} (x,s)^{\alpha(x,s)}
 &\mapsto
\begin{dcases*}
 0 & \text{if 
\begin{tabular}{l}
 $\alpha(x,s) > 0$ for some \\
 $s \in S-C$ and $x \in \B$,
\end{tabular}
 }
 \\
 \prod_{(x,t) \in \B \times T} (x,t)^{\alpha\left(x,\,\phi^{-1}(t)\right)}
 & 
 \text{otherwise.}
\end{dcases*}
 \\
 0 &\mapsto 0 \, .
\end{align*}
Clearly this sends the identity morphisms of $\sh$ to identity maps on the corresponding pointed sets. Next, let $(C,D,\phi) \colon S \rarr T$ and $(K,L,\psi) \colon T \rarr U$ be composable morphisms in $\sh$, that is,
\begin{align*}
 C \subseteq S\,, \quad D \subseteq T\,, \quad \phi \colon C \rarr D \text{ is a bijection.}
 \\
 K \subseteq T\,, \quad L \subseteq U\,, \quad \psi \colon K \rarr L \text{ is a bijection.}
\end{align*}
 We see that the composite $(K,L,\psi)_{*} \circ (C,D,\phi)_{*}$ is given by
\begin{align*}
  \mon^{\J}_{S}  
 &\rarr \mon^{\J}_{U} 
 \\
 \prod_{(x,s) \in \B \times S} (x,s)^{\alpha(x,s)}
 &\mapsto
\begin{dcases*}
 0 & \text{if 
\begin{tabular}{l}
 $\alpha(x,s) > 0$ for some \\
 $s \in S-C$ and $x \in \B$, or \\
 $\alpha(x,\,\phi^{-1}(t)) > 0$ for some\\
 $t \in D - K$ and $x \in \B$.
\end{tabular}
 }
 \\
 \prod_{(x,u) \in \B \times U} (x,u)^{\alpha\left(x,\, \phi^{-1}(\psi^{-1}(u))\right)}
 & 
 \text{otherwise.}
\end{dcases*}
\end{align*}
On the other hand, inside $\sh$ we have by \cite[Definition 4.1.1]{cef}
\begin{align*}
 (K,L,\psi) \circ (C,D,\phi) 
 &= (\phi^{-1}(D \cap K),\,\psi(D \cap K),\, \psi \circ \phi) \, .
\end{align*}
Therefore $\left( (K,L,\psi) \circ (C,D,\phi)  \right)_{*}$ is given by
\begin{align*}
 \mon^{\J}_{S} 
 &\rarr \mon^{\J}_{U} 
 \\
 \prod_{(x,s) \in \B \times S} (x,s)^{\alpha(x,s)}
 &\mapsto
\begin{dcases*}
 0 & \text{if 
\begin{tabular}{l}
 $\alpha(x,s) > 0$ for some\\
 $s \in S-\phi^{-1}(D \cap K)$ \\
 and $x \in \B$,
\end{tabular}
 }
 \\
 \prod_{(x,u) \in \B \times U} (x,u)^{\alpha\left(x,\,(\psi \circ \phi)^{-1}(u)\right)}
 & 
 \text{otherwise.}
\end{dcases*}
\end{align*}
We observe that $s \in S-\phi^{-1}(D \cap K)$ if and only if either $s \in S - C$, or $s \in C$ and $\phi(s) \in D-K$. In other words, 
\begin{align*}
 S-\phi^{-1}(D \cap K) &= (S-C) \, \cup \, \phi^{-1}(D-K) \, \text{, so}
 \\
 \B \times \left( S-\phi^{-1}(D \cap K) \right) 
 &=
 \B \times (S-C) \,\, \cup \,\, \B \times \phi^{-1}(D-K) \, ,
\end{align*}
verifying $(K,L,\psi)_{*} \circ (C,D,\phi)_{*} = \left( (K,L,\psi) \circ (C,D,\phi)  \right)_{*}$.

Noting that $\nu$ sends a morphism $\iota \colon T \hookleftarrow S$ in $\FI^{\opp}$ to $(\iota(S),S,\iota) \colon T \rarr S$ in $\sh$, we inspect that 
\begin{align*}
 (\iota(S),S,\iota)_{*} = R[\B \times \iota]
\end{align*}
as desired.
\end{proof}

\begin{cor} \label{about-sym-check}
 Let $R$ be a commutative ring, $E$ a free $R$-module with a finite basis $\B$, and $\J \colon \B \rarr \nn$ be a multi-degree. Then there exists a pointed $\FB$-set 
\begin{align*}
 Y^{\J}_{\bul} \colon \FB \rarr \setz
\end{align*}
such that 
\begin{birki}
 \item $Y^{\J}_{S} = \{0\}$ if and only if $|S| > |\J|$.
 \vspace{0.08in}
 \item The $\ters{\FI}$-module $\Sym^{\J}(E^{\oplus \bul})$ extends to an $\sh$-module which is
\begin{itemize}
 \item isomorphic to $\induce(F_{R}(Y^{\J}_{\bul}))$, and
 \item generated in degrees $\leq |\J|$.
\end{itemize}
\end{birki}
\end{cor}
\begin{proof}
 The recipe for defining $\cofi{0}$ of an $\FI$-module as mentioned in the introduction can be mimicked for pointed $\FI$-sets, thanks to the existence of zero morphisms. There is a functor $\pi^{*} \colon [\FB,\setz] \rarr [\FI,\setz]$ which ``extends by zero'' and it has a left adjoint, for which we again write $\cofi{0}$. Now we declare $Y^{\J}_{\bul} := \cofi{0}(\mon_{\bul}^{\J})$,\footnote{Let us resolve the abuse of notation here: we take the pointed $\sh$-set $\mon_{\bul}^{\J}$ in Lemma \ref{sh-on-monomial}, consider its underlying pointed $\FI$-set (via the usual covariant inclusion $\FI \emb \sh$, not $\nu \colon \FI^{\opp} \emb \sh$ of Lemma \ref{sh-on-monomial}), and finally apply $\cofi{0}$ to it.} and observe that it can be described as 
\begin{align*}
 Y^{\J}_{S} = 
 \left\{
 \prod_{(x,s) \in \B \times S} (x,s)^{\alpha(x,s)} :
\begin{array}{l}
 \sum_{s \in S} \alpha(x,s) = \J(x) \text{ for each } x \in \B \,, \\
 \quad\quad\textbf{and} \\
 \text{for every $s \in S$, there exists } x \in \B \\
 \text{such that } \alpha(x,s) > 0\,.
\end{array}
 \!\!\right\} \sqcup \{0\} \, .
\end{align*}
To prove (1), first assume $Y^{\J}_{S} \neq \{0\}$. Then it contains a monomial \ds{\prod_{(x,s) \in \B \times S} (x,s)^{\alpha(x,s)}} which has to satisfy 
\begin{align*}
 |\J| = \sum_{x \in \B} \J(x) = \sum_{x \in \B} \sum_{s \in S} \alpha(x,s) = \sum_{s \in S}\sum_{x \in \B} \alpha(x,s) \geq \sum_{s \in S} 1 = |S| \, .
\end{align*}
Conversely, assume $|\J| \geq |S|$. Then letting $J_{x}$ to be a finite set of size $\J(x)$ for each $x \in \B$, there exists an injection 
\begin{align*}
 \lambda \colon S \emb \bigsqcup_{x \in \B} J_{x} \, .
\end{align*}
Now defining 
\begin{align*}
 \alpha_{0} \colon \B \times S &\rarr \{0,1,\dots\}
 \\
 (x,s) &\mapsto 
\begin{cases}
 1 & \text{if $\lambda(s) \in J_{x}$,} \\
 0 & \text{otherwise,}
\end{cases}
\end{align*}
we see that the monomial 
\begin{align*}
 m_{0} := \prod_{(x,s) \in \B \times S}(x,s)^{\alpha_{0}(x,s)}
\end{align*}
belongs to $Y_{S}^{\J}$ and hence $Y_{S}^{\J} \neq \empt$.

To prove (2), we first note that the ``extend by zero'' functors from $\FB$-objects to $\FI$-objects commute with the forgetful functor $\lMod{R} \rarr \setz$, therefore the corresponding left adjoints $\cofi{0}$ and $F_{R}$ also commute. In particular, we have 
\begin{align*}
 \cofi{0}(F_{R}(\mon^{\J}_{\bul})) \cong F_{R}(\cofi{0}(\mon^{\J}_{\bul})) = F_{R}(Y^{\J}_{\bul})
\end{align*}
as $\FB$-modules. We conclude by Lemma \ref{sh-on-monomial}, \cite[Theorem 4.1.5]{cef}, and \cite[Theorem 4.1.7]{cef}.
\end{proof}

\begin{lem} \label{lem-inv}
 Let $R$ be a commutative ring, $E$ a free $R$-module with a finite basis, and $\J \colon \B \rarr \nn$ be a multi-degree whose total degree is $|\J| \geq 1$. Then the assignment
\begin{align*}
 S \mapsto \orb^{\J}_{S} := \mon^{\J}_{S} / \sym{S} 
\end{align*}
defines a functor $\orb^{\J}_{\bul} \colon \FI^{\opp} \rarr \setz$ such that 
\begin{birki}
 \item $\inv{\J}{}{E} \cong F_{R}(\orb^{\J}_{\bul})$.
 \item Given a proper injection $\iota \colon T-\{t_{0}\} \emb T$ of finite sets, the transition map
\begin{align*}
 \orb^{\J}_{\iota} \colon \orb^{\J}_{T} \rarr \orb^{\J}_{T-\{t_{0}\}}
\end{align*}
 is always surjective, and injective if and only if $|T| > |\J|$.
\end{birki}
\end{lem}
\begin{proof}
We have proved in Lemma \ref{sh-on-monomial} that 
\begin{align*}
 \mon^{\J}_{S} - \{0\} = \left\{
 \prod_{(x,s) \in \B \times S} (x,s)^{\alpha(x,s)} :
 \sum_{s \in S} \alpha(x,s) = \J(x) \text{ for each } x \in \B
 \right\}
\end{align*}
forms a basis of $\Sym^{\J}(E^{\oplus S})$ as an $R$-module. The $\sym{S}$ action
\begin{align*}
 \sigma \cdot \prod_{(x,s) \in \B \times S} (x,s)^{\alpha(x,s)} = \prod_{(x,s) \in \B \times S} (x,\sigma s)^{\alpha(x,s)}\, ,
\end{align*}
makes $\Sym^{\J}(E^{\oplus S})$ a permutation $R\sym{S}$-module defined on the $\sym{S}$-set $\mon^{\J}_{S} -\{0\}$. Thus by \cite[Lemma 3.2.1]{neusel-smith}, the set of orbit sums 
\begin{align*}
  \left\{
 \sum_{m \in \oo}m \,\,:\,\, \oo \subseteq \mon^{\J}_{S} -\{0\} \text{ an $\sym{S}$-orbit}
 \right\}
\end{align*}
form an $R$-basis of $\inv{\J}{S}{E} = \left(\Sym^{\J}(E^{\oplus S})\right)^{\sym{S}}$. In other words, this shows that the $R$-linear map defined by 
\begin{align*}
 \vphi_{S} \colon F_{R}(\orb^{\J}_{S}) &\rarr \inv{\J}{S}{E} 
 \\
 \oo &\mapsto \sum_{m \in \oo} m
\end{align*}
is an isomorphism. Moreover for an injection $\iota \colon S \emb T$, we define $\orb^{\J}_{\iota} \colon \orb^{\J}_{T} \rarr \orb^{\J}_{S}$ as follows: given an $\sym{T}$-orbit $\oo \subseteq \mon^{\J}_{T}-\{0\}$, we declare $\orb^{\J}_{\iota}(\oo) := 0$ if for \textbf{every} monomial 
\begin{align*}
 \prod_{(x,t) \in \B \times T} (x,t)^{\alpha(x,t)} \in \oo
\end{align*}
there exists $t \in T - \iota(S)$ and $x \in \B$ such that $\alpha(x,t) > 0$. If, on the other hand, $\oo$ contains a monomial of the form 
\begin{align*}
 \prod_{(x,s) \in \B \times S} (x,\iota(s))^{\alpha_{0}(x,\iota(s))} \, ,
\end{align*}
we declare $\orb^{\J}_{\iota}(\oo)$ to be the $\sym{S}$-orbit of the monomial 
\begin{align*}
 \prod_{(x, s) \in \B \times S} (x,s)^{\alpha_{0}(x,s)} \in \mon^{\J}_{S} \, .
\end{align*}
We see here that if $\orb^{\J}_{\iota}(\oo) \neq 0$, the $\sym{T}$-orbit $\oo$ can be reconstructed from the $\sym{S}$-orbit $\orb^{\J}_{\iota}(\oo)$.  It is straightforward to check that $\orb^{\J}_{\bul} \colon \FI^{\opp} \rarr \setz$ is a functor. Moreover for every injection $\iota \colon S \emb T$, the diagram 
\begin{align*}
 \xymatrix{
 F_{R}(\orb^{\J}_{T}) \ar[d]_{F_{R}(\orb^{\J}_{\iota})} \ar[r]^{\vphi_{T}} 
 & \inv{\J}{T}{E} \ar[d]^{R[\B \times \iota]}
 \\
 F_{R}(\orb^{\J}_{S}) \ar[r]^{\vphi_{S}} & \inv{\J}{S}{E}
 }
\end{align*}
commutes, and (1) follows. The surjectivity claim of (2) is also evident from the above description. 

For the forward direction of the injectivity claim in (2), assume $|T| \leq |\J|$ so that letting $J_{x}$ to be a finite set of size $\J(x)$ for each $x \in \B$, there exists an injection 
\begin{align*}
 \lambda \colon T \emb \bigsqcup_{x \in \B} J_{x} \, .
\end{align*}
Now defining 
$ \alpha (x,t) :=
\begin{cases}
 1 & \text{if $\lambda(t) \in J_{x}$,} \\
 0 & \text{otherwise,}
\end{cases}
$ we see that the monomial 
\begin{align*}
 m := \prod_{(x,t) \in \B \times T}(x,t)^{\alpha(x,t)}
\end{align*}
belongs to $\mon^{\J}_{T}$ with the property that for every $t \in T$ there exists $x \in \B $ (namely the unique $x$ with $\lambda(t) \in J_{x}$) such that $\alpha(x,t) > 0$. Therefore every monomial in the $\sym{T}$-orbit, say $\oo$, of $m$ also has this property and therefore $\orb^{\J}_{\iota}(\oo) = 0 = \orb^{\J}_{\iota}(0)$, hence $\orb^{\J}_{\iota}$ is not injective. For the backward direction of the injectivity claim in (2), assume $\orb^{\J}_{\iota}$ is not injective. Then there must be an $\sym{T}$-orbit $\oo \subseteq \mon^{\J}_{T} - \{0\}$ such that $\orb^{\J}_{\iota}(\oo) = 0$. Let us fix 
\begin{align*}
 m = \prod_{(x,t) \in \B \times T} (x,t)^{\alpha(x,t)} \in \oo \, ,
\end{align*}
so there should be an $x_{0} \in \B$ such that $\alpha(x_{0},t_{0}) \geq 1$. And moreover given another $t_{1} \in T$ we can pick $\sigma \in \sym{T}$ with $\sigma t_{1} = t_{0}$, and the monomial
\begin{align*}
 \sigma \cdot m = \prod_{(x,t) \in \B \times T} (x,\sigma(t))^{\alpha(x,t)} =  \prod_{(x,t) \in \B \times T} (x,t)^{\alpha(x,\sigma^{-1}t)} \in \oo 
\end{align*}
should also have a positive exponent with $t_{0}$, that is, there should be an $x_{1} \in \B$ such that $0 < \alpha(x_{1},\sigma^{-1}t_{0}) = \alpha(x_{1},t_{1})$. It follows that we have 
$
 \sum_{x \in \B} \alpha(x,t_{1}) \geq 1 \
$, and this holds for every $t_{1} \in T$. As a result,
\begin{align*}
 |\J| = \sum_{x \in \B} \J(x) = \sum_{x \in \B} \sum_{t \in T} \alpha(x,t) = \sum_{t \in T}\sum_{x \in \B} \alpha(x,t) \geq \sum_{t \in T} 1 = |T| \, .
\end{align*}
\end{proof}


\begin{cor} \label{about-inv-check}
 Let $R$ be a commutative ring, $E$ a free $R$-module with a finite basis $\B$, and $\J \colon \B \rarr \nn$ be a multi-degree whose total degree is $|\J| \geq 1$. Then the $\FI$-module is $\inv{\J}{}{E}^{\vee}$ satisfies the following:
\begin{birki}
  \item $\weak(\inv{\J}{}{E})^{\vee}) = 0$.
  \vspace{0.1in}
  \item For each $i \geq 0$, we have $t_{i}(\inv{\J}{}{E}^{\vee}) = |\J| + i$.
 \vspace{0.1in}
 \item For each $j \geq 0$, we have 
 $
 h^{j}(\inv{\J}{}{E}^{\vee}) = 
\begin{cases}
 |\J| - 1 & \text{if $j=1$,}
 \\
 -1 & \text{otherwise.}
\end{cases}
 $
 \vspace{0.1in} 
 \item There exists an $\FB$-module $W$ concentrated in degree $0$ and an $\FI$-module $T$ with $\deg(T) = |\J|-1$ such that there is a short exact sequence 
\begin{align*}
 0 \rarr \inv{\J}{}{E}^{\vee} \rarr \induce(W) \rarr T \rarr 0 \
\end{align*}
of $\FI$-modules defined over $R$.
\end{birki}
\end{cor}
\begin{proof}
 We shall dualize the relevant parts of Lemma \ref{lem-inv}. Let us shortly write 
\begin{align*}
 V := \inv{\J}{}{E}^{\vee} \, .
\end{align*}
The first claim of part (2) in Lemma \ref{lem-inv} says that the transition maps of the $\ters{\FI}$-module $\inv{\J}{}{E}$ are surjective, hence the transition maps of the $\FI$-module $V$ are injective, that is, $h^{0}(V) = -1$.
 
By Lemma \ref{lem-inv} part (2), the transition map $V_{n} \rarr V_{n+1}$ of the $\FI$-module $V$ is an isomorphism once $n \geq |\J|$. As a result, we have $\weak(V) = 0$, $\tgen(V) \leq |\J|$, and every transition map of the shifted $\FI$-module 
$\shift{V}{|\J|}$
is an isomorphism. Thus there exists an $\FB$-module $W$ concentrated in degree $0$ such that 
\begin{align*}
 \shift{V}{|\J|} \cong \induce(W) \, .
\end{align*}
Since this is an $\cofi{0}$-acyclic $\FI$-module generated in degrees $\leq 0$ by Theorem \ref{characterize-H0-acyclic}, Theorem \ref{pass-to-cmnr} yields that $(V,\,\, |\J|-1,\,\,0)$ satisfies Hypothesis \ref{FI-hyp}. Moreover by Lemma \ref{lem-inv} the transition map of $\orb^{\J}_{\bul}$ in degree $|\J|-1$ sits in an exact sequence 
\begin{align*}
 0 \rarr K \rarr F_{R}(\orb^{\J}_{|\J|}) \rarr F_{R}(\orb^{\J}_{|\J|-1}) \rarr 0 \, .
\end{align*}
of $R$-modules with $K \neq 0$ which has to split. Thus the transition map of $V$ in degree $|\J|-1$ sits in a split exact sequence
\begin{align*}
 0 \rarr V_{|\J|-1} \rarr V_{|\J|} \rarr K^{\vee} \rarr 0
\end{align*}
with $K^{\vee} \neq 0$.
As a result, $\tgen(V) = |\J|$ and $\shift{V}{s}$ cannot be $\cofi{0}$-acyclic for $s < |\J|$, that is, $(V,\,\, s-1,\,\,0)$ does \textbf{not} satisfy Hypothesis \ref{FI-hyp}. Thus by Theorem \ref{pass-to-cmnr} we have $\local(V) = |\J| - 1$. By Theorem \ref{structure-complex}, we have $h^{j}(V) = -1$ for $j \geq 2$ and hence 
\begin{align*}
 h^{1}(V) = \local(V) = |\J| - 1\, .
\end{align*}
Applying \cite[Corollary 4.15]{ramos-coh} we get 
\begin{align*}
 t_{i}(V) - i \leq \max\{h^{j}(V) + j : h^{j}(V) \neq - 1\} = |\J|
\end{align*}
for every $i \geq 1$. On the other hand, by Corollary \ref{t0-vs-t1} and Corollary \ref{t1-not-0} we have 
\begin{align*}
 |\J|  = \tgen(V) \leq \max\{0,\trel(V)-1\} = \trel(V) - 1 \, ,
\end{align*}
so $\trel(V) - 1 = |\J|$. We get $t_{i}(V) - i = |\J|$ for $i \geq 2$ as well by Theorem \ref{gan-chain}.

Finally, we note that because $V$ is torsion-free, the natural map $V \rarr \shift{V}{|\J|}$ is injective whose cokernel $T$ is torsion because the transition maps of $V$ are eventually isomorphisms. Thus we can apply \cite[Theorem 2.10]{cmnr-range} to the complex 
\begin{align*}
 0 \rarr V \rarr \shift{V}{|\J|} \rarr 0
\end{align*}
to deduce that $\locoh{1}(V) \cong T$ and hence $\deg(T) = h^{1}(V) = |\J|-1$.
\end{proof}

%
%
%

\begin{prop} \label{about-tens-check}
 Let $R$ be a commutative ring, $E$ a free $R$-module with a finite basis $\B$, and $\I,\J \colon \B \rarr \nn$ be multi-degrees. Then the $\FI$-module 
\begin{align*}
 U(\I,\J) := \left( \Sym^{\I}(E^{\oplus \bul}) \otimes_{R} \inv{\J}{}{E} \right)^{\vee}
\end{align*}
is presented in finite degrees and satisfies the following:
\begin{birki}
 \item $\weak(U(\I,\J)) = |\I|$.
 \vspace{0.1in}
 \item For each $j \geq 0$, we have 
 $
 h^{j}(U(\I,\J)) = 
\begin{cases}
 |\J| - 1 & \text{if $j=1$,}
 \\
 -1 & \text{otherwise.}
\end{cases}
 $
\end{birki}
\end{prop}
\begin{proof}
 Taking the $R$-dual distributes over a tensor product of finitely generated free $R$-modules \cite{dual-tensor}. Thus by Corollary \ref{about-sym-check} and Lemma \ref{lem-inv} we have isomorphisms
\begin{align*}
 U(\I,\J) 
 &\cong \Sym^{\I}(E^{\oplus \bul})^{\vee} \otimes_{R} \inv{\J}{}{E}^{\vee} 
 \\
 &\cong \induce(F_{R}(Y^{\I}_{\bul}))^{\vee} \otimes_{R} 
 \inv{\J}{}{E}^{\vee} 
 \\
 &\cong \induce(F_{R}(Y^{\I}_{\bul})^{\vee}) \otimes_{R} 
 \inv{\J}{}{E}^{\vee}
\end{align*}
of $\FI$-modules defined over $R$, where the commuting of the functors $\induce$ and $(-)^{\vee}$ in the last isomorphism follows from the description \cite[Definition 2.2.2]{cef} (the same observation is used in \cite[proof of Lemma 2.5]{miller-wilson-PConf-hyper}). Being $R$-free pointwise, the functor 
\begin{align*}
 \induce(F_{R}(Y^{\I}_{\bul})^{\vee}) \otimes_{R} - \colon [\FI,\lMod{R}] \rarr [\FI,\lMod{R}]
\end{align*}
is exact. Now applying it to the short exact sequence in part (4) of Corollary \ref{about-inv-check} yields a short exact sequence 
\begin{align*}
 0 \rarr U(\I,\J) \rarr 
 \induce(F_{R}(Y^{\I}_{\bul})^{\vee}) \otimes_{R} \induce(W) \rarr \wt{T} \rarr 0
\end{align*}
where $\deg(W) = 0$ and $\deg(\wt{T}) = |\J|-1$. We observe by the description in \cite[Definition 2.2.2]{cef} that given a finite set $S$, we have 
\begin{align*}
 \left(\induce(F_{R}(Y^{\I}_{\bul})^{\vee}) \otimes_{R} \induce(W)\right)_{S}
 &= \induce(F_{R}(Y^{\I}_{\bul})^{\vee})_{S} \otimes_{R} \induce(W)_{S}
 \\
 &= \bigoplus_{T \subseteq S} F_{R}(Y^{\I}_{T})^{\vee} \otimes_{R} \bigoplus_{T \subseteq S} W_{T}
 \\
 &= \bigoplus_{T \subseteq S} F_{R}(Y^{\I}_{T})^{\vee} \otimes_{R} W_{0}
 \\
 &\cong \induce(F_{R}(Y^{\I}_{\bul})^{\vee} \otimes_{R} W_{0})_{S} \, ,
\end{align*}
hence the middle term of the above exact sequence is an $\cofi{0}$-acyclic $\FI$-module generated in degrees $\leq |\J|$ by Corollary \ref{about-sym-check}. Now we can conclude $U(\I,\J)$ is presented in finite degrees via \cite[Theorem 2.3]{cmnr-range}. Moreover, the complex  
\begin{align*}
 I^{\star} : 0 \rarr U(\I,\J) \rarr \induce(F_{R}(Y^{\J}_{\bul})^{\vee}) \otimes_{R} \induce(W) \rarr 0
\end{align*}
satisfies the hypotheses of \cite[Theorem 2.10]{cmnr-range} and hence 
\begin{align*}
 \locoh{j}(U(\I,\J)) = \co^{j}(I^{\star}) =
\begin{cases}
 \wt{T} & \text{if $j = 1$,}
 \\
 0 & \text{otherwise.}
\end{cases}
\end{align*}
Moreover by \cite[Proposition 2.9]{cmnr-range} we have 
\begin{align*}
 |\I| = \max\{\weak(U(\I,\J)),-1\} 
\end{align*}
and hence $\weak(U(\I,\J)) = |\I|$.
\end{proof}

\begin{thm} \label{about-coinv-check}
 Let $R$ be a commutative ring, $E$ a free $R$-module with a finite basis $\B$, and  $\J \colon \B \rarr \nn$ be a multi-degree whose total degree is $|\J| \geq 1$. Then the $\FI$-module $\coinv{\J}{}{E}^{\vee}$ is presented in finite degrees and satisfies the following:
\begin{birki}
 \item $\weak(\coinv{\J}{}{E}^{\vee}) \leq |\J|$.
 \item For each $j \geq 0$, we have 
 $
 h^{j}(\coinv{\J}{}{E}^{\vee}) \leq 
\begin{cases}
 2|\J| - 2j + 2 & \text{if $2 \leq j \leq  |\J|+1$,}
 \\
 -1 & \text{otherwise.}
\end{cases}
 $
\end{birki}
\end{thm}

\begin{proof} Because dualizing is left exact, with the notation of Proposition \ref{about-tens-check}, from (\ref{coinv-seq}) we get an exact sequence
\begin{align*}
 0 \rarr
 \coinv{\J}{}{E}^{\vee}
 \rarr
 \Sym^{\J}(E^{\oplus \bul})^{\vee}
 \rarr
 \bigoplus_{0 \neq \I \leq \J} U(\J-\I,\J) 
\end{align*}
of $\FI$-modules. Here (1) follows from \cite[Proposition 2.9]{cmnr-range} and Corollary \ref{about-sym-check}. For (2), by \cite[proof of Proposition 3.3, page 11]{cmnr-range}, we have 
\begin{align*}
 h^{0}(\coinv{\J}{}{E}^{\vee}) &\leq 
 h^{0}(\Sym^{\J}(E^{\oplus \bul})^{\vee}) = -1 \text{ and}
 \\
 h^{1}(\coinv{\J}{}{E}^{\vee}) &\leq \max\!
 \bigg(
 \{h^{1}(\Sym^{\J}(E^{\oplus \bul})^{\vee})\}
 \cup
 \{h^{0}(U(\J-\I,\J)) : 0 \neq \I \leq \J\}
 \bigg) 
 \\ &\leq -1 \, ,
\end{align*}
using Proposition \ref{about-inv-check} and Proposition \ref{about-tens-check}. The rest of (2) follows from (1), \cite[Theorem 2.3]{cmnr-range}, Theorem \ref{structure-complex}, and Theorem \ref{pass-to-cmnr}.
\end{proof}

\begin{proof}[Proof of \textbf{\emph{Theorem \ref{main-coinv}}}]
 By Theorem \ref{about-coinv-check} and Theorem \ref{pass-to-cmnr}, the triple 
\begin{align*}
 \mathlarger{
 \left(
 \coinv{\J}{}{E}^{\vee},\, 2|\J|-2,\,|\J|
 \right)
 }
\end{align*}
satisfies Hypothesis \ref{FI-hyp}. We get the desired stable ranges by Theorem \ref{main-ranges}.
\end{proof}


\subsection{Ordered configuration spaces}
In this section we shall prove Theorem \ref{main-config}. We first state a result that transforms stable ranges for a manifold $\manif$ with punctures into those for $\manif$ itself. We closely follow Miller--Wilson's treatment \cite{miller-wilson-PConf-hyper}.

\begin{thm}\label{reduce-PConf-to-punctured}
 Let $\manif$ be a connected manifold of dimension $\geq 2$, $\ab$ be an abelian group, and $(\weak_{k} : k \geq -1)$ be a weakly increasing sequence of integers with $\weak_{-1} = -1$ such that the $\FI$-module
\begin{align*}
 \co^{k}(\confix{}{\manif-Q};\ab) \text{ is generated in degrees $\leq \weak_{k}$}
\end{align*}
for every nonempty finite subset $Q \subseteq \manif$ and $k \geq 0$. Then for every $k \geq 1$, the $\FI$-module $\co^{k}(\confix{}{\manif};\ab)$ is identically zero if $\weak_{k} \leq 0$, and otherwise satisfies the following:
\begin{birki}
 \item $\weak(\co^{k}(\confix{}{\manif}; \ab)) \leq \weak_{k}$.
 \vspace{0.1cm}
 \item $\tgen(\co^{k}(\confix{}{\manif}; \ab)) \leq 2\weak_{k}$.
 \item $h^{j}(\co^{k}(\confix{}{\manif}; \ab)) \leq 
\begin{cases}
 2\weak_{k}-2 & \text{if $j=0$,} \\
 \max\{-1,\,2\weak_{k}-4\} & \text{if $j=1$,} \\ 
 \max\{-1,\,2\weak_{k}-2j+2\} & \text{if $j \geq 2$.}
\end{cases}
$
\vspace{0.1cm}
 \item $\local(\co^{k}(\confix{}{\manif}; \ab)) \leq 2\weak_{k}-2$.
\vspace{0.1cm}
 \item $\reg(\co^{k}(\confix{}{\manif}; \ab)) \leq 2\weak_{k}$.
\end{birki}
\end{thm}

\begin{proof}
In \cite[pages 7,8]{miller-wilson-PConf-hyper} it is shown that letting $x_{0}, x_{1}, x_{2}, \dots$ be distinct points in $\manif$ and writing 
\begin{align*}
 \tbol_{q}^{N}(s) := \tgen\!\left(
 \co^{N-s-q}(
 \confix{}{\manif - \{x_{0}, \dots, x_{s}\}}; \ab)
 \right)
\end{align*}
for every $q,N,s$, there is a chain complex $C_{*}^{\leq k+1}$ of $\FI$-modules such that
\begin{align*}
 \co_{1}(C_{*}^{\leq k+1}) &= \co^{k}(\confix{}{\manif}; \ab)\, ,
 \\
 \tbol_{1}(C_{*}^{\leq k+1}) &\leq 
 \max\{\tbol_{1}^{k+1}(0), \tbol_{1}^{k+1}(1)\}  ,
 \\
 \tbol_{2}(C_{*}^{\leq k+1}) &\leq 
 \max\{\tbol_{2}^{k+1}(0), \tbol_{2}^{k+1}(1), \tbol_{2}^{k+1}(2)\} \, .
\end{align*}
Invoking the hypothesis we get $\tbol_{1}(C_{*}^{\leq k+1}) \leq \weak_{k}$ and $\tbol_{2}(C_{*}^{\leq k+1}) \leq \weak_{k-1}$. Now parts (1),(2) of Theorem \ref{hyper-to-usual} applied to $\co_{1}(C_{*}^{\leq k+1})$ yield (1) and (2) here.

If $\weak_{k} \leq 0$, the $\FI$-module $\co_{1}(C_{*}^{\leq k+1})$ being generated in degrees $\leq 0$ forces it to vanish because
\begin{align*}
 \left(\co_{1}(C_{*}^{\leq k+1})\right)_{0} = \co^{k}(\confix{0}{\manif}; \ab) = \co^{k}(pt; \ab) = 0 
\end{align*}
as we are assuming $k \geq 1$. Assuming $\weak_{k} \geq 1$, parts (3), (4), (5) of Theorem \ref{hyper-to-usual} applied to $\co_{1}(C_{*}^{\leq k+1})$ yield (3), (4), (5) here.
\end{proof}

Before dealing with configuration spaces of punctured (and more generally non-compact) manifolds, we analyze the easier to understand $\ters{\FI}$-space $\manif^{\bul}$ defined as $S \mapsto \manif^{S}$ with the product topology.
 
\begin{prop} \label{high-conn-slope}
 Let $X$ be $u$-connected space with the homotopy type of a CW-complex with $u \geq 0$, and let $k \geq 0$ be a cohomological degree. Then 
\begin{align*}
 \tgen(\co^{k}(X^{\bul};\ab)) \leq \floor*{\frac{k}{u+1}}
\end{align*}
for every abelian group $\ab$.
\end{prop}
\begin{proof}
 Note that if $X$ is homotopy equivalent to $Y$, then the $\ters{\FI}$-space $X^{\bul}$ is homotopy equivalent to $Y^{\bul}$,\footnote{Unlike $\confix{}{X}$ and $\confix{}{Y}$.} so they have the same cohomology. Thus by \cite[page 58, second theorem]{fomenko-fuchs} we may assume that $X$ \textbf{is} a CW-complex with only one 0-cell and without any $j$-cells for $1\leq j \leq u$. Since every power of $X$ has a CW structure induced from that of $X$ with the transition maps being cellular, $X^{\bul}$ becomes a functor from $\FI^{\opp}$ to the category of CW-complexes. It may therefore be post-composed with the cellular chains functor $\cell{\star}$, yielding a chain complex $\cell{\star}(X^{\bul})$ of $\ters{\FI}$-modules, which evaluates to a chain complex of free abelian groups at any finite set. Arguing as in the fourth paragraph of the proof of \cite[Lemma 4.1]{cefn}, by the Eilenberg--Zilber theorem there is a quasi-isomorphism 
\begin{align*}
 \cell{\star}(X^{\bul}) \simeq \left(\cell{\star}(X)\right)^{\otimes \bul}
\end{align*}
 of chain complexes of $\ters{\FI}$-modules where the $(-)^{\otimes \bul}$ functor is as described in \cite[Remark 4.2.6]{cef}. Therefore the $\FI$-module $\co^{k}(X^{\bul};\ab)$ is isomorphic to the $k$-th cohomology group of the cochain complex
\begin{align*}
 \Hom_{\zz}\!\left(
 \left(\cell{\star}(X)\right)^{\otimes \bul},\, \ab
 \right) 
\end{align*}
of $\FI$-modules.
 Here, the $k$-th cochain $\FI$-module evaluated at a finite set of size $n$ is an abelian group of the form	
\begin{align*}
 \Hom_{\zz}\!\left(
 \left(\cell{\star}(X)\right)^{\otimes \bul},\, \ab
 \right)^{(k)}_{n} \cong \bigoplus _{k_{1} + \cdots + k_{n} = k} 
 \Hom_{\zz}\! \left(
 \cell{k_{1}}(X) \otimes \cdots \otimes \cell{k_{n}}(X),\, \ab
 \right) \, .
\end{align*}
Let us momentarily fix one of the summands above and write $J := \{1 \leq j \leq n : k_{j} = 0 \}$ for the set of zero indices. Because $\cell{j}(X) = 0$ for $1 \leq j \leq u$ we have 
\begin{align*}
 k = \sum_{j \in \{1,\dots,n\} - J} \!\!k_{j} \,\,\,\geq \sum_{j \in \{1,\dots,n\} - J} \!\!(u+1) = (n-|J|)(u+1) \, .
\end{align*}
Therefore if $n > \floor*{\frac{k}{u+1}}$, then $n(u+1) > k$ and $J$ has to be nonempty, that is, a zero index has to appear. As this is so for each summand and $\cell{0}(X) = \zz$, every summand lies in the image of a transition map of the $\FI$-module
\begin{align*}
 V^{k} := \Hom_{\zz}\!\left(
 \left(\cell{\star}(X)\right)^{\otimes \bul},\, \ab
 \right)^{(k)}
\end{align*}
induced by some injection $\{1,\dots,n-1\} \emb \{1,\dots,n\}$. This means that the $\FI$-module $V^{k}$ is generated in degrees $\leq \floor*{\frac{k}{u+1}}$. It is also $\cofi{0}$-acyclic by \cite[Definition 4.2.5 and Theorem 4.1.5]{cef}, hence is presented in finite degrees. Since \begin{align*}
 \co^{k}(X^{\bul};\ab) \cong \coker (V^{k-1} \rarr \ker (V^{k} \rarr V^{k+1})) \, ,
\end{align*}
we conclude by \cite[Propositions 3.1 and 3.3]{cmnr-range} that
$\weak(\co^{k}(X^{\bul};\ab)) \leq \floor*{\frac{k}{u+1}}$. 
Finally, note that $X^{\bul}$ extends to an $\sh$-space \cite[Remark 6.1.3]{cef} so 
\begin{align*}
 \tgen(\co^{k}(X^{\bul};\ab)) = \weak(\co^{k}(X^{\bul};\ab)) 
 \leq \floor*{\frac{k}{u+1}} \, .
\end{align*}
\end{proof}

We are now ready to prove the necessary input to Theorem \ref{reduce-PConf-to-punctured}.

\begin{thm} \label{non-compact-generation}
Let $\manif$ be a \textbf{non-compact} $u$-connected $d$-manifold with $u \geq 0$, $d \geq 2$, and let $k \geq 0$ be a cohomological degree. Write
\begin{align*}
 k = q_{k}(d-1) + r_{k},\, \quad 0 \leq r_{k} \leq d-2
\end{align*}
 via Euclidean division so that $q_{k} = \floor*{\frac{k}{d-1}}$, and set 
\begin{align*}
\weak_{k} :=
\begin{dcases*}
 \floor*{\frac{k}{u+1}} & \text{if $u+1 < d/2$,} \\
  2q_{k} + 1  & \text{if $d/2 \leq u+1 \leq r_{k}$,} \\
  2q_{k}  & \text{if $u+1 \geq \max\{d/2,r_{k}+1\}$.}
\end{dcases*}
\end{align*}
In case $d=2$, $k \geq 1$ and $\manif \neq \mathbb{S}^{2}-C$ for some closed subset $C \subseteq \mathbb{S}^{2}$, we reset $\weak_{k} := 2k-1$. Then for every abelian group $\ab$, the $\FI$-module $\co^{k}(\confix{}{\manif}; \ab))$ is generated in degrees $ \leq \weak_{k}$.
\end{thm}
\begin{proof}
The $\FI$-module $\co^{k}(\confix{}{\manif}; \ab)$ can be extended to an $\sh$-module by \cite[Proposition 6.4.2]{cef} and \cite[Section 3.1]{miller-wilson-secondary-stab}. Thus by \cite[Theorem 4.1.5]{cef}, there exists an $\FB$-module $W$ such that
\begin{align*}
 \co^{k}(\confix{}{\manif}; \ab) \cong \induce(W)\, ,
\end{align*}
which is $\cofi{0}$-acyclic by \cite[Lemma 2.3]{ce-homology}. Therefore 
\begin{align*}
 \tgen(\co^{k}(\confix{}{\manif}; \ab)) = \weak(\co^{k}(\confix{}{\manif}; \ab))
\end{align*}
by Corollary \ref{t0-vs-t1}. First we assume $u=0$ so that 
\begin{align*}
\weak_{k} = 
\begin{dcases*}
  k & \text{if $d \geq 3$,} \\
  2k-1  & \text{if $k \geq 1$, $d=2$ and $\manif \neq \mathbb{S}^{2} - C$,} \\
  2k & \text{otherwise.}
\end{dcases*}
\end{align*}
Setting $\HH_{j} := \co_{j}(\confix{}{\manif};\zz)$ for the $j$-th \textbf{homology} group (which is also an $\sh^{\opp} = \sh$-module as the $\sh$-action on $\confix{}{\manif}$ can be realized on the space level up to homotopy), by the universal coefficient theorem there is a short exact sequence 
\begin{align*}
 0 \rarr \Ext^{1}_{\zz}(\HH_{k-1}, \ab) 
 \rarr \co^{k}(\confix{}{\manif}; \ab)
 \rarr \Hom_{\zz}(\HH_{k}, \ab) \rarr 0
\end{align*}
of $\sh$-modules. Here $\tgen(\HH_{k}) \leq \weak_{k}$ by \cite[Corollary 2.6]{miller-wilson-PConf-hyper} and \cite[Corollary 3.36]{miller-wilson-secondary-stab}, hence the desired conclusion follows by \cite[Lemma 2.5]{miller-wilson-PConf-hyper}.

Next, we assume $u \geq 1$. Then $\manif$ is orientable and we may refer to \cite[Theorem 1]{totaro-config} for analyzing the Leray spectral sequence
\begin{align*}
 E_{2}^{p,q}(\ab) \imp \co^{p+q}(\confix{}{\manif};\ab)
\end{align*}
 of $\FI$-modules associated to the inclusion $\confix{}{\manif} \emb \manif^{\bul}$ of $\ters{\FI}$-spaces. As alluded to in \cite[proof of Theorem 6.2.1]{cef}, the second page $E_{2}^{\star,\star}(\zz)$ is generated as a bigraded $\FI$-algebra by  
\begin{align*}
 E_{2}^{\star,0}(\zz) \cong \co^{\star}(\manif^{\bul};\zz) 
 \quad \text{and} \quad
 E_{2}^{0,d-1}(\zz) \cong \co^{d-1}(\confix{}{\rr^{d}};\zz) \, .
\end{align*}
In a similar vein, the second page $E_{2}^{\star,\star}(\ab)$ is generated as a bigraded module over $E_{2}^{\star,\star}(\zz)$ by
\begin{align*}
 E_{2}^{\star,0}(\ab) \cong \co^{\star}(\manif^{\bul};\ab) 
 \quad \text{and} \quad
 E_{2}^{0,d-1}(\ab) \cong \co^{d-1}(\confix{}{\rr^{d}};\ab) \, .
\end{align*}
Thus if $d-1$ does not divide $b$ then $E_{2}^{a,b}(\ab) = 0$, and for every $a,s \geq 0$ the $\FI$-module $E_{2}^{a,s(d-1)}(\ab)$ receives a surjection from 
\begin{align*}
 &\bigoplus_{a_{1} + \cdots + a_{r} = a} 
 \left(
 E_{2}^{a_{1},0}(\zz) 
 \otimes \cdots \otimes E_{2}^{a_{r-1},0}(\zz)
 \otimes E_{2}^{a_{r},0}(\ab)
 \right)
 \otimes
 (E_{2}^{0,d-1}(\zz))^{\otimes s} 
\\
 \oplus
 &\bigoplus_{a_{1} + \cdots + a_{r} = a} 
 \left(
 E_{2}^{a_{1},0}(\zz) 
 \otimes \cdots \otimes
 E_{2}^{a_{r},0}(\zz)
 \right)
 \otimes
 (E_{2}^{0,d-1}(\zz))^{\otimes s-1} \otimes E_{2}^{0,d-1}(\ab) \, .
\end{align*}
Here $\tgen(E_{2}^{0,d-1}(\ab)) \leq 2$ and $\tgen(E_{2}^{a,0}(\ab)) \leq \floor*{\frac{a}{u+1}}$ by Proposition \ref{high-conn-slope}, noting that $\manif$ has the homotopy type of a CW-complex \cite[Corollary 5.2.4]{fri-pic-cellular}. Since \cite[Proposition 2.3.6]{cef} holds for $\FI$-modules presented in finite degrees as well, we have 
\begin{align*}
\tgen(E_{2}^{a,s(d-1)}(\ab)) 
 &\leq 
 \max\left\{
 \sum_{j} \floor*{\frac{a_{j}}{u+1}}
 : \sum_{j} a_{j} = a
 \right\} \,+\, 2s
 \\ 
 &\leq \floor*{\frac{a}{u+1}} + 2s \, .
\end{align*}
Now by \cite[Proposition 2.9, part (4)]{cmnr-range} and \cite[Proposition 4.1, part (1)]{cmnr-range}, we get
\begin{align*}
 \weak(\co^{k}(\confix{}{\manif}; \ab))
 &\leq \max\{\tgen(E_{2}^{a,b}(\ab)) : a+b = k,\, a,b \geq 0\} 
 \\
 &\leq \max\left\{
 \floor*{\frac{a}{u+1}} + 2s : a+s(d-1)=k,\, a,s \geq 0
 \right\}
 \\
 &\leq \max\left\{
 \floor*{\frac{k-s(d-1)}{u+1}} + 2s : 0 \leq s \leq q_{k} 
 \right\} \, .
\end{align*}
 Recalling that $k,d,u$ are fixed, let us write 
\begin{align*}
 f(s) &:= \frac{k-s(d-1)}{u+1} + 2s 
 = \frac{s(2u-d+3) + k}{u+1} \, ,
\end{align*}
so that
\begin{align*}
 \weak(\co^{k}(\confix{}{\manif}; \ab)) \leq \max\{
 \floor*{f(s)} : 0 \leq s \leq q_{k} 
 \} 
 = \floor*{\max\{f(s) : 0 \leq s \leq q_{k} \}} \, .
\end{align*}
We see that $f$ is strictly increasing if $2u+3 > d$ which is equivalent to $u+1 \geq d/2$. It is non-increasing if $2u+3 \leq d$, which is equivalent to $u+1 < d/2$. Thus
\begin{align*}
 \max \{
 f(s) : 0 \leq s \leq q_{k}
 \}
 &= 
\begin{cases}
 f(0) & \text{if $u+1 < d/2$},
 \\
 f(q_{k}) & \text{if $u+1 \geq d/2$,}
\end{cases}
\\
&=
\begin{dcases*}
 \frac{k}{u+1} & \text{if $u+1 < d/2$},
 \vspace{0.02in}
 \\
 \frac{r_{k}}{u+1} + 2q_{k} & \text{if $u+1 \geq d/2$.}
\end{dcases*}
\end{align*}
Here note that in the case $u + 1 \geq d/2$, we have 
\begin{align*}
 0\leq \frac{r_{k}}{u+1} \leq \frac{d-2}{u+1} \leq \frac{2d-4}{d} = 2 - \frac{4}{d} < 2
\end{align*}
and so 
\begin{align*}
 \floor*{\frac{r_{k}}{u+1}} = 
\begin{cases}
 1 & \text{if $u+1 \leq r_{k}$,} \\
 0 & \text{if $u \geq r_{k}$.}
\end{cases}
\end{align*}
Thus $\floor*{\max \{f(s) : 0 \leq s \leq q_{k}\}} = \weak_{k}$ and we are done.
\end{proof}

\begin{proof}[Proof of \textbf{\emph{Theorem \ref{main-config}}}]
For every nonempty finite subset $Q \subseteq \manif$, the $d$-manifold $\manif - Q$ is non-compact, and by \cite[Lemma 3.3 and the paragraph before Remark 2.2]{ggg-homotopy-fiber} $\manif-Q$ is also $u$-connected (this is where we use the $u \leq d-2$ assumption). Therefore we can apply Theorem \ref{reduce-PConf-to-punctured} to $\manif$ with the prescribed sequence $(\weak_{k} : k \geq 0)$ in Theorem \ref{non-compact-generation}, which is weakly increasing and satisfies $\weak_{k} \geq 1$ for $k \geq d-1$, so that by Theorem \ref{pass-to-cmnr} the triple 
\begin{align*}
 \mathlarger{
 (\co^{k}(\confix{}{\manif};\ab),\,2\weak_{k}-2,\,\weak_{k})
 }
\end{align*}
satisfies Hypothesis \ref{FI-hyp}. We conclude by applying Theorem \ref{main-ranges}.
\end{proof}

\begin{rem}[$u$-acyclic instead of $u$-connected]
Our proof shows that the conclusion of Theorem \ref{main-config} holds more generally for the class 
\begin{align*}
 \mathcal{C}_{u} := \left\{
 \manif : \weak\!\left(
 \co^{k}\!\big(
 (\manif - Q)^{\bul};\ab
 \big)
 \right) \leq \floor*{\frac{k}{u+1}} \text{
\begin{tabular}{l}
 for every $k \geq 0$ and\\
 nonempty finite subset $Q \subseteq \manif$
\end{tabular}
}\!\!\right\}
\end{align*}
of manifolds, and that $u$-connected manifolds belong to $\mathcal{C}_{u}$. It is likely that one can show $u$-acyclic manifolds belong to $\mathcal{C}_{u}$, by generalizing Proposition \ref{high-conn-slope} to $u$-acyclic spaces via a multiple version of the K\"unneth formula. The vanishing $\wt{\co}_{j}(\manif;\ab) = 0$ for $j \leq u$ with $\ab$-coefficients might already suffice for $\manif \in \mathcal{C}_{u}$.
\end{rem}

\begin{ex} \label{ex-sphere-deep}
 Fix $d \geq 2$. For every nonempty finite subset $Q \subseteq \mathbb{S}^{d}$, the non-compact $d$-manifold $\mathbb{S}^{d}-Q$ is $(d-2)$-connected; consequently by Theorem \ref{non-compact-generation} the $\FI$-module
\begin{align*}
 \co^{d-1}(\confix{}{\mathbb{S}^{d}-Q};\ab) \text{ is generated in degrees $\leq 2$} \, .
\end{align*}
Therefore by Theorem \ref{reduce-PConf-to-punctured}, the $\FI$-module $V := \co^{d-1}(\confix{}{\mathbb{S}^{d}};\qq)$ satisfies 
\begin{align*}
 \tgen(V) \leq 4 \,,\,\,\,
 \trel(V) \leq 5  \,,\,\,\,
 \local(V) \leq 2  \,,\,\,\,
 \weak(V) \leq 2  \,.\,\,\,
\end{align*}
Now assume $d$ is even. Then the dimension sequence of $V$ is
\begin{align*}
 \dim_{\qq} V_{n} = 
\begin{dcases*}
 0 & \text{if $n \leq 2$,}
 \\
 \frac{n(n-3)}{2} & \text{if $n \geq 3$,}
\end{dcases*}
\end{align*}
where the $n \leq 2$ case follows from the homotopy equivalence $\mathbb{S}^{d} \simeq \confix{2}{\mathbb{S}^{d}}$ via inserting antipodes and the $n \geq 3$ case was mentioned in Example \ref{ex-sphere}. Thus by \cite[Proposition 2.14]{cmnr-range}, we actually have 
\begin{align*}
 \local(V) = \weak(V) = 2  \,.\,\,\,
\end{align*}
We also see that $V_{3} = 0$, so the nonzero $\FI$-module $V$ cannot be generated in degrees $\leq 3$, hence $\tgen(V) = 4$ and Corollary \ref{t0-vs-t1} yields $\trel(V) = 5$. Due to dimension reasons the additive structure decomposition in Definition \ref{stable-ranges-defn} can only be satisfied with the $\qq$-vector spaces 
\begin{align*}
 \ab_{r} = 
\begin{cases}
 \qq & \text{if $r=2$,}
 \\
 0 & \text{otherwise.}
\end{cases}
\end{align*}
in the range $n \geq 3$, which is sharp. Finally, by Specht stability we know that there are constants $a,b,c_{1},c_{2} \in \nn$ (in characteristic zero Specht modules are simple, hence their multiplicities are non-negative) such that in the range $n \geq 4$,
\begin{align*}
 [V_{n}] = a[\specht{\qq}{n}] + b[\specht{\qq}{n-1,1}] 
 + c_{1}[\specht{\qq}{n-2,2}] + c_{2}[\specht{\qq}{n-2,1,1}]
\end{align*}
in the Grothendieck group of finite dimensional $\qq\sym{n}$-modules. Taking  dimensions on both sides and using the hook length formula \cite[Theorem 20.1]{james-sym-book}, we deduce $a = b = c_{2} = 0$ and $c_{1} = 1$. In other words, we have 
\begin{align*}
 V_{n} \cong 
\begin{cases}
 0 & \text{if $n < 4$,}
 \\
 \specht{\qq}{n-2,2} & \text{if $n \geq 4$,}
\end{cases}
\end{align*}
as a $\qq\sym{n}$-module. Thus the range $n \geq 4$ for Specht stability is also sharp. Let us also exhibit the character polynomial that computes $V$ in the range $n \geq 3$: it is denoted $q_{(2)}$ in \cite[I.2]{garsia-goupil} and hence by \cite[Corollary I.1]{garsia-goupil} it is equal to 
\begin{align*}
 \darr\!\left(\frac{1}{2}(\cyc_{1}-1)^{2} + \frac{1}{2}(2\cyc_{2}-1)\right) \,&=\,\, 
 \darr\!\left(\frac{\cyc_{1}^{2}}{2} - \cyc_{1} + \cyc_{2} \right)
 \\
 &= \frac{\cyc_{1}(\cyc_{1}-1)}{2} - \cyc_{1} + \cyc_{2} 
 = \frac{\cyc_{1}(\cyc_{1}-3)}{2} + \cyc_{2} \\
 &= \binom{\cyc_{1}-3}{2} + \cyc_{2} + 2(\cyc_{1}-3)\, ,
\end{align*}
where $\darr$ is the \emph{umbral operator} \cite[I.4]{garsia-goupil}.
\end{ex}
\subsection{Congruence subgroups}
In this section we prove Theorem \ref{main-cong}. We first recast a result of Djament in our notation.

\begin{thm}[{\cite{djament-congruence-stab}}] \label{djament-2k-result}
 Let $I$ be a proper ideal in a ring $R$. Then for every homological degree $k \geq 0$, we have 
\begin{birki}
 \item $\weak(\co_{k}(\GL_{\bul}(R,I);\zz)) \leq 2k$.
 \vspace{0.03in}
 \item If $I \neq I^{2}$, then $\weak(\co_{k}(\GL_{\bul}(R,I);\zz)) = 2k$.
\end{birki}
\end{thm}

 We note that the sharpness of part (2) in Theorem \ref{djament-2k-result} yields a generalization of \cite[Theorem D]{cmnr-range} in $\zz$-coefficients:

\begin{thm} \label{nonvan-spb}
 Let $I$ be a proper ideal in a ring $R$ such that $\sr(R) \leq s$ and $I \neq I^{2}$. Then writing $\pbc_{\bul}(R,I)$ for the $\FI$-simplicial complex of mod-$I$ split partial bases defined in \emph{\cite[Definition 7.5, Definition 7.8]{cmnr-range}}, for every $k \geq 1$ there exist integers $q_{k},n_{k}$ such that the following hold:
\begin{birki}
 \item $\max\!\left\{0,\,\,k - \floor*{\frac{s+1}{2}}\right\} \leq q_{k} \leq k-1$.
 \vspace{0.1cm}
 \item $2k \leq n_{k} \leq 2q_{k} + s + 1$.
 \vspace{0.1cm}
 \item  $\wt{\co}_{q_{k}}\!\left(
 \pbc_{n_{k}}(R,I);\zz
 \right) \neq 0$.
 \vspace{0.1cm}
 \item $\floor*{\frac{n_{k} - s - 2}{2}} < \min\!\left\{ a \geq 0:
 \wt{\co}_{a}\!\left(
 \pbc_{n_{k}}(R,I);\zz
 \right) \neq 0
 \right\} \leq q_{k}$.
\end{birki}
\end{thm}
\begin{proof}
By Theorem \ref{djament-2k-result}, \cite[Theorem 5.1, part (1)]{cmnr-range}, \cite[Theorem C]{cmnr-range},  and the equivariant homology spectral sequence
\begin{align*}
 E^{2}_{p,q} = \co_{p}\!\left(
 \GL_{\bul}(R,I);\, \wt{\co}_{q}(\pbc_{\bul}(R,I);\zz)
 \right) \imp \wt{\co}_{p+q}^{\GL_{\bul}(R,I)}(\pbc_{\bul}(R,I);\zz) \, ,
\end{align*}
we obtain
\begin{align*}
 2k = \weak(\co_{k}(\GL_{\bul}(R,I);\zz)) &\leq 
 \tbol_{k}\!\left(
 \chain_{\star}(\GL_{\bul}(R,I);\zz)
 \right) = \deg\!\left(
 \wt{\co}_{k-1}^{\GL_{\bul}(R,I)}(\pbc_{\bul}(R,I);\zz)
 \right)
 \\
 &\leq \max\!\left\{\deg\!\left(
 \wt{\co}_{q}(\pbc_{\bul}(R,I);\zz)
 \right) : 0 \leq q \leq k-1
 \right\}
\end{align*}
where each $\deg\!\left(
 \wt{\co}_{q}(\pbc_{\bul}(R,I);\zz)
 \right) \leq 2q + s + 1$ as noted in \cite[Section 7.2, proof of Theorem C]{cmnr-range}. Therefore there exists $0 \leq q_{k} \leq k-1$ such that setting 
\begin{align*}
 n_{k} := \deg\!\left(
 \wt{\co}_{q_{k}}(\pbc_{\bul}(R,I);\zz)
 \right) \, ,
\end{align*}
we have $2k \leq n_{k} \leq 2q_{k} + s + 1$ and $\wt{\co}_{q_{k}}(\pbc_{n_{k}}(R,I);\zz) \neq 0$. Finally, note that $\pbc_{n_{k}}(R,I)$ is $\floor*{\frac{n_{k} - s - 2}{2}}$-acyclic as explained in \cite[proof of Proposition 5.4]{cmnr-range}.
\end{proof}

\begin{cor}
 Let $R,I$ satisfy the hypotheses of Theorem \ref{nonvan-spb} with $s=1$. Then for every $k \geq 1$ we have
\begin{align*}
 \min\!\left\{ a \geq 0:
 \wt{\co}_{a}\!\left(
 \pbc_{2k}(R,I);\zz
 \right) \neq 0
 \right\} = k-1  \, .
\end{align*}
\end{cor}
\begin{proof}
 Part (1) of Theorem \ref{nonvan-spb} forces $q_{k} = k-1$, and part (2) of Theorem \ref{nonvan-spb} forces $n_{k} = 2k$. Noting that here $\floor*{\frac{n_{k} - s - 2}{2}} = k-2$, the equality follows from part (4) of Theorem \ref{nonvan-spb}.
\end{proof}

\begin{cor}
 Let $R,I$ satisfy the hypotheses of Theorem \ref{nonvan-spb} with $s=2$. Then for every $k \geq 1$, either we have
\begin{align*}
 \min\!\left\{ a \geq 0:
 \wt{\co}_{a}\!\left(
 \pbc_{2k}(R,I);\zz
 \right) \neq 0
 \right\} = k-1 \, ,
\end{align*}
or we have
\begin{align*}
 \min\!\left\{ a \geq 0:
 \wt{\co}_{a}\!\left(
 \pbc_{2k+1}(R,I);\zz
 \right) \neq 0
 \right\} = k-1 \, .
\end{align*}
\end{cor}
\begin{proof}
 Part (1) of Theorem \ref{nonvan-spb} forces $q_{k} = k-1$, and part (2) of Theorem \ref{nonvan-spb} forces $n_{k}  \in \{2k,\,2k+1\}$. Noting that here $\floor*{\frac{n_{k} - s - 2}{2}} = k-2$, the claim follows from part (4) of Theorem \ref{nonvan-spb}.
\end{proof}

\begin{proof}[Proof of \textbf{\emph{Theorem \ref{djament-2k-result}}}]
Noting that the degree of a weakly polynomial functor does not depend on the automorphism groups of the domain category in the setup of \cite{djament-congruence-stab}, by taking $e=1$ in \cite[Corollaire 2.42]{djament-congruence-stab}, $\co_{k}(\GL_{\bul}(R,I);\zz)$ is weakly polynomial of degree $\leq 2k$ in the sense of \cite[D\'efinition 2.22]{djament-vespa-weakly}. Now (1) follows from Proposition \ref{relate-to-faible}. If $I \neq I^{2}$, \cite[Corollaire 2.42]{djament-congruence-stab} also shows that $\co_{k}(\GL_{\bul}(R,I);\zz)$ is \textbf{not} weakly polynomial of degree $\leq 2k-1$, hence (2) follows.  
\end{proof}

We are now ready to obtain all the promised stable ranges for the homology groups of the congruence $\FI$-group $\GL_{\bul}(R,I)$.

\begin{thm} \label{congruence-nitpicks}
 Let $I$ be a proper ideal in a ring $R$ with $\sr(R) \leq s$. Then for every homological degree $k \geq 0$ and abelian group $\ab$, we have
\begin{birki}
 \item $\weak(\co_{k}(\GL_{\bul}(R,I);\ab)) \leq 2k$.
 \vspace{0.1cm}
 \item $\tgen(\co_{k}(\GL_{\bul}(R,I);\ab)) \leq 
\begin{cases}
  0 & \text{if $k=0$,}
  \\
  s+1 & \text{if $k = 1$,}
  \\
  2s+5 & \text{if $k=2$,}
  \\
  4k+2s-2 & \text{if $k \geq 3$.}
\end{cases}$
\vspace{0.1cm}
 \item $\trel(\co_{k}(\GL_{\bul}(R,I);\ab)) \leq 
\begin{cases}
  -1 & \text{if $k=0$,}
  \\
  s+3 & \text{if $k = 1$,}
  \\
  2s+6 & \text{if $k=2$,}  
  \\
  4k+2s-1 & \text{if $k \geq 3$.}
\end{cases}
$
\vspace{0.1cm}
\item $
  h^{j}(\co_{k}(\GL_{\bul}(R,I);\ab)) \leq 
\begin{cases}
 -1 & \text{if $k=0$,} \\
 2s+3 & \text{if $k=1$ and $j=0$,} \\
 4k+2s & \text{if $k \geq 2$ and $j=0$,} \\
 4k+2s-2 & \text{if $k \geq 1$ and $j=1$,} \\ 
 \max\{-1,\,4k-2j+2\} & \text{if $k \geq 1$ and $j \geq 2$.}
\end{cases}
$
\vspace{0.1cm}
\item $\reg(\co_{k}(\GL_{\bul}(R,I);\ab)) \leq 
\begin{cases}
  -2 & \text{if $k=0$,}
  \\
  2s+3 & \text{if $k = 1$,}
  \\
  4k+2s & \text{if $k \geq 2$.}
\end{cases}
$
\end{birki}
\end{thm}

\begin{proof}
For $k=0$, the $\FI$-module $\co_{0}(\GL_{\bul}(R,I);\ab)$ is constant at $\ab$, hence is $\cofi{0}$-acyclic and is generated in degrees $ \leq 0$, so the claims follow from Theorem \ref{characterize-H0-acyclic}. We shall now establish the bounds for a general $k \geq 1$, before improving them for the lower degrees $k=1,2$. Noting that $\sr(R) \leq s$ is equivalent to $R$ satisfying Bass's condition SR$_{s+1}$, \cite[Proposition 5.4]{cmnr-range} says that
\begin{align*}
 \tbol_{k}(\chain_{\star}(\GL_{\bul}(R,I); \ab)) \leq 2k+s-1 \, .
\end{align*}
Thus by Theorem \ref{hyper-to-usual} the $\FI$-module $\co_{k}(\GL_{\bul}(R,I);\ab)$ satisfies the following: 
\begin{itemize}
 \item$\weak(\co_{k}(\GL_{\bul}(R,I);\ab)) \leq 2k+s-1$. In fact this bound can be improved to
\begin{align*}
 \weak(\co_{k}(\GL_{\bul}(R,I);\ab)) \leq 2k
\end{align*}
as we now explain. In case $\ab = \zz$, this is part (1) of Theorem \ref{djament-2k-result}. In general, let us write $\HH_{k} := \co_{k}(\GL_{\bul}(R,I);\zz)$ so that we have established $\weak(\HH_{k}) \leq 2k$ for every $k \geq 0$, and by the universal coefficient theorem and its naturality we have a short exact sequence 
\begin{align*}
 0 \rarr \HH_{k} \otimes_{\zz} \ab \rarr
 \co_{k}(\GL_{\bul}(R,I);\ab) \rarr \Tor_{1}^{\zz}(\HH_{k-1},\,\ab)
 \rarr 0
\end{align*}
of $\FI$-modules. Consequently we have 
\begin{align*}
 \weak(\co_{k}(\GL_{\bul}(R,I);\ab)) = \max\!
 \left\{
 \weak(\HH_{k} \otimes_{\zz} \ab),\,
 \weak(\Tor_{1}^{\zz}(\HH_{k-1},\,\ab))
 \right\}
\end{align*}
by \cite[Proposition 2.9, part (5)]{cmnr-range}. Let us pick free abelian groups $\F,\F' \neq 0$ such that there is an exact sequence  $0 \rarr \F' \rarr \F \rarr \ab \rarr 0$. From here we get 
\begin{align*}
 \HH_{k} \otimes_{\zz} \ab &\cong \coker(\HH_{k} \otimes_{\zz} \F' \rarr \HH_{k} \otimes_{\zz} \F)\,, \text{ so}
 \\
 \weak(\HH_{k} \otimes_{\zz} \ab) &\leq \weak(\HH_{k} \otimes_{\zz} \F) = \weak(\HH_{k}) \leq 2k
\end{align*}
by \cite[Proposition 3.3, part (2)]{cmnr-range} and
\begin{align*}
 \Tor_{1}^{\zz}(\HH_{k-1},\,\ab) &\cong \ker(\HH_{k-1} \otimes_{\zz} \F' \rarr \HH_{k-1} \otimes_{\zz} \F)\,, \text{ so}
 \\
 \weak(\Tor_{1}^{\zz}(\HH_{k-1},\,\ab)) &\leq
 \weak(\HH_{k-1} \otimes_{\zz} \F') = \weak(\HH_{k-1}) \leq 
\begin{cases}
 2(k-1) & \text{if $k \geq 1$,}
 \\
 -1 & \text{if $k = 0$,}
\end{cases}
\end{align*}
by \cite[Proposition 3.3, part (1)]{cmnr-range}.
 \vspace{0.1cm}
 \item $\tgen(\co_{k}(\GL_{\bul}(R,I);\ab)) \leq 4k+2s-2$.
 \vspace{0.1cm}
 \item $
  h^{j}(\co_{k}(\GL_{\bul}(R,I);\ab)) \leq 
\begin{cases}
 4k+2s & \text{if $j=0$,} \\
 4k+2s-2 & \text{if $j=1$,} \\ 
 \max\{-1,\,4k-2j+2\} & \text{if $j \geq 2$.}
\end{cases}
$
\vspace{0.1cm}
 \item $\reg(\co_{k}(\GL_{\bul}(R,I);\ab)) \leq 4k+2s$.
\end{itemize}

Shortly writing $\chain_{\star} := \chain_{\star}(\GL_{\bul}(R,I); \ab)$,
as alluded to in \cite[proof of Lemma 5.3]{cmnr-range}, by the hyperhomology spectral sequence $E^{2}_{p,q} = \cofi{p}(\co_{q}(\chain_{\star})) \imp \cofibol{p+q}(\chain_{\star})$ we have
\begin{align*}
 \tgen(\co_{k}(\chain_{\star})) &\leq
 \max\big( 
 \{\tbol_{k}(\chain_{\star})\} \cup \{t_{p}(\co_{q}(\chain_{\star})) : p+q = k+1,\,0 < q < k\}
 \big) \, ,
\\
 &\leq \max\left(
\begin{array}{l}
  \{2k+s-1\} \,\,\cup \\ 
  \{p+4q+2s : p+q = k+1,\,0 < q < k\}
\end{array}
 \right)
 \\
 &\leq \max\left(
\begin{array}{l}
  \{2k+s-1\} \,\,\cup \\ 
  \{k+3q+2s+1 : 0 < q < k\}
\end{array}
 \right)
 \\ 
 &\leq 
\begin{cases}
 s+1 & \text{if $k=1$,}
 \\
 4k+2s-2 & \text{if $k \geq 2$,}
\end{cases}
\end{align*}
and
\begin{align*} 
 \trel(\co_{k}(\chain_{\star})) &\leq
 \max\big( 
 \{\tbol_{k+1}(\chain_{\star})\} \cup \{t_{p}(\co_{q}(\chain_{\star})) : p+q = k+2,\,0 < q < k\}
 \big) 
 \\
 &\leq \max\left(
\begin{array}{l}
  \{2(k+1)+s-1\} \,\,\cup \\ 
  \{p+4q+2s : p+q = k+2,\,0 < q<k\}
\end{array}
 \right)
 \\
 &\leq \max\left(
\begin{array}{l}
  \{2k+s+1\} \,\,\cup \\ 
  \{k+3q+2s+2 : 0 < q < k\}
\end{array}
 \right)
 \\
 &\leq 
\begin{cases}
 s+3 & \text{if $k=1$,}
 \\
 4k+2s-1 & \text{if $k \geq 2$.}
\end{cases}
\end{align*}

For $k=1$, by Theorem \ref{improved-regularity} and \cite[Proposition 3.1, part (2)]{cmnr-range}, we get
\begin{align*}
\reg(\co_{1}(\chain_{\star})) \leq 2s+3 \quad \text{and} \quad
 \local(\co_{1}(\chain_{\star})) \leq 2s+3 \, .
\end{align*}


For $k=2$, we have
\begin{align*}
 \tgen(\co_{2}(\chain_{\star})) &\leq
 \max\!\left\{
 \tbol_{2}(\chain_{\star}),\, 
 t_{2}(\co_{1}(\chain_{\star}))
 \right\} \leq 2s+5 \, ,
\\
 \trel(\co_{2}(\chain_{\star})) &\leq
 \max\!\left\{ 
 \tbol_{3}(\chain_{\star}),\, t_{3}(\co_{1}(\chain_{\star})) 
 \right\} \leq 2s+6 \, .
\end{align*}
\end{proof}

\begin{proof}[Proof of \textbf{\emph{Theorem \ref{main-cong}}}]
Combine Theorem \ref{congruence-nitpicks}, Corollary \ref{cor-inductive}, and Theorem \ref{main-ranges}, noting that by Theorem \ref{pass-to-cmnr}, the triple
\begin{align*}
\begin{cases}
  \left(\co_{0}(\GL_{\bul}(R,I);\ab),\, -1,\, 0 \right)
  & \text{if $k = 0$,}
  \\
  \left(\co_{1}(\GL_{\bul}(R,I);\ab),\, 2s+3,\, 2 \right)
  & \text{if $k = 1$,}
  \\
  \left(\co_{k}(\GL_{\bul}(R,I);\ab),\, 4k+2s,\, 2k \right)
  & \text{if $k \geq 2$,}
\end{cases}
\end{align*}
satisfies Hypothesis \ref{FI-hyp}.
\end{proof}

\bibliographystyle{hamsalpha}
\bibliography{stable-boy}

\end{document}